\documentclass[a4paper,10pt]{article}
 
\usepackage{ae,lmodern} 
\usepackage[utf8]{inputenc}
\usepackage[T1]{fontenc}
\usepackage[english]{babel}
\usepackage{geometry}
\usepackage{url,xspace}
\usepackage{amsmath,amssymb,stmaryrd,spalign,bbm,mathtools,mathrsfs,dsfont,thmtools}
\usepackage[thmmarks,amsmath,framed,thref,hyperref]{ntheorem}
\usepackage{comment}
\usepackage[colorlinks=true,citecolor=teal,linkcolor=blue]{hyperref}
\usepackage{array}
\usepackage{enumitem}
\usepackage{changepage}
\usepackage{euscript}
\usepackage{cite}
\usepackage{graphicx}
\usepackage{pgfplots}
\usepackage[toc,page]{appendix}
\usepackage{etoolbox}

\counterwithin{equation}{section}

\setlist[enumerate,1]{label={(\roman*)}}

\pgfplotsset{width=10cm,compat=1.9}

\newcommand{\zerodisplayskips}{%
  \setlength{\abovedisplayskip}{5pt}%
  \setlength{\belowdisplayskip}{5pt}%
  \setlength{\abovedisplayshortskip}{5pt}%
  \setlength{\belowdisplayshortskip}{5pt}}
\appto{\normalsize}{\zerodisplayskips}
\appto{\small}{\zerodisplayskips}
\appto{\footnotesize}{\zerodisplayskips}

\newtheorem{theorem}{Theorem}[section]
    \newtheorem{corollary}[theorem]{Corollary}
    \newtheorem{lemma}[theorem]{Lemma}
    \newtheorem{proposition}[theorem]{Proposition}

    \newtheorem{assumption}[theorem]{Assumption}
    
    \theorembodyfont{\upshape}
    \newtheorem{definition}[theorem]{Definition}
    \newtheorem{notations}[theorem]{Notations}
    \newtheorem{remark}[theorem]{Remark}
    
\counterwithin{equation}{section}

\setlist[enumerate,1]{label={(\roman*)}}

\theoremstyle{nonumberplain}
\theoremheaderfont{\itshape}
\theorembodyfont{\normalfont}
\theoremseparator{.\,—}
\theoremsymbol{}
\newtheorem{proof-wo}{Proof}
\theoremsymbol{\ensuremath{\blacksquare}}
\newtheorem{proof}{Proof}

\newcommand{\eus}{\EuScript}

\newcommand{\llb}{\llbracket}
\newcommand{\rrb}{\rrbracket}

\let\div\undefined
\DeclareMathOperator{\div}{div\,}
\DeclareMathOperator{\curl}{curl\,}

\DeclareMathOperator{\supp}{supp\,}

\newcommand{\iprod}{\mathbin{\lrcorner}}

\title{Hodge decompositions and maximal regularities for Hodge Laplacians in homogeneous function spaces on the half-space }
\author{Anatole \textsc{Gaudin}\thanks{{Aix-Marseille Université, CNRS, Centrale Marseille, I2M, Marseille, France} - \textbf{email:} anatole.gaudin@univ-amu.fr}}

\begin{document}

\maketitle
\begin{abstract} In this article, the Hodge decomposition for any degree of differential forms is investigated on the whole space $\mathbb{R}^n$ and the half-space $\mathbb{R}^n_+$  on different scale of function spaces namely homogeneous and inhomogeneous Besov and Sobolev space, $\dot{\mathrm{H}}^{s,p}$, $\dot{\mathrm{B}}^{s}_{p,q}$, ${\mathrm{H}}^{s,p}$ and ${\mathrm{B}}^{s}_{p,q}$, for all $p\in(1,+\infty)$ , $s\in(-1+\frac{1}{p},\frac{1}{p})$. The bounded holomorphic functional calculus, and other functional analytic properties, of Hodge Laplacians is also investigated in the half-space, and yields similar results for Hodge-Stokes and other related operators via the proven Hodge decomposition. As consequences, the homogeneous operator and interpolation theory revisited by Danchin, Hieber, Mucha and Tolksdorf is applied to homogeneous function spaces subject to boundary conditions and leads to various maximal regularity results with global-in-time estimates that could be of use in fluid dynamics. Moreover, the bond between the Hodge Laplacian and the Hodge decomposition will even enable us to state the Hodge decomposition for higher order Sobolev and Besov spaces with additional compatibility conditions, for regularity index $s\in(-1+\frac{1}{p},2+\frac{1}{p})$. In order to make sense of all those properties in desired function spaces, we also give appropriate meaning of partial traces on the boundary in the appendix.

La raison d'être of this paper lies in the fact that the chosen realization of homogeneous function spaces is suitable for non-linear and boundary value problems, but requires a careful approach to reprove results that are already morally known.
\end{abstract}

\addtocontents{toc}{\protect\thispagestyle{empty}}
\tableofcontents

\renewcommand{\arraystretch}{1.5}



\section{Introduction}

\subsection{Motivations and interests}

\subsubsection{One Laplacian to rule (almost) them all: the differential form formalism and the Hodge decomposition}

The study of incompressible fluid dynamics, and in particular the treatment of Navier-Stokes equations, relies mostly on the Helmholtz decomposition of vector fields in appropriate function spaces. The Helmholtz decomposition of vector field $u\,:\,\Omega \longrightarrow \mathbb{C}^n$, is given by a vector field $v\,:\,\Omega \longrightarrow \mathbb{C}^n$ and  and a function $q\,:\,\Omega \longrightarrow \mathbb{C}$, such that
\begin{align*}
    u = v+\nabla q \,\text{ and }\, \div v = 0 \text{ ( with possibly } v\cdot\nu_{|_{\partial\Omega}} = 0 \text{). }
\end{align*}
This point is central since incompressibility condition for the velocity of a fluid $u$ is carried over by the condition $\div u=0$.

In the interest of the Navier-Stokes and related equations, one wants the above decomposition to hold topologically in an appropriate normed vector space of functions\footnote{From here the divergence will be understood in the distributional sense.} with uniqueness (up to a constant for $q$). It is indeed true in $\mathrm{L}^2(\Omega,\mathbb{C}^n)$, since $\mathbb{P}$, the usual Helmholtz-Leray projector on divergence free vector fields with null tangential trace at the boundary, \textit{i.e.} such that
\begin{align*}
    \mathbb{P}\,:\, \mathrm{L}^2(\Omega,\mathbb{C}^n) \longrightarrow \mathrm{L}^2_{\sigma}(\Omega)=\{\, u\in \mathrm{L}^2(\Omega,\mathbb{C}^n)\,{|}\, \div u = 0\text{, } u\cdot\nu_{|_{\partial\Omega}} =0\,\}\text{, }
\end{align*}
is well defined, linear, bounded and unique by construction of the orthogonal projector on a closed subspace of an Hilbert space, here $\mathrm{L}^2_{\sigma}(\Omega)\subset\mathrm{L}^2(\Omega,\mathbb{C}^n)$. It gives the classical orthogonal and topological Helmholtz decomposition, see \cite[Chapter~2,~Section~2.5]{SohrBook2001},
\begin{align*}
    \mathrm{L}^2(\Omega,\mathbb{C}^n)= \mathrm{L}^2_{\sigma}(\Omega)\overset{\perp}{\oplus}\overline{\nabla {\mathrm{H}}^{1,2}(\Omega,\mathbb{C})}\text{, }
\end{align*}
for any (bounded) Lipschitz domain $\Omega$, see \cite[Lemma~2.5.3]{SohrBook2001}. Here $\mathrm{H}^{1,2}(\Omega,\mathbb{C})$ is the standard $\mathrm{L}^2$-Sobolev space of order 1 on $\Omega$.

The $\mathrm{L}^2$-theory for the Helmholtz decomposition on a domain $\Omega$ relies mostly on pure Hilbertian operator theory. However, the question about the $\mathrm{L}^p$-theory, $p\neq 2$, \textit{i.e.} to know if
\begin{align}\label{HodgeDecompLpIntro}
    \mathrm{L}^p(\Omega,\mathbb{C}^n)= \mathrm{L}^p_{\sigma}(\Omega){\oplus}\overline{\nabla {\mathrm{H}}^{1,p}(\Omega,\mathbb{C})}\text{, }
\end{align}
(or even the Sobolev or Besov counterpart) is actually a harder question, which falls generally in the field of harmonic analysis. The underlying range of Lebesgue and Sobolev exponents for which such decomposition holds will generally depends on the regularity of the boundary and the geometry of the domain $\Omega$.

The $\mathrm{L}^p$ setting has been widely studied, we mention the work of Fabes, Mendez and Mitrea, \cite[Theorem~12.2]{FabesMendezMitrea1998}, where the result has been proven for bounded Lipschitz domains: \eqref{HodgeDecompLpIntro} holds whenever $p\in (3/2-\varepsilon,3+\varepsilon)$. The work of Sohr and Simader \cite[Theorem~1.4]{SimaderSohr1992} yields \eqref{HodgeDecompLpIntro}  for $\mathrm{C}^1$ bounded and exterior domains, allowing $p\in(1,+\infty)$. 
For general unbounded domains, when $p\neq 2$, the decomposition \eqref{HodgeDecompLpIntro} may fail: see the counterexample by Bogovski\v{i} \cite[Section~2]{Bogovskii1986}. Tolksdorf has shown in his PhD dissertation \cite[Theorem~5.1.10]{TolksdorfPhDThesis2017} that \eqref{HodgeDecompLpIntro} is true for all $p\in(\tfrac{2n}{(2n+1)}-\varepsilon,\tfrac{2n}{(2n-1)}+\varepsilon)$, provided $\Omega$ is a special Lipschitz domain, $\varepsilon>0$ depending on $\Omega$. We also mention the works of Farwig, Kozono and Sohr where the decomposition is investigated in a more exotic setting in \cite{FarwigKozonoSohr2005,FarwigKozonoSohr2007} for general uniformly $\mathrm{C}^1$ unbounded domains.

Our interest here is the case of the half-space $\mathbb{R}^n_+$, where the Helmholtz is mainly known to be true on $\mathrm{L}^p(\mathbb{R}^n_+,\mathbb{C}^n)$ for all $p\in(1,+\infty)$, see \cite[Remark~III.1.2]{bookGaldi2011}: we aim to generalize this result to the scale of inhomogeneous, and homogeneous Sobolev and Besov spaces on the half-space. To be more precise, we want to investigate decompositions of the type
\begin{align}\label{eq:HelmholtzDecompHspIntro}
    \dot{\mathrm{H}}^{s,p}(\mathbb{R}^n_+,\mathbb{C}^n) = \dot{\mathrm{H}}^{s,p}_{\sigma}(\mathbb{R}^n_+,\mathbb{C}^n) \oplus \overline{\nabla \dot{\mathrm{H}}^{s+1,p}(\mathbb{R}^n_+,\mathbb{C}^n)}\text{, } 
\end{align}
and similarly for Besov spaces, and their inhomogeneous counterparts, provided $s\in\mathbb{R}$, $p\in(1,+\infty)$.

In the scale of inhomogeneous and homogeneous Besov and Sobolev spaces on bounded and exterior $\mathrm{C}^{2,1}$ domains the Helmholtz decomposition was shown by Fujiwara and Yamazaki \cite[Theorem~3.1]{FujiwaraYamazaki2007}: the Helmholtz decomposition holds on $\mathrm{H}^{s,p}(\Omega,\mathbb{C}^n)$ and $\mathrm{B}^{s}_{p,q}(\Omega,\mathbb{C}^n)$, $p\in (1,+\infty)$, $s\in(-1+1/p,1/p)$, $q\in[1,+\infty]$, even allowing $p=1,+\infty$ in case of Besov spaces. We also mention the work of Monniaux and Mitrea \cite[Proposition~2.16]{MitreaMonniaux2008}  on bounded Lipschitz domains where the results is true for (inhomogeneous) Sobolev spaces that lies near the family $(\mathrm{H}^{s,2})_{|s|<1/2}$.

It has been notified in several works, e.g. see \cite[Introduction]{GeissertHeckTrunk2013}, \cite[Section~4]{MonniauxShen2018}, that the following Laplace operator acting on vector fields,
\begin{align}\label{eq:HodgeLapR3}
    -\Delta_{\mathcal{H}}:=-\Delta u = \curl \curl u - \nabla \div u \text{, and }\, [ u\cdot\nu_{|_{\partial\Omega}} = 0 \text{, } \nu\times\curl u_{|_{\partial\Omega}} = 0]
\end{align}
called the\footnote{In  fact, this is \underline{\textit{a}} Hodge Laplacian, the one with tangential boundary conditions, we do not make the distinction here for introductory purposes.} Hodge Laplacian, has a strong bond with, and respects, the Helmholtz decomposition in the sense that for all $u$ in the domain of above Laplacian, $\mathbb{P}u$ also lies in, and we have
\begin{align*}
    -\mathbb{P}\Delta u = \curl \curl u =  -\Delta \mathbb{P} u\text{, and }\, [\mathbb{P} u\cdot\nu_{|_{\partial\Omega}} = 0 \text{, } \nu\times\curl \mathbb{P} u_{|_{\partial\Omega}} = 0]\text{. }
\end{align*}

Therefore, since the Hodge Laplacian and the Helmholtz-Leray projector seem to copy the corresponding behavior of the whole space, it seems reasonable to infer that
\begin{align}\label{eq:HodgeHelmotzLerayProj1}
    \mathbb{P} = \mathrm{I} +\nabla\underline{\div}(-\Delta_{\mathcal{H}})^{-1}\text{, }
\end{align}
where $\underline{\div}$ drives a boundary condition $\nu \cdot u_{|_{\partial\Omega}} =0$. 

But, the above use of $\curl$ operators restricts us to the three dimensional case. We can avoid such trouble, by the mean of the differential forms formalism, so that \eqref{eq:HodgeLapR3} becomes
\begin{align}\label{eq:HodgeLapdefIntro}
    -\Delta_{\mathcal{H}}u:=-\Delta u = \mathrm{d}^\ast \mathrm{d} + \mathrm{d} \mathrm{d}^\ast u = (\mathrm{d}+\mathrm{d}^\ast)^2 u \text{, and }\, [ \nu\iprod u_{|_{\partial\Omega}} = 0 \text{, } \nu\iprod \mathrm{d} u_{|_{\partial\Omega}} = 0]
\end{align}
where $\mathrm{d}\,:\,\Lambda^{k}\longrightarrow\Lambda^{k+1}$ is the exterior derivative, defined on the complexified exterior algebra of $\mathbb{R}^n$, $\Lambda = \Lambda^0 \oplus \Lambda^1 \oplus \ldots \oplus \Lambda^n$, and satisfies $\mathrm{d}^2=0$. The operator $\mathrm{d}^\ast:\,\Lambda^{k}\longrightarrow\Lambda^{k-1}$ is the formal dual operator of $\mathrm{d}$, satisfying also $(\mathrm{d}^\ast)^2=0$ so that on $\mathbb{R}^3$, we can make the identifications
\begin{align*}
    \mathrm{d}_{|_{\Lambda^{1}}} = \curl \text{, } &\mathrm{d}_{|_{\Lambda^{0}}} = \nabla\text{, }\\
    \mathrm{d}^{\ast}_{|_{\Lambda^{2}}} = \underline{\curl} \text{, } &\mathrm{d}^\ast_{|_{\Lambda^{1}}} = -\underline{\div}\text{, }\\
    \nu\iprod ()_{|_{\Lambda^{1}}} = \nu \cdot () \text{, } & \nu\iprod ()_{|_{\Lambda^{2}}} = \nu\times () \text{. }
\end{align*}
The $\underline{\curl}$ operator drives a boundary condition $\nu\times u_{|_{\partial\Omega}}=0$.

Notice this definition still makes sense for differential forms of any degree, in arbitrary dimension. One would check that \eqref{eq:HodgeLapdefIntro} reduce to the Neumann Laplacian in the case of $0$-forms identified with scalar-valued functions.

Going back to the case of vector fields, instead of \eqref{eq:HodgeHelmotzLerayProj1}, the above formalism and the fact that $\mathrm{d}$ and $\mathrm{d}^\ast$ are nilpotent, and then commutes (at least formally) with $\Delta_{\mathcal{H}}$, we may infer the next formula, similar to the one mentioned in \cite[Section~5]{AuscherCouhlonDuongHofmann2004}:
\begin{align}\label{eq:HodgeHelmotzLerayProj2}
    \mathbb{P} = \mathrm{I} -\mathrm{d}\mathrm{d}^\ast(-\Delta_{\mathcal{H}})^{-1} = \mathrm{I} -\mathrm{d} (-\Delta_{\mathcal{H}})^{-1/2}\mathrm{d}^\ast(-\Delta_{\mathcal{H}})^{-1/2} \text{. }
\end{align}
Under the use of the differential forms formalism, the desired Helmholtz decomposition \eqref{eq:HelmholtzDecompHspIntro} becomes, for $k\in\llb 0,n\rrb$ different degrees of differential forms,
\begin{align}\label{eq:HodgeDecompHspIntro}
    \dot{\mathrm{H}}^{s,p}(\mathbb{R}^n_+,\Lambda^{k}) = \dot{\mathrm{H}}^{s,p}_{\mathfrak{t},\sigma}(\mathbb{R}^n_+,\Lambda^{k}) \oplus \overline{\mathrm{d} \dot{\mathrm{H}}^{s+1,p}(\mathbb{R}^n_+,\Lambda^{k-1})}
\end{align}
which is called the \textbf{Hodge decomposition} instead of the Helmholtz decomposition. Here, the space $\dot{\mathrm{H}}^{s,p}_{\mathfrak{t},\sigma}(\mathbb{R}^n_+,\Lambda^{k})$ stands for $k$-differential forms $u$ whose coefficients lies in $\dot{\mathrm{H}}^{s,p}(\mathbb{R}^n_+,\mathbb{C})$, and such that $\mathrm{d}^{\ast} u =0$, and $\nu\iprod u_{|_{\partial \mathbb{R}^n_+}} =0$.

The Hodge decomposition for differential forms is treated by Schwartz \cite[Theorem~2.4.2, Theorem~2.4.14]{Schwartz1995} on smooth compact Riemannian manifold $M$ with smooth boundary where the decomposition is stated on $\mathrm{H}^{k,p}(M)$, $k\in\mathbb{N}$, $p\in(1,+\infty)$. For the case of $\Omega$ a bounded Lipschitz domains of $\mathbb{R}^n$, we refer to the work of Monniaux and M${}^{\text{c}}$Intosh \cite[Theorem~4.3,~Theorem~7.1]{McintoshMonniaux2018} where the Hodge decomposition is proved to be true on $\mathrm{L}^p(\Omega,\Lambda)$ for all $p\in(\tfrac{2n}{(2n+1)}-\varepsilon,\tfrac{2n}{(2n-1)}+\varepsilon)$ where $\varepsilon>0$ depends on $\Omega$. The bounded holomorphic functional calculus of the Hodge Laplacian is also proved for the same range of indices. One may also consult the work of Mitrea and Monniaux, and Hofmann, Mitrea and Monniaux, \cite{MitreaMonniaux2009-2,HofmannMitreaMonniaux2011}, for the treatment of the Hodge Laplacian on bounded Lipschitz domains of compact Riemannian manifolds, where functional analytic properties like analyticity of the generated semigroup, or boundedness of associated Riesz transforms are investigated.

One may wonder about the superficiality of proving an identity like \eqref{eq:HodgeDecompHspIntro} for general differential forms, instead of vector fields (differential forms of degree $1,n-1$) only. In fact, the differential forms formalism has shown its efficiency allowing to treat some partial differential equations initially restricted to the three dimensional setting in arbitrary dimension. See for instance \cite{Monniaux2021,Denis2022}, where the magnetohydrodynamical (MHD) system is treated, so that either the triplet $\Lambda^{1},\Lambda^{2},\Lambda^{3}$ or the triplet $\Lambda^{n-3},\Lambda^{n-2},\Lambda^{n-1}$ are involved. Indeed, the magnetic field is in fact not an effective vector field but a $2$-form, identified, when $n=3$, with a vector field. We also mention that reformulation using differential forms for this kind of systems allows to look at vorticity-like formulation of the Navier-Stokes (and related) equations, it is also purely intrinsic so that one can perform a similar treatment on manifolds.

To reach our goal, the idea will be to prove that the formula \eqref{eq:HodgeHelmotzLerayProj2} holds on $\mathrm{L}^2(\mathbb{R}^n_+,\Lambda)$, yielding an operator for which we can also prove its boundedness on Sobolev and Besov spaces, so that we are able to obtain

\begin{theorem}[{see Theorem \ref{thm:HodgeDecompRn+} \& Corollary \ref{cor:extendedHodgeDecompComptabilityConditions}}]\label{thm:HodgeDecompRn+Intro}Let $p\in(1,+\infty)$, $s\in(-1+1/p,1/p)$, and let $k\in\llb 0,n\rrb$. It holds that
\begin{enumerate}
    \item The (generalized) Helmholtz-Leray projector $\mathbb{P}\,:\, \dot{\mathrm{H}}^{s,p}(\mathbb{R}^n_+,\Lambda^k)\longrightarrow \dot{\mathrm{H}}^{s,p}_{\mathfrak{t},\sigma}(\mathbb{R}^n_+,\Lambda^k)$ is well defined and bounded. Moreover the following identity is true
    \begin{align*}
        \mathbb{P} = \mathrm{I} - \mathrm{d}(-\Delta_{\mathcal{H}})^{-\frac{1}{2}}\mathrm{d}^\ast(-\Delta_{\mathcal{H}})^{-\frac{1}{2}} \text{.}
    \end{align*}

    \item The following Hodge decomposition holds
    \begin{align*}
        \dot{\mathrm{H}}^{s,p}(\mathbb{R}^n_+,\Lambda^k) = \dot{\mathrm{H}}^{s,p}_{\mathfrak{t},\sigma}(\mathbb{R}^n_+,\Lambda^k) \oplus \dot{\mathrm{H}}^{s,p}_{\gamma}(\mathbb{R}^n_+,\Lambda^k)\text{. }
    \end{align*}
\end{enumerate}
Moreover, the result remain true if we replace
\begin{itemize}
    \item $\dot{\mathrm{H}}^{s,p}$ by $\dot{\mathrm{B}}^{s}_{p,q}$, $q\in[1,+\infty]$ ;
    \item  $(\dot{\mathrm{H}},\dot{\mathrm{B}})$ by $({\mathrm{H}},{\mathrm{B}})$.
\end{itemize}
The symbol $\mathrm{X}_\gamma$ stands for the range of $\mathrm{I}-\mathbb{P}$ in $\mathrm{X}$.\footnote{The subscript (or exponent in case of Besov spaces) $\gamma$ is a legacy of the writing of $\mathrm{G}$ spaces as spaces of gradients of scalar functions in the case of vector fields.}
\end{theorem}

The way we reach Theorem \ref{thm:HodgeDecompRn+Intro} through intermediate results and proofs is so that we recover many different properties of the Hodge Laplacian as well as its bounded holomorphic functional calculus on Sobolev and Besov spaces almost for free. This is due to the particular structure of the boundary of $\mathbb{R}^n_+$, and the properties of the Laplacian on the whole space $\mathbb{R}^n$. This, above Theorem \ref{thm:HodgeDecompRn+Intro}, and the fact that one can define the Hodge-Stokes operator as
\begin{align*}
    u\in \dot{\mathrm{D}}^{s}_{p}({\mathbb{A}_\mathcal{H}}) = \mathbb{P}\dot{\mathrm{D}}^{s}_{p}({\Delta_\mathcal{H}})\,\text{ and }\,\mathbb{A}_{\mathcal{H}} u := -{\Delta_\mathcal{H}} u = \mathrm{d}^{\ast}\mathrm{d} u \text{, }
\end{align*}
will yield automatically
\begin{theorem}[{see Theorem \ref{thm:HinftyFuncCalcHodgeStokesMaxwell}}]\label{thm:HinftyFuncCalcHodgeStokesMaxwellIntro} Let $p\in(1,+\infty)$, $s\in(-1+1/p,1/p)$. For all $\mu\in (0,\pi)$, the operator $\mathbb{A}_{\mathcal{H}}$ admits a bounded ($\mathrm{\mathbf{H}}^{\infty}(\Sigma_\mu)$-)holomorphic functional calculus on $\dot{\mathrm{H}}^{s,p}_{\mathfrak{t},\sigma}(\mathbb{R}^n_+,\Lambda)$.
Moreover, the result remains true if we replace
\begin{itemize}
    \item $\dot{\mathrm{H}}^{s,p}$ by $\dot{\mathrm{B}}^{s}_{p,q}$, $q\in[1,+\infty]$;
    \item  $(\dot{\mathrm{H}},\dot{\mathrm{B}})$ by $({\mathrm{H}},{\mathrm{B}})$.
\end{itemize}
\end{theorem}

We mention that our strategy is not morally so different from the one presented in \cite[Beginning~of~Section~4]{GeissertHeckTrunk2013}, identifying some Neumann and Dirichlet boundary conditions on various components. However, the treatment of boundary value is done in a more careful way, adapted with the scale of homogeneous function spaces, thanks to a weak-strong correspondence of (partial) traces by the mean of appropriate results in the Appendix \ref{Append:TracesofFunctions}. 

\subsubsection{Global-in-time estimates in \texorpdfstring{$\mathrm{L}^q$}{Lq}-maximal regularity: the role of homogeneous function spaces and their interpolation} Another tool which is central in the study of parabolic equations and also for a large class of fluid dynamics problems is the $\mathrm{L}^q$-maximal regularity.

The general problem of global in time $\mathrm{L}^q$-maximal regularity is: for a closed operator $(\mathrm{D}(A),A)$ on a Banach space $X$, let us consider the evolution equation
\begin{align}\label{eq:ACPinftyIntro}
    \left\{\begin{array}{rl}
            \partial_t u(t) +Au(t)  =& f(t) \,\text{, } t\in(0,+\infty) \text{, }\\
            u(0) =& 0\text{. }
    \end{array}
    \right.\text{. }
\end{align}
Provided $q\in[1,+\infty]$ and $f\in\mathrm{L}^{q}((0,+\infty),X)$, can we solve uniquely \eqref{eq:ACPinftyIntro}, with \textit{a priori} estimate
\begin{align}\label{ineq:estimateMaxRegIntro}
    \lVert (\partial_t u, Au)\rVert_{\mathrm{L}^q((0,+\infty),X)} \lesssim \lVert f\rVert_{\mathrm{L}^q((0,+\infty),X)} \text{ ?}
\end{align}

When $X$ is a UMD Banach space (\textit{i.e.} a space such that the Hilbert transform is bounded on $\mathrm{L}^r(\mathbb{R},X)$ for one (or equivalently all) $r\in(1,+\infty)$), the problem has been extensively studied in \cite[Chapter~III,~Section~4]{AmmanBookParabolicVolI1995}, \cite{bookDenkHieberPruss2003} and \cite{KunstmannWeis2004}. It has been proved in this case, that the truthfulness of \eqref{ineq:estimateMaxRegIntro} for $q\in(1,+\infty)$ is equivalent to the $\mathcal{R}$-boundedness of the resolvent $A$ of angle $\phi_{{A}}^{\mathcal{R}}<\frac{\pi}{2}$, \textit{i.e.} if $\sigma(A)\subset \overline{\Sigma}_{\phi}$, and for some $\frac{\pi}{2}>\mu>\phi$, the set
\begin{align*}
    \{\lambda(\lambda \mathrm{I}-A)^{-1}\}_{\lambda \in \mathbb{C}\setminus {\Sigma}_{\mu}}
\end{align*}
is $\mathcal{R}$-bounded, and $\phi_{{A}}^{\mathcal{R}}$ is the infimum on all such $\mu$. This result was initially due to Weis, see \cite[Theorem~4.2]{Weis2001}. 

It has been shown in many cases that Stokes operators satisfy the $\mathrm{L}^q$-maximal regularity on $\mathrm{L}^p_{\sigma}(\Omega)$, for various class of open sets $\Omega$ for $p,q\in(1,+\infty)$, with various boundary conditions and this has been widely used to treat various fluid dynamics problems, mainly Navier-Stokes equations. See for instance \cite{GigaSohr1991,Tolksdorf2018-1,TolksdorfWatanabe,MitreaMonniaux2009-1,Monniaux2013,Monniaux2021,HieberMonniaux2013,HieberNesensohnPrussSchade2016,Hieber2020}.

Getting back to the abstract problem \eqref{eq:ACPinftyIntro}, when $X$ is UMD, $q\in(1,+\infty)$ and $A$ is invertible, \textit{i.e.} $0\in\rho(A)$, it is known that the solution $u$ belongs to $\mathrm{C}^0_0(\mathbb{R}_+,(X,\mathrm{D}(A))_{1-1/q,q})$ with estimate
\begin{align}\label{eq:estimatecontinuousMaxRegIntro}
    \lVert u \rVert_{\mathrm{L}^{\infty}(\mathbb{R}_+,(X,\mathrm{D}(A))_{1-1/q,q})} \lesssim_{A,q} \lVert f\rVert_{\mathrm{L}^q(\mathbb{R}_+,X)}\text{, }
\end{align}
where $(\cdot,\cdot)_{\theta,r}$, $(\theta,r)\in(0,1)\times [1,+\infty]$, stands for the real interpolation functor; check for instance \cite[Chapter~III,~Theorem~4.10.2]{AmmanBookParabolicVolI1995}.

If $0\in\sigma(A)$, one only has $u\in \mathrm{C}^0(\mathbb{R}_+,(X,\mathrm{D}(A))_{1-1/q,q})$, with for each $T<+\infty$,
\begin{align}\label{eq:traceEstnoninvertible}
    \lVert u \rVert_{\mathrm{L}^{\infty}([0,T),(X,\mathrm{D}(A))_{1-1/q,q})} \lesssim_{A,q,T} \lVert f\rVert_{\mathrm{L}^q(\mathbb{R}_+,X)}\text{, }
\end{align}
where the implicit constant blows up as $T$ goes to $+\infty$. Notice that this is the case for the (negative) Laplacian $A=-\Delta$ on $X=\mathrm{L}^p(\mathbb{R}^n)$, $p\in(1,+\infty)$, which is quite inconvenient, see \cite[Section~1]{Gaudin2023} and reference therein for more details.

Another issue is that one cannot reach $\mathrm{L}^1$ and $\mathrm{L}^\infty$-maximal regularity estimates through above theory, but one may recover such kind of results if $X$ is replaced by a real interpolation space between $X$ and $\mathrm{D}(A)$, say $Y_{r}^{\theta}:=(X,\mathrm{D}(A))_{\theta,r}$ $(\theta,r)\in(0,1)\times[1,+\infty]$. Indeed, a theorem of Da Prato and Grisvard, see \cite[Theorem~4.15]{DaPratoGrisvard1975}, gives us that, provided either
\begin{itemize}
    \item $0\in\rho(A)$ and $T\in(0,+\infty]$,
    \item $0\in\sigma(A)$ and $T\in(0,+\infty)$,
\end{itemize}
for $q\in[1,+\infty)$, $\theta\in(0,1/q)$, the solution $u$ to \eqref{eq:ACPinftyIntro} belongs to $\mathrm{C}^{0}_b([0,T),Y_{q}^{1+\theta-1/q})$ and satisfies
\begin{align}\label{eq:DPGIntro}
    \lVert u\rVert_{\mathrm{L}^\infty((0,T),Y_{q}^{1+\theta-1/q})} \lesssim_{A,q(,T)} \lVert (\partial_t u, Au)\rVert_{\mathrm{L}^q((0,T),Y_{q}^{\theta})} \lesssim_{A,q(,T)} \lVert f\rVert_{\mathrm{L}^q((0,T),Y_{q}^{\theta})}\text{. }
\end{align}

If one wants to recover global in time estimates in \eqref{eq:DPGIntro} by the mean of \cite[Theorem~4.15]{DaPratoGrisvard1975}, we have to assume that $A$ is invertible, hence $0\in\rho(A)$. However, for $q=1$ similar estimates have been shown for several non-invertible operators and were of major importance to achieve existence in critical function spaces for some fluid dynamic problems like global well-posedness of Navier-Stokes equations, even for inhomogeneous flows, or free boundary problems, see for instance \cite{Chemin1999,DanchinMucha2009,DanchinMucha2015,OgawaShimizu2016,OgawaShimizu2021,OgawaShimizu2022}.

While the work of Ogawa and Shimizu \cite{OgawaShimizu2022} provides a powerful framework for many applications, we are mainly restricted to a specific class of second order elliptic operator with "smooth enough" coefficients. A different and more abstract approach was brought by the recent work of Danchin, Hieber, Mucha and Tolksdorf \cite{DanchinHieberMuchaTolk2020}, where the idea was to give an homogeneous version of the Da Prato-Grisvard theorem \cite[Theorem~4.15]{DaPratoGrisvard1975}, in the sense that \eqref{eq:DPGIntro} holds with implicit constant uniform with respect to the time variable even when $0\in\sigma(A)$. But further assumptions have to be made, mainly the injectivity of $A$ on $X$. Their idea was to replace the real interpolation space $Y^{\theta}_q=(X,\mathrm{D}(A))_{\theta,q}$ in \eqref{eq:DPGIntro} by
\begin{align*}
    (X,\mathrm{D}(\mathring{A}))_{\theta,q}
\end{align*}
where $\mathrm{D}(\mathring{A})$ is called the homogeneous domain of $A$ and stands \textit{morally} for the \textit{closure} of $\mathrm{D}(A)$ with respect to the (semi-)norm $\lVert A \cdot \rVert_X$. Such kind of investigation was already made in Haase's book \cite[Chapter~6]{bookHaase2006} where the \textit{completion} is considered instead of the closure. 

If $A$ is a non-degenerate elliptic operator of order $2m$ equal to its principal part, densely defined on $X=\mathrm{L}^{p}(\mathbb{R}^n)$, such that $\mathrm{D}({A})=\mathrm{H}^{2m,p}(\mathbb{R}^n)$, we should have $\mathrm{D}(\mathring{A})=\dot{\mathrm{H}}^{2m,p}(\mathbb{R}^n)$ which encounter no trouble of definition, if one consider the construction of homogeneous Sobolev and Besov spaces as equivalence classes of tempered distributions up to a polynomial, see \cite[Chapter~8,~Section~3]{bookHaase2006}, \cite[chapter~5]{bookTriebel1983}. In this case, we obtain
\begin{align*}
    (X,\mathrm{D}({A}))_{\theta,q} = {\mathrm{B}}^{2m\theta}_{p,q}(\mathbb{R}^n) \text{ and }(X,\mathrm{D}(\mathring{A}))_{\theta,q} = \dot{\mathrm{B}}^{2m\theta}_{p,q}(\mathbb{R}^n) \text{. }
\end{align*}

However, if one wants to consider a similar problem on a domain, here the half-space $\mathbb{R}^n_+$, with some boundary conditions, with the definition of homogeneous function spaces as class of tempered distributions up to a polynomial, it is not clear that we can make a proper meaning of boundary conditions or traces. To overcome such difficulties, a construction of homogeneous Sobolev spaces $\dot{\mathrm{H}}^{s,p}(\mathbb{R}^n_+)$ and a review of homogeneous Besov spaces on $\mathbb{R}^n_+$ allowing to check interpolation between homogeneous spaces, and to recover boundary conditions in some cases, is done in \cite[Chapters~3~\&~4]{DanchinHieberMuchaTolk2020} continued in \cite{Gaudin2022}. This construction is based on, and consistent with, the one of homogeneous Besov spaces on the whole space achieved in \cite[Chapter~2]{bookBahouriCheminDanchin}. This leads in some cases to non-complete normed vector spaces as $\mathrm{D}(\mathring{A})$ could be if one wants to consider it as the (moral) \textit{closure} of $\mathrm{D}(A)$ in an appropriate subspace of $\dot{\mathrm{H}}^{2m,p}(\mathbb{R}^n_+)$, which may be not complete. That is why the construction of the homogeneous domain of $A$, its real interpolation spaces with $X$, and the homogeneous Da Prato-Grisvard theorem from \cite[Chapter~2]{DanchinHieberMuchaTolk2020} are interesting: they allow $\mathrm{D}(\mathring{A})$ to be a non-complete normed vector space\footnote{We notice that real interpolation of non-complete vector space makes sense, see \cite[Chapter~3]{BerghLofstrom1976} where completeness is not needed to deal with the $K$-method.}. This could be necessary if one wants to deal with boundary conditions. One can even recover the construction given in \cite[Chapter~6]{bookHaase2006} when the completion is used instead of the closure to construct the homogeneous domain.

We also mention that the homogeneous operator and interpolation theory revisited by Danchin, Hieber, Mucha, and Tolksdorf in \cite[Chapter~2]{DanchinHieberMuchaTolk2020} also gives a way to circumvent the lack of global-in-time estimates in the usual $\mathrm{L}^q$-maximal regularity framework, for the trace estimate issue \eqref{eq:traceEstnoninvertible}. This have been done by the author in a previous paper by the author, see \cite[Sections~1,~2\&~4]{Gaudin2023}.

Finally, if one applies the homogeneous interpolation and then the homogeneous Da Prato-Grisvard theorem as done in \cite[Chapters~2,~3~\&~4]{DanchinHieberMuchaTolk2020}, choosing $X\subset \mathrm{L}^p(\mathbb{R}^n_+)$ to be a closed subspace, and $\mathrm{D}(A)$ to be a closed subset of $\mathrm{H}^{2m,p}(\mathbb{R}^n_+)$, it would lead to $\mathrm{L}^1_t(\dot{\mathrm{B}}^{2m\theta}_{p,1})$-maximal regularity results with $\theta \in(0,1)$. Proceeding this way disallow to obtain  $\mathrm{L}^1_t(\dot{\mathrm{B}}^{0}_{p,1})$ or $\mathrm{L}^1_t(\dot{\mathrm{B}}^{\alpha}_{p,1})$-maximal regularity results, for $\alpha<0$.
Our idea is to replace the use of $\mathrm{L}^p(\mathbb{R}^n_+)$ as a ground space by $\dot{\mathrm{H}}^{s,p}(\mathbb{R}^n_+)$, with $p\in(1,+\infty)$, $s\in(-1+1/p,1/p)$, so that we may expect to realize $A$ on $\dot{\mathrm{H}}^{s,p}(\mathbb{R}^n_+)$ with domain
\begin{align*}
    \mathrm{D}(A) \subset \dot{\mathrm{H}}^{s,p}(\mathbb{R}^n_+)\cap \dot{\mathrm{H}}^{s+2m,p}(\mathbb{R}^n_+)\text{.}
\end{align*}
Therefore, for $s\in(-1+1/p,1/p)$ and $\theta\in(0,1)$, it seems reasonable to expect $\mathrm{L}^1_t(\dot{\mathrm{B}}^{s+2m\theta}_{p,1})$-maximal regularity results, and then recover maximal regularity for some non-positive index of regularity.

To reach such realizations of $A$ on homogeneous Sobolev spaces of fractional order on the whole space or on the half-space, we are going to use the construction started \cite[Chapter~3]{DanchinHieberMuchaTolk2020}, and continued in \cite[Section~2]{Gaudin2022}. The Appendix \ref{Append:TracesofFunctions} is dedicated to the meaning of partial traces in such function spaces to ensure that one can realize operators with boundary conditions on $\dot{\mathrm{H}}^{s,p}(\mathbb{R}^n_+)$, provided $s\in(-1+1/p,1/p)$, see also \cite[Section~3]{Gaudin2022} for usual traces results. We will also provide additional tools that will be useful to compute homogeneous interpolation spaces in presence of boundary conditions as in Section \ref{Sect:HomInterpDomainsandHomMaxReg}.

In our case, considering first the Hodge Laplacian, then the Hodge-Stokes operator, we will be able to apply the homogeneous Da Prato-Grisvard theorem \cite[Theorem~2.20]{DanchinHieberMuchaTolk2020}, as well as the usual $\mathrm{L}^q$-maximal regularity for UMD Banach spaces or \cite[Theorem~4.7]{Gaudin2023}, to reach various maximal regularity results as the next one.
\begin{theorem}[{see Theorems \ref{thm:MaxRxBesovUMDHodgeStokesRn+}, \ref{thm:MaxRxHspUMDHodgeStokesRn+}, \ref{thm:MaxRxBesovHodgeStokesRn+} \& \ref{thm:MaxRxBesovNavierSlipStokesRn+}}]Let $p\in(1,+\infty)$, $q\in[1,+\infty)$, $s\in(-1+1/p,1/p+2/q)$, $s,s+2-2/q\notin \mathbb{N}+\frac{1}{p}$, such that $(\mathcal{C}_{s+2-2/q,p,q})$\footnote{The condition ($\mathcal{C}_{\cdot,\cdot,\cdot}$) is here to ensure the completeness of considered function spaces. See \eqref{AssumptionCompletenessExponents} below.} is satisfied and let $T\in(0,+\infty]$.
For any $f\in \mathrm{L}^q((0,T),\dot{\mathrm{B}}^{s,\sigma}_{p,q,\mathcal{\mathcal{H}}}(\mathbb{R}^n_+,\Lambda))$, $u_0\in \dot{\mathrm{B}}^{2+s-\frac{2}{q},\sigma}_{p,q,\mathcal{H}}(\mathbb{R}^n_+,\Lambda)$, there exists a unique mild solution $u \in \mathrm{C}^0_b([0,T),\dot{\mathrm{B}}^{2+s-\frac{2}{q},\sigma}_{p,q,\mathcal{H}}(\mathbb{R}^n_+,\Lambda))$ to
\begin{equation*}\tag{HSS}\footnote{For introductory purpose, the notations here are either not precise enough or quite redundant. For instance, the condition  $\mathrm{d}^\ast u =0$ already implies the boundary condition $\nu \iprod u_{|_{\partial\mathbb{R}^n_+}}=0$.}
    \left\{ \begin{array}{rclr}
         \partial_t u - \Delta u  & = &f,& \text{ on } (0,T)\times\mathbb{R}^n_+,\\
         \mathrm{d}^\ast u& = & 0,& \text{ on } (0,T)\times\mathbb{R}^n_+,\\
         \nu\iprod\mathrm{d} u_{|_{\partial\mathbb{R}^n_+}}& = & 0,& \text{ on } (0,T)\times\partial\mathbb{R}^n_+,\\
         \nu\iprod u_{|_{\partial\mathbb{R}^n_+}}& = & 0,& \text{ on } (0,T)\times\partial\mathbb{R}^n_+,\\
         u(0)& = & u_0,& \text{ in } \dot{\mathrm{B}}^{2+s-\frac{2}{q}}_{p,q}(\mathbb{R}^n_+,\Lambda),
    \end{array}
    \right.
\end{equation*}
with estimate
\begin{align*}
    \left\lVert u \right\rVert_{\mathrm{L}^\infty((0,T),\dot{\mathrm{B}}^{2+s-\frac{2}{q}}_{p,q}(\mathbb{R}^n_+))} + \left\lVert (\partial_t u, \nabla^2 u) \right\rVert_{\mathrm{L}^q((0,T),\dot{\mathrm{B}}^{s}_{p,q}(\mathbb{R}^n_+))} \lesssim_{p,q,s,n} \left\lVert f \right\rVert_{\mathrm{L}^q((0,T),\dot{\mathrm{B}}^{s}_{p,q}(\mathbb{R}^n_+))} + \left\lVert u_0 \right\rVert_{\dot{\mathrm{B}}^{2+s-\frac{2}{q}}_{p,q}(\mathbb{R}^n_+)}.
\end{align*}
In the case $q=+\infty$, if we assume in addition $u_0\in \dot{\mathrm{D}}^{s}_{p}(\mathbb{A}_{\mathcal{H}}^2)$, we have
\begin{align*}
     \left\lVert (\partial_t u, \nabla^2 u) \right\rVert_{\mathrm{L}^\infty((0,T),\dot{\mathrm{B}}^{s}_{p,\infty}(\mathbb{R}^n_+))} \lesssim_{p,s,n} \left\lVert f \right\rVert_{\mathrm{L}^\infty((0,T),\dot{\mathrm{B}}^{s}_{p,\infty}(\mathbb{R}^n_+))} + \left\lVert \mathbb{A}_{\mathcal{H}} u_0 \right\rVert_{\dot{\mathrm{B}}^{s}_{p,\infty}(\mathbb{R}^n_+)}.
\end{align*}
\end{theorem}

\subsection*{Acknowledgment} The author would like to thank Sylvie Monniaux for her numerous feedbacks which helped to improve the manuscript. The author would also like to thank Yoshizaku Giga for personal discussions that heavily inspired the last section \ref{sec:NaviernoSlipBCRn+}.

\subsection{Notations, definitions and review of usual concepts}

Throughout this paper the dimension will be $n\geqslant 2$, and $\mathbb{N}$ will be the set of non-negative integers.

\subsubsection{Smooth and measurable functions}

Denote by $\eus{S}(\mathbb{R}^n,\mathbb{C})$ the space of complex valued Schwartz function, and $\eus{S}'(\mathbb{R}^n,\mathbb{C})$ its dual called the space of tempered distributions. The Fourier transform on $\eus{S}'(\mathbb{R}^n,\mathbb{C})$ is written $\eus{F}$, and is pointwise defined for any $f\in\mathrm{L}^1(\mathbb{R}^n,\mathbb{C})$ by
\begin{align*}
  \eus{F}f(\xi) :=\int_{\mathbb{R}^n} f(x)\,e^{-ix\cdot\xi}\,\mathrm{d}x\text{, } \xi\in\mathbb{R}^n\text{. }
\end{align*}
Additionnally, for $p\in[1+\infty]$, we will write $p'=\tfrac{p}{p-1}$ its \textit{\textbf{H\"{o}lder conjugate}}.

For any $m\in\mathbb{N}$, the map $\nabla^m\,:\,\eus{S}'(\mathbb{R}^n,\mathbb{C})\longrightarrow \eus{S}'(\mathbb{R}^n,\mathbb{C}^{n^m})$ is defined as $\nabla^m u := (\partial^\alpha u)_{|\alpha|=m}$.
We denote by $(e^{t\Delta})_{t\geqslant0}$ and $(e^{-t(-\Delta)^\frac{1}{2}})_{t\geqslant0}$ respectively the heat and Poisson semigroup on $\mathbb{R}^n$. We also introduce operators $\nabla'$ and $\Delta'$ which are respectively the gradient and the Laplacian on $\mathbb{R}^{n-1}$ identified with the $n-1$ first variables of $\mathbb{R}^n$, \textit{i.e.} $\nabla'=(\partial_{x_1}, \ldots, \partial_{x_{n-1}})$ and $\Delta' = \partial_{x_1}^2 + \ldots + \partial_{x_{n-1}}^2$.

When $\Omega$ is an open set of $\mathbb{R}^n$, for $p\in[1,+\infty)$, $\mathrm{L}^p(\Omega,\mathbb{C})$ is the normed vector space of complex valued (Lebesgue-) measurable functions whose $p$-th power is integrable with respect to the Lebesgue measure, $\eus{S}(\overline{\Omega},\mathbb{C})$ (\textit{resp.} $\mathrm{C}_c^\infty(\overline{\Omega},\mathbb{C})$) stands for functions which are restrictions on $\Omega$ of elements of $\eus{S}(\mathbb{R}^n,\mathbb{C})$ (\textit{resp.} $\mathrm{C}_c^\infty(\mathbb{R}^n,\mathbb{C})$). Unless the contrary is explicitly stated, we will always identify $\mathrm{L}^p(\Omega,\mathbb{C})$ (resp. $\mathrm{C}_c^\infty(\Omega,\mathbb{C})$) as the subspace of function in $\mathrm{L}^p(\mathbb{R}^n,\mathbb{C})$ (resp. $\mathrm{C}_c^\infty(\mathbb{R}^n,\mathbb{C})$) supported in $\overline{\Omega}$ through the extension by $0$ outside $\Omega$. $\mathrm{L}^\infty(\Omega,\mathbb{C})$ stands for the space of essentially bounded (Lebesgue-) measurable functions.

For $s\in\mathbb{R}$, $p\in[1,+\infty)$, $\ell^p_s(\mathbb{Z},\mathbb{C})$, stands for the normed vector space of $p$-summable sequences of complexes numbers with respect to the counting measure $2^{ksp}\mathrm{d}k$;  $\ell^\infty_s(\mathbb{Z},\mathbb{C})$ stands for sequences $(x_k)_{k\in\mathbb{Z}}$ such that $(2^{ks}x_k)_{k\in\mathbb{Z}}$  is bounded.
More generally, when $X$ is a Banach space, for $p\in[1,+\infty]$, one may also consider $\mathrm{L}^p(\Omega,X)$ which stands for the space of (Bochner-)measurable functions $u\,:\,\Omega\longrightarrow X$, such that $t\mapsto\lVert u(t)\rVert_X \in \mathrm{L}^p(\Omega,\mathbb{R})$, similarly one may consider $\ell^p_s(\mathbb{Z},X)$.

\subsubsection{(Bi)Sectorial operators on Banach spaces}

We introduce the following subsets of the complex plane
\begin{align*}
    \Sigma_\mu &:=\{ \,z\in\mathbb{C}^\ast\,:\,\lvert\mathrm{arg}(z)\rvert<\mu\,\}\text{, if } \mu\in(0,\pi)\text{, }\\
    S_\mu & := (-\Sigma_\mu)\cup \Sigma_\mu \text{, if } \mu\in(0,\frac{\pi}{2})\text{, }
\end{align*}
we also define $\Sigma_0 := (0,+\infty)$, $S_0 := \mathbb{R}$, and later we are going to consider $\overline{\Sigma}_\mu$, and $\overline{S}_\mu$ their closure.

An operator $(\mathrm{D}(A),A)$ on complex valued Banach space $X$ is said to be $\omega$-\textit{\textbf{sectorial}}, if for a fixed $\omega\in (0,\pi)$, both conditions are satisfied
\begin{enumerate}
    \item $\sigma(A)\subset \overline{\Sigma}_\omega $, where $\sigma(A)$ stands for the spectrum of $A$ ;
    \item For all $\mu\in(\omega,\pi)$, $\sup_{\lambda\in \mathbb{C}\setminus\overline{\Sigma}_\mu}\lVert \lambda(\lambda \mathrm{I}-A)^{-1}\rVert_{X\rightarrow X} < +\infty$.
\end{enumerate}
Similarly, $(\mathrm{D}(A),A)$ is said to be $\omega$\textit{\textbf{-bisectorial}}, for a fixed $\omega\in (0,\frac{\pi}{2})$, if $\sigma(A)\subset \overline{S}_\omega $, and for all $\mu\in(\omega,\frac{\pi}{2})$, $\sup_{\lambda\in \mathbb{C}\setminus\overline{S}_\mu}\lVert \lambda( \lambda\mathrm{I}- A)^{-1}\rVert_{X\rightarrow X} < +\infty$.

The following proposition will be of paramount importance throughout the present work.
\begin{proposition}[{\cite[Proposition~3.2.2]{EgertPhDThesis2015}} ]Let $(\mathrm{D}(A),A)$ be a sectorial operator on a Banach space $X$. Then the following assertions hold.
\begin{enumerate}
    \item If $k\in\mathbb{N}$, and $x\in\overline{\mathrm{D}(A)}$, then
    \begin{align*}
        \lim_{t\rightarrow +\infty} t^{k}(t\mathrm{I}+A)^{-k}x =x \,\text{ and }\, \lim_{t\rightarrow +\infty} A^{k}(t\mathrm{I}+A)^{-k}x =0 \text{. }
    \end{align*}
    \item If $k\in\mathbb{N}$, and $x\in\overline{\mathrm{R}(A)}$, then
    \begin{align*}
        \lim_{t\rightarrow 0} t^{k}(t\mathrm{I}+A)^{-k}x =0 \,\text{ and }\, \lim_{t\rightarrow 0} A^{k}(t\mathrm{I}+A)^{-k}x =x \text{. }
    \end{align*}
    In particular, $\mathrm{N}(A)\cap\overline{\mathrm{R}(A)}=\{ 0 \}$, so that $X=\overline{\mathrm{R}(A)}$ implies that $A$ is injective.
    \item For every $k\in\mathbb{N}$, $\mathrm{D}(A^k)\cap{\mathrm{R}(A^k)}$ is dense in $\overline{\mathrm{D}(A)}\cap\overline{\mathrm{R}(A)}$.
    \item If $X$ is reflexive, then $A$ is densely defined and induces a topological decomposition
    \begin{align*}
        X = \mathrm{N}(A)\oplus\overline{\mathrm{R}(A)}\text{. }
    \end{align*}
\end{enumerate}
\end{proposition}

Provided $\mu\in(0,\pi)$, we denote by $\mathrm{\mathbf{H}}^\infty(\Sigma_\mu)$, the set of bounded holomorphic functions on $\Sigma_\mu$, the same goes with $S_\mu$ instead of $\Sigma_\mu$, for $\mu\in[0,\frac{\pi}{2})$.

If $(\mathrm{D}(A),A)$ is $\omega$-(bi)sectorial with $\omega\in[0,\pi)$ (resp. $[0,\frac{\pi}{2})$), for $\mu\in(\omega,\pi)$ (resp. $(\omega,\frac{\pi}{2})$) we say that $A$ admits a \textit{\textbf{bounded}} (or $\mathrm{\mathbf{H}}^\infty(\Sigma_\mu)$- (resp. $\mathrm{\mathbf{H}}^\infty(S_\mu)$-)) \textit{\textbf{holomorphic functional calculus}} on $X$ (of angle $\mu$), if for $\theta\in (\omega,\mu)$, there exists a constant $K_\theta$, such that for all $f\in\mathrm{\mathbf{H}}^\infty(\Sigma_\theta)$ (resp. $\mathrm{\mathbf{H}}^\infty(S_\theta)$), we have that
\begin{align*}
    \lVert f(A)\rVert_{X\rightarrow X}\leqslant K_\theta \lVert f \rVert_{\mathrm{L}^\infty}\text{. }
\end{align*}
For all $x\in \mathrm{D}(A)\cap\mathrm{R}(A)$, $f(A)x$ is defined by the following convergent integral, 
\begin{align}\label{eq:DunfordintegralFuncCalc}
    f(A)x = \frac{1}{2 i \pi}\int_{\partial \Sigma_\theta} f(z)(z\mathrm{I}-A)^{-1}x\,\mathrm{d}z\text{, }
\end{align}
and the same goes with $\partial S_\theta$ instead of $\partial \Sigma_\theta$ for the bisectorial case, both boundary being oriented counterclockwise.

We also say that $A$ has  {\textbf{bounded imaginary powers}} (BIP) of type $\theta_A$ if $f(z)=z^{is}$ plugged in \eqref{eq:DunfordintegralFuncCalc} yields a bounded linear operator for all $s\in\mathbb{R}$, and
\begin{align*}
    \theta_A := \inf \left\{ \omega\geqslant 0\,\bigg|\, \sup_{s\in\mathbb{R}} e^{-\omega|s|}\lVert A^{is}\rVert_{X \rightarrow X}<+\infty\right\}\text{ . }
\end{align*}
Notice that if $A$ has a bounded holomorphic functional calculus then it has bounded imaginary powers.

The functional calculus of sectorial operators is widely reviewed in several references but we mention here Haase's book \cite{bookHaase2006}. However, there is only few references known to the author that deal with a systematic treatment of bisectorial operators, Egert did such a treatment producing proofs of corresponding usual results in his Thesis, see \cite[Chapter~3]{EgertPhDThesis2015}.

\subsubsection{Interpolation of normed vector spaces}

Let $(X,\left\lVert\cdot\right\rVert_X)$ and $(Y,\left\lVert\cdot\right\rVert_Y)$ be two normed vector spaces. We write $X\hookrightarrow Y$ to say that $X$ embeds continuously in $Y$. Now let us recall briefly basics of interpolation theory. If there exists a Hausdorff topological vector space $Z$, such that $X,Y\subset Z$, then $X\cap Y$ and $X+Y$ are normed vector spaces with their canonical norms, and one can define the $K$-functional of $z\in X+Y$, for any $t>0$ by
\begin{align*}
    K(t,z,X,Y) := \underset{\substack{(x,y)\in X\times Y,\\ z=x+y}}{\inf}\left({\left\lVert{x}\right\rVert_{X}+t\left\lVert{y}\right\rVert_{Y}}\right)\text{. }
\end{align*}
This allows us to construct, for any $\theta\in(0,1)$, $q\in[1,+\infty]$, the real interpolation spaces between $X$ and $Y$ with indexes $\theta,q$ as
\begin{align*}
    (X,Y)_{\theta,q} := \left\{\, x\in X+Y\,\Big{|}\,t\longmapsto t^{-\theta}K(t,x,X,Y)\in\mathrm{L}^q_\ast(\mathbb{R}_+)\,\right\}\text{, }
\end{align*}
where $\mathrm{L}^q_\ast(\mathbb{R}_+):=\mathrm{L}^q((0,+\infty),\mathrm{d}t/t)$. The interested reader could check \cite[Chapter~1]{bookLunardiInterpTheory}, \cite[Chapter~3]{BerghLofstrom1976} for more informations about real interpolation and its applications.

\subsubsection{Sobolev and Besov spaces on \texorpdfstring{$\mathbb{R}^n$}{Rn}}
To deal with Sobolev and Besov spaces on the whole space, we need to introduce Littlewood-Paley decomposition given by $\phi\in \mathrm{C}_c^\infty(\mathbb{R}^n)$, radial, real-valued, non-negative,
such that
\begin{itemize}[label={$\bullet$}]
    \item $\supp \phi \subset B(0,4/3)$;
    \item ${\phi}_{|_{B(0,3/4)}}=1$;
\end{itemize}
so we define the following functions for any $j\in\mathbb{Z}$ for all $\xi\in\mathbb{R}^n$,
\begin{align*}
    \phi_j(\xi):=\phi(2^{-j}\xi)\text{, }\qquad \psi_j(\xi) := \phi_{j}(\xi/2)-\phi_{j}(\xi)\text{,}
\end{align*}
and the family $(\psi_j)_{j\in\mathbb{Z}}$ has the following properties
\begin{itemize}[label={$\bullet$}]
    \item $\mathrm{supp}(\psi_j)\subset \{\,\xi\in\mathbb{R}^n\,|\,3\cdot 2^{j-2}\leqslant\left\lvert{\xi}\right\rvert \leqslant 2^{j+2}/3\,\}$;
    \item $\forall\xi\in\mathbb{R}^n\setminus\{0\}$, $\sum\limits_{j=-M}^N{\psi_j}(\xi)\xrightarrow[N,M\rightarrow+\infty]{} 1$.
\end{itemize}
Such a family $(\phi,(\psi_j)_{j\in\mathbb{Z}})$ is called a Littlewood-Paley family. Now, we consider the two following families of operators associated with their Fourier multipliers :
\begin{itemize}[label={$\bullet$}]
    \item The \textit{\textbf{homogeneous}} family of Littlewood-Paley dyadic decomposition operators $(\dot{\Delta}_j)_{j\in\mathbb{Z}}$, where
    \begin{align*}
        \dot{\Delta}_j := \eus{F}^{-1}\psi_j\eus{F},
    \end{align*}
    \item The \textit{\textbf{inhomogeneous}} family of Littlewood-Paley dyadic decomposition operators $({\Delta}_k)_{k\in\mathbb{Z}}$, where
    \begin{align*}
       {\Delta}_{-1} := \eus{F}^{-1}\phi\eus{F}\text{, }
    \end{align*}
    $\Delta_k:=\dot{\Delta}_k$ for any $k\geqslant 0$, and $\Delta_k:=0$ for any $k\leqslant-2$.
\end{itemize}
One may notice, as a direct application of Young's inequality for the convolution, that they are all uniformly bounded families of operators on $\mathrm{L}^p(\mathbb{R}^n)$, $p\in[1,+\infty]$.

Both family of operators lead for $s\in\mathbb{R}$, $p,q\in[1,+\infty]$, $u\in \eus{S}'(\mathbb{R}^n)$ to the following quantities,
\begin{align*}
    \left\lVert u \right\rVert_{\mathrm{B}^{s}_{p,q}(\mathbb{R}^n)}= \left\lVert(2^{ks}\left\lVert {\Delta}_k u \right\rVert_{\mathrm{L}^{p}(\mathbb{R}^n)})_{k\in\mathbb{Z}}\right\rVert_{\ell^{q}(\mathbb{Z})}\text{ and }
    \left\lVert u \right\rVert_{\dot{\mathrm{B}}^{s}_{p,q}(\mathbb{R}^n)} = \left\lVert(2^{js}\left\lVert \dot{\Delta}_j u \right\rVert_{\mathrm{L}^{p}(\mathbb{R}^n)})_{j\in\mathbb{Z}}\right\rVert_{\ell^{q}(\mathbb{Z})}\text{, }
\end{align*}
respectively named the inhomogeneous and homogeneous Besov norms, but the homogeneous norm is not really a norm since $\left\lVert u \right\rVert_{\dot{\mathrm{B}}^{s}_{p,q}(\mathbb{R}^n)}=0$ does not imply that $u=0$. Thus, following \cite[Chapter~2]{bookBahouriCheminDanchin} and \cite[Chapter~3]{DanchinHieberMuchaTolk2020}, we introduce a subspace of tempered distributions such that $\left\lVert \cdot \right\rVert_{\dot{\mathrm{B}}^{s}_{p,q}(\mathbb{R}^n)}$ is point-separating, say
\begin{align*}
   \eus{S}'_h(\mathbb{R}^n) &:= \left\{ u\in \eus{S}'(\mathbb{R}^n)\,\Big{|}\,\forall \Theta \in \mathrm{C}_c^\infty(\mathbb{R}^n),\,  \left\lVert \Theta(\lambda \mathfrak{D}) u \right\rVert_{\mathrm{L}^\infty(\mathbb{R}^n)} \xrightarrow[\lambda\rightarrow+\infty]{} 0\right\}\text{,}
\end{align*}
where for $\lambda>0$, $\Theta(\lambda \mathfrak{D})u = \eus{F}^{-1}{\Theta}(\lambda\cdot)\eus{F}u$. Notice that $\eus{S}'_h(\mathbb{R}^n)$ does not contain any polynomials, and for any $p\in[1,+\infty)$, $\mathrm{L}^p(\mathbb{R}^n)\subset\eus{S}'_h(\mathbb{R}^n)$.

One can also define the following quantities called the inhomogeneous and homogeneous Sobolev spaces' potential norms
\begin{align*}
    \left\lVert {u} \right\rVert_{\mathrm{\mathrm{H}}^{s,p}(\mathbb{R}^n)} := \left\lVert {(\mathrm{I}-\Delta)^\frac{s}{2} u} \right\rVert_{\mathrm{L}^{p}(\mathbb{R}^n)}\text{ and } \left\lVert {u} \right\rVert_{\dot{\mathrm{H}}^{s,p}(\mathbb{R}^n)} := \Big\lVert \sum_{j\in\mathbb{Z}} (-\Delta)^\frac{s}{2}\dot{\Delta}_{j} u  \Big\rVert_{{\mathrm{L}}^{p}(\mathbb{R}^n)}\text{, }
\end{align*}
where $(-\Delta)^\frac{s}{2}$ is understood on $u\in \eus{S}'_h(\mathbb{R}^n)$ by the action on its dyadic decomposition, \textit{i.e.}
\begin{align*}
    (-\Delta)^\frac{s}{2}\dot{\Delta}_j u := \eus{F}^{-1}|\xi|^s\eus{F}\dot{\Delta}_j u\text{,}
\end{align*}
which gives a family of $\mathrm{C}^\infty$ functions with at most polynomial growth.

Hence for any  $p,q\in[1,+\infty]$, $s\in\mathbb{R}$, we define
\begin{itemize}[label={$\bullet$}]
    \item the inhomogeneous and homogeneous Sobolev (Bessel and Riesz potential) spaces,\\
    \resizebox{0.90\textwidth}{!}{$
        \mathrm{\mathrm{H}}^{s,p}(\mathbb{R}^n)=\left\{\, u\in\eus{S}'(\mathbb{R}^n) \,\big{|}\, \left\lVert {u} \right\rVert_{\mathrm{\mathrm{H}}^{s,p}(\mathbb{R}^n)}<+\infty \,\right\}\text{, }\dot{\mathrm{H}}^{s,p}(\mathbb{R}^n)=\left\{\, u\in\eus{S}'_h(\mathbb{R}^n) \,\big{|}\, \left\lVert {u} \right\rVert_{\dot{\mathrm{H}}^{s,p}(\mathbb{R}^n)}<+\infty \,\right\}\text{ ; }$}
    \item and the inhomogeneous and homogeneous Besov spaces,\\
    \resizebox{0.90\textwidth}{!}{$\mathrm{B}^{s}_{p,q}(\mathbb{R}^n)=\left\{\, u\in\eus{S}'(\mathbb{R}^n) \,\big{|}\, \left\lVert {u} \right\rVert_{\mathrm{B}^{s}_{p,q}(\mathbb{R}^n)}<+\infty \,\right\}\text{, }\dot{\mathrm{B}}^{s}_{p,q}(\mathbb{R}^n)=\left\{\, u\in\eus{S}'_h(\mathbb{R}^n) \,\big{|}\, \left\lVert {u} \right\rVert_{\dot{\mathrm{B}}^{s}_{p,q}(\mathbb{R}^n)}<+\infty \,\right\}\text{, }
    $}
\end{itemize}
which are all normed vector spaces.

The treatment of homogeneous Besov spaces $\dot{\mathrm{B}}^{s}_{p,q}(\mathbb{R}^n)$, $s\in\mathbb{R}$, $p,q\in[1,+\infty]$, defined on $\eus{S}'_h(\mathbb{R}^n)$ has been done in an extensive manner in \cite[Chapter~2]{bookBahouriCheminDanchin}. However, the corresponding construction for homogeneous Sobolev spaces $\dot{\mathrm{H}}^{s,p}(\mathbb{R}^n)$, $s\in\mathbb{R}$, $p\in(1,+\infty)$. See \cite[Chapter~1]{bookBahouriCheminDanchin} for the case $p=2$, \cite[Chapter~3]{DanchinHieberMuchaTolk2020} for the case $s\in\mathbb{N}$, \cite[Subsection~ 2.1]{Gaudin2022} for the case $s\in\mathbb{R}$.

The following subspace of Schwartz functions, say
\begin{align*}
    \eus{S}_0(\mathbb{R}^n):=\left\{\, u\in \eus{S}(\mathbb{R}^n)\,\left|\, 0\notin \supp\left(\eus{F}f\right) \,\right.\right\}\text{, }
\end{align*}
is a nice dense subspace in $\mathrm{L}^p(\mathbb{R}^n)$, $\mathrm{\mathrm{H}}^{s,p}(\mathbb{R}^n)$, $\dot{\mathrm{H}}^{s,p}(\mathbb{R}^n)$, $\mathrm{B}^{s}_{p,q}(\mathbb{R}^n)$ and $\dot{\mathrm{B}}^{s}_{p,q}(\mathbb{R}^n)$, for all $p\in(1,+\infty)$, $q\in[1,+\infty)$, $s\in\mathbb{R}$

The inhomogeneous spaces  $\mathrm{L}^p(\mathbb{R}^n)$, $\mathrm{\mathrm{H}}^{s,p}(\mathbb{R}^n)$, and $\mathrm{B}^{s}_{p,q}(\mathbb{R}^n)$ are all complete for all $p,q\in [1,+\infty]$, $s\in\mathbb{R}$, but in this setting homogenenous spaces are no longer always complete (see \cite[Proposition~1.34,~Remark~2.26]{bookBahouriCheminDanchin}). Indeed, it can be shown (see \cite[Theorem~2.25]{bookBahouriCheminDanchin}) that homogeneous Besov spaces $\dot{\mathrm{B}}^{s}_{p,q}(\mathbb{R}^n)$ are complete whenever $(s,p,q)\in\mathbb{R}\times(1,+\infty)\times[1,+\infty]$ satisfies
\begin{align*}\tag{$\mathcal{C}_{s,p,q}$}\label{AssumptionCompletenessExponents}
    \left[ s<\frac{n}{p} \right]\text{ or }\left[q=1\text{ and } s\leqslant\frac{n}{p} \right]\text{, }
\end{align*}
From now, and until the end of this paper, we write $(\mathcal{C}_{s,p})$ for the statement $(\mathcal{C}_{s,p,p})$.  One may show that, similarly, $\dot{\mathrm{H}}^{s,p}(\mathbb{R}^n)$ is complete whenever $(\mathcal{C}_{s,p})$ is satisfied, see \cite[Proposition~2.4]{Gaudin2022}.

We recall that all $s>0$, $(p,q)\in(1,+\infty)\times[1,+\infty]$, we have $\mathrm{L}^p(\mathbb{R}^n)\cap \dot{\mathrm{H}}^{s,p}(\mathbb{R}^n) = \mathrm{\mathrm{H}}^{s,p}(\mathbb{R}^n)$, and $\mathrm{L}^p(\mathbb{R}^n)\cap \dot{\mathrm{B}}^{s}_{p,q}(\mathbb{R}^n) =\mathrm{B}^{s}_{p,q}(\mathbb{R}^n)$ with equivalent norms, see \cite[Theorem~6.3.2]{BerghLofstrom1976} for more details.

According to \cite[Section~6.4]{BerghLofstrom1976}, for all $s\in\mathbb{R}$, $p,q\in(1,+\infty)\times[1,+\infty]$, $\mathrm{\mathrm{H}}^{s,p}(\mathbb{R}^n)$ and $\mathrm{B}^{s}_{p,q}(\mathbb{R}^n)$ are both complete, and moreover they are reflexive when $q\neq1,+\infty$, and we have
\begin{align}
    (\mathrm{\mathrm{H}}^{s,p}(\mathbb{R}^n))' = \mathrm{H}^{-s,p'}(\mathbb{R}^n)\text{, } (\mathrm{B}^{s}_{p,q}(\mathbb{R}^n))'=\mathrm{B}^{-s}_{p',q'}(\mathbb{R}^n) \text{, }\\
    (\mathcal{B}^{s}_{p,\infty}(\mathbb{R}^n))'=\mathrm{B}^{-s}_{p',1}(\mathbb{R}^n) \text{, } (\mathrm{B}^{s}_{p,1}(\mathbb{R}^n))'=\mathrm{B}^{-s}_{p',\infty}(\mathbb{R}^n) \text{. }
\end{align}

We recall also the usual real interpolation identities,
\begin{align*}
    (\mathrm{H}^{s_0,p}(\mathbb{R}^n),\mathrm{H}^{s_1,p}(\mathbb{R}^n))_{\theta,q}= \mathrm{B}^{s}_{p,q}(\mathbb{R}^n)\text{, }&\qquad (\mathrm{B}^{s_0}_{p,q_0}(\mathbb{R}^n),\mathrm{B}^{s_1}_{p,q_1}(\mathbb{R}^n))_{\theta,q} = \mathrm{B}^{s}_{p,q}(\mathbb{R}^n)\text{, }
\end{align*}
whenever $(p,q_0,q_1,q)\in(1,+\infty)\times[1,+\infty]^3$($p\neq 1,+\infty$, when dealing with Sobolev (Bessel potential) spaces), $\theta\in(0,1)$, $s_0\neq s_1$ two real numbers, such that
\begin{align*}
    s:= (1-\theta)s_0+ \theta s_1\text{, }
\end{align*}
see \cite[Theorem~6.4.5]{BerghLofstrom1976}. A similar statement is available for homogeneous function spaces.
\begin{proposition}[ {\cite[Proposition~2.10]{Gaudin2022}} ]\label{prop:InterpHomSpacesRn}Let $(p,q,q_0,q_1)\in(1,+\infty)\times[1,+\infty]^3$, $s_0,s_1\in\mathbb{R}$, such that $s_0\neq s_1$, and set
\begin{align*}
    s:= (1-\theta)s_0+ \theta s_1\text{.}
\end{align*}
Assuming $(\mathcal{C}_{s_0,p})$ (resp. $(\mathcal{C}_{s_0,p,q_0})$), we get the following 
\begin{align}
    (\dot{\mathrm{H}}^{s_0,p}(\mathbb{R}^n),\dot{\mathrm{H}}^{s_1,p}(\mathbb{R}^n))_{\theta,q}=(\dot{\mathrm{B}}^{s_0}_{p,q_0}(\mathbb{R}^n),\dot{\mathrm{B}}^{s_1}_{p,q_1}(\mathbb{R}^n))_{\theta,q}=\dot{\mathrm{B}}^{s}_{p,q}(\mathbb{R}^n)\text{.}\label{eq:realInterpHomBspqRn}
\end{align}
If moreover, one consider $p_0,p_1\in(1,+\infty)$, and assume that $(\mathcal{C}_{s_0,p_0})$ and $(\mathcal{C}_{s_1,p_1})$ are true then also is $(\mathcal{C}_{s,p_\theta})$ and
\begin{align}
    [\dot{\mathrm{H}}^{s_0,p_0}(\mathbb{R}^n),\dot{\mathrm{H}}^{s_1,p_1}(\mathbb{R}^n)]_{\theta} = \dot{\mathrm{H}}^{s,p_\theta}(\mathbb{R}^n) \text{,}\label{eq:complexInterpHomHspRn}
\end{align}
and if $(\mathcal{C}_{s_0,p_0,q_0})$ and $(\mathcal{C}_{s_1,p_1,q_1})$ are satisfied then $(\mathcal{C}_{s,p_\theta,q_\theta})$ is also satisfied with
\begin{align}
    [\dot{\mathrm{B}}^{s_0}_{p_0,q_0}(\mathbb{R}^n),\dot{\mathrm{B}}^{s_1}_{p_1,q_1}(\mathbb{R}^n)]_{\theta} = \dot{\mathrm{B}}^{s}_{p_\theta,q_\theta}(\mathbb{R}^n)\text{,}\label{eq:complexInterpHomBspqRn}
\end{align}
where $\frac{1}{p_\theta}:=\frac{1-\theta}{p_0}+\frac{\theta}{p_1}$ and $q_\theta$ defined similarly.
\end{proposition}

\begin{proposition}[ {\cite[Proposition~2.15]{Gaudin2022}} ]\label{prop:SobolevMultiplier} For all $p\in(1,+\infty)$, $q\in[1,+\infty]$, for all $s\in (-1+\frac{1}{p},\frac{1}{p})$, for all $u\in\dot{\mathrm{H}}^{s,p}(\mathbb{R}^n)$ (resp. $\dot{\mathrm{B}}^{s}_{p,q}(\mathbb{R}^{n})$),
\begin{align*}
    \lVert \mathbbm{1}_{\mathbb{R}^n_+} u \rVert_{\dot{\mathrm{H}}^{s,p}(\mathbb{R}^{n})} \lesssim_{s,p,n} \lVert u \rVert_{\dot{\mathrm{H}}^{s,p}(\mathbb{R}^{n})} \text{ }\text{ (resp. }  \lVert \mathbbm{1}_{\mathbb{R}^n_+} u \rVert_{\dot{\mathrm{B}}^{s}_{p,q}(\mathbb{R}^{n})} \lesssim_{s,p,n} \lVert u \rVert_{\dot{\mathrm{B}}^{s}_{p,q}(\mathbb{R}^{n})} \text{ ). }
\end{align*}
The same results still holds with $({\mathrm{H}},{\mathrm{B}})$ instead of $(\dot{\mathrm{H}},\dot{\mathrm{B}})$.
\end{proposition}

\subsubsection{Homogeneous Sobolev and Besov spaces on \texorpdfstring{$\mathbb{R}^n_+$}{Rn+}}\label{sec:FunctionSpacesRn+}
Let $s\in\mathbb{R}$, $p\in(1,+\infty)$, $q\in[1,+\infty]$. Then for any $\mathrm{X}\in\{ \mathrm{B}^{s}_{p,q}, \dot{\mathrm{B}}^{s}_{p,q},  \mathrm{\mathrm{H}}^{s,p}, \dot{\mathrm{H}}^{s,p}\}$, and we define
\begin{align*}
    \mathrm{X}(\mathbb{R}^n_+) := \mathrm{X}(\mathbb{R}^n)_{|_{\mathbb{R}^n_+}}\text{, }
\end{align*}
with the usual quotient norm $\lVert u \rVert_{\mathrm{X}(\mathbb{R}^n_+)}:= \inf\limits_{\substack{\Tilde{u}\in \mathrm{X}(\mathbb{R}^n),\\ \tilde{u}_{|_{\mathbb{R}^n_+}}=u\,.}} \lVert \Tilde{u} \rVert_{\mathrm{X}(\mathbb{R}^n)}$. A direct consequence of the definition of those spaces is the density of $\eus{S}_0(\overline{\mathbb{R}^n_+})\subset\eus{S}(\overline{\mathbb{R}^n_+})$ in each of them, and also, and the completeness and reflexivity when their counterpart on $\mathbb{R}^n$ also are. We can also define,
\begin{align*}
    \mathrm{X}_0(\mathbb{R}^n_+) := \left\{\,u\in \mathrm{X}(\mathbb{R}^n) \,\Big{|}\, \supp u \subset \overline{\mathbb{R}^n_+} \right\}\text{, }
\end{align*}
with natural induced norm $\lVert  u \rVert_{\mathrm{X}_0(\mathbb{R}^n_+)}:= \lVert  u \rVert_{\mathrm{X}(\mathbb{R}^n)}$.

We have
\begin{itemize}
    \item \textbf{density results \cite[Proposition~2.24, Lemma~2.32]{Gaudin2022}:}\\
    for $p\in(1,+\infty)$, $q\in[1,+\infty)$, $s>-1+\frac{1}{p}$, when \eqref{AssumptionCompletenessExponents} is satisfied, 
\begin{align}
    \dot{\mathrm{H}}^{s,p}_0(\mathbb{R}^n_+)= \overline{\mathrm{C}_c^\infty(\mathbb{R}^n_+)}^{\lVert \cdot \rVert_{\dot{\mathrm{H}}^{s,p}(\mathbb{R}^n)}}\text{,}\quad\text{and}\quad \dot{\mathrm{B}}^{s}_{p,q,0}(\mathbb{R}^n_+)= \overline{\mathrm{C}_c^\infty(\mathbb{R}^n_+)}^{\lVert \cdot \rVert_{\dot{\mathrm{B}}^{s}_{p,q}(\mathbb{R}^n)}}\text{ ;}
\end{align}
    \item  \textbf{duality results, \cite[Propositions~2.27~\&~2.40]{Gaudin2022}:}\\
    for all $p\in(1,+\infty)$, $q\in(1,+\infty]$, $s>-1+\tfrac{1}{p}$, when \eqref{AssumptionCompletenessExponents} is satisfied,
\begin{align}
    (\dot{\mathrm{H}}^{s,p}(\mathbb{R}^n_+))' = &\dot{\mathrm{H}}^{-s,p'}_0(\mathbb{R}^n_+)\text{, }\, (\dot{\mathrm{B}}^{-s}_{p',q'}(\mathbb{R}^n_+))'=\dot{\mathrm{B}}^{s}_{p,q,0}(\mathbb{R}^n_+) \text{, }\\
    (\dot{\mathrm{H}}^{s,p}_0(\mathbb{R}^n_+))' = &\dot{\mathrm{H}}^{-s,p'}(\mathbb{R}^n_+)\text{, }\,(\dot{\mathrm{B}}^{-s}_{p',q',0}(\mathbb{R}^n_+))'=\dot{\mathrm{B}}^{s}_{p,q}(\mathbb{R}^n_+)\text{. }
\end{align}
    
    \item \textbf{interpolation results, \cite[Propositions~2.33~\&~2.39]{Gaudin2022}:} \\
    if $(\dot{\mathfrak{h}},\dot{\mathfrak{b}})\in\{(\dot{\mathrm{H}},\dot{\mathrm{B}}), (\dot{\mathrm{H}}_0,\dot{\mathrm{B}}_{\cdot,\cdot,0})\}$, with $(p,q,q_0,q_1)\in(1,+\infty)\times[1,+\infty]^3$ ($p,p_j\neq 1,+\infty$ is assumed, when dealing with Sobolev (Riesz potential) spaces), $\theta\in(0,1)$, $s_j,s>-1+1/p_j$, $j\in\{0,1\}$, with $s>-1+1/p$, where $s_0,s_1,s$ are three real numbers, with
\begin{align*}
    s= (1-\theta)s_0+ \theta s_1\text{, }
\end{align*}
such that \eqref{AssumptionCompletenessExponents} is satisfied. Then, one has
\begin{align}
    (\dot{\mathfrak{h}}^{s_0,p}(\mathbb{R}^n_+),\dot{\mathfrak{h}}^{s_1,p}(\mathbb{R}^n_+))_{\theta,q}= \dot{\mathfrak{b}}^{s}_{p,q}(\mathbb{R}^n_+)\text{, }&\qquad(\dot{\mathfrak{b}}^{s_0}_{p,q_0}(\mathbb{R}^n_+),\dot{\mathfrak{b}}^{s_1}_{p,q_1}(\mathbb{R}^n_+))_{\theta,q} = \dot{\mathfrak{b}}^{s}_{p,q}(\mathbb{R}^n_+)\text{.}
\end{align}
\end{itemize}

Note that, due to Proposition \ref{prop:SobolevMultiplier}, we have for free the  following equalities of homogeneous Sobolev and Besov spaces, with equivalent norms, for all $p\in(1,+\infty)$, $q\in[1,+\infty]$, $s\in(-1+\frac{1}{p},\frac{1}{p})$,
\begin{align}
    \dot{\mathrm{H}}^{s,p}(\mathbb{R}^n_+) = \dot{\mathrm{H}}^{s,p}_0(\mathbb{R}^n_+)\text{, }\, \dot{\mathrm{B}}^{s}_{p,q}(\mathbb{R}^n_+)=\dot{\mathrm{B}}^{s}_{p,q,0}(\mathbb{R}^n_+) \text{. }
\end{align}

\subsubsection{Operators on Sobolev and Besov spaces}
We introduce domains for an operator $A$ acting on Sobolev or Besov spaces, denoting
\begin{itemize}
    \item $\mathrm{D}^{s}_{p}(A)$ (resp. $\dot{\mathrm{D}}^{s}_{p}(A)$)  its domain on $\mathrm{H}^{s,p}$ (resp. $\dot{\mathrm{H}}^{s,p}$);
    \item $\mathrm{D}^{s}_{p,q}(A)$ (resp. $\dot{\mathrm{D}}^{s}_{p,q}(A)$)  its domain on $\mathrm{B}^{s}_{p,q}$ (resp. $\dot{\mathrm{B}}^{s}_{p,q}$);
    \item $\mathrm{D}_{p}(A) = \mathrm{D}^{0}_{p}(A) = \dot{\mathrm{D}}^{0}_{p}(A)$ its domain on $\mathrm{L}^p$.
\end{itemize}

Similarly, $\mathrm{N}^{s}_{p}(A)$, $\mathrm{N}^{s}_{p,q}(A)$ will stand for its nullspace on $\mathrm{H}^{s,p}$ and $\mathrm{B}^{s}_{p,q}$, and range spaces will be given respectively by $\mathrm{R}^{s}_{p}(A)$ and $\mathrm{R}^{s}_{p,q}(A)$. We replace $\mathrm{N}$ and $\mathrm{R}$ by $\dot{\mathrm{N}}$ and $\dot{\mathrm{R}}$ for their corresponding corresponding sets on homogeneous function spaces. 

If the operator $A$ has different realizations depending on various function spaces and on the considered open set, we may write its domain $\mathrm{D}(A,\Omega)$, and similarly for its nullspace $\mathrm{N}$ and range space $\mathrm{R}$. We omit the open set $\Omega$ if there is no possible confusion.


\section{Hodge Laplacians, Hodge decomposition and Hodge-Stokes operators}

This section is dedicated to the study of Hodge Laplacians, the Hodge decomposition and Hodge-Stokes operators, on Sobolev and Besov spaces on $\mathbb{R}^n$ and $\mathbb{R}^n_+$.

We will first introduce here the formalism of differential forms in the Euclidean setting. Resolvent estimates for the Hodge Laplacian and Hodge-Stokes like operators on the whole space will follow from standard Fourier and Harmonic analysis, from which we will deduce the related Hodge decomposition on $\mathbb{R}^n$ as well as the boundedness of holomorphic functional calculus for each operator.

Secondly, we are going to give all corresponding similar results for the Hodge Laplacians, the Hodge decomposition and Hodge-Stokes operators on the half-space $\mathbb{R}^n_+$. Those results are going to be built from what happens on the whole space $\mathbb{R}^n$, mimicking the behavior of Dirichlet and Neumann Laplacians on the half-space, see \cite[Section~4]{Gaudin2022}.

\subsection{Differential forms on Euclidean space, and corresponding function spaces}

Here $\Omega$ stands for a domain of $\mathbb{R}^n$ with at least, if not empty, Lipschitz boundary. The open set $\Omega$ will be specified later on to be either, the whole space $\mathbb{R}^n$ or the half-space $\mathbb{R}^n_+$. Recall briefly that $\partial\mathbb{R}^n = \emptyset$, and $\partial\mathbb{R}^n_+ = \mathbb{R}^{n-1}\times \{0\}$. We also recall that the outer normal unit at $\partial\mathbb{R}^n_+$ is ${\nu} = -\mathfrak{e}_n$, where $(\mathfrak{e}_k)_{k\in\llb 1, n\rrb}$ is the canonical basis of $\mathbb{R}^n$, identified with its dual basis denoted by $(\mathrm{d}x_{k})_{k\in\llb 1,n\rrb}$, where $\mathrm{d}x_{k}(\mathfrak{e}_j) = \mathbbm{1}_{\{k\}}(j)$, $(k,j)\in\llb 1,n\rrb^2$.

Following \cite{McintoshMonniaux2018,Monniaux2021}, we introduce the \textit{exterior derivative} $\mathrm{d}:=\nabla \wedge=\sum_{k=1}^{n} \partial_{x_k} \mathfrak{e}_{k} \wedge$ and the \textit{interior derivative} (or \textit{coderivative}) $\left.{\delta}:=-\nabla\lrcorner=-\sum_{k=1}^{n} \partial_{x_k} \mathfrak{e}_{k}\right\lrcorner$ acting on differential forms on a domain $\Omega \subset \mathbb{R}^{n}$, \textit{i.e.} acting on functions defined on $\Omega$ which take value the complexified exterior algebra $\Lambda=\Lambda^{0} \oplus \Lambda^{1} \oplus \cdots \oplus \Lambda^{n}$ of $\mathbb{R}^{n}$. We allow us a slight abuse of notation: here we will not distinguish vectors of $\mathbb{R}^n$, vector fields, and $1$-differential forms.

We also recall that for $k\in \llb 0,n\rrb$, $u \in \Lambda^k$ can be uniquely determined by $(u_{I})_{I\in\mathcal{I}^k_n}\in\mathbb{C}^{\binom{n}{k}}$ such that
\begin{align*}
    u = \sum_{I\in\mathcal{I}^k_n} u_{I} \,\mathrm{d}x_{I} \text{, }
\end{align*}
where $\mathcal{I}^k_n=\{ (\ell_j)_{j\in\llb 1,k\rrb}\in\llb 1,n\rrb^k\,|\, \ell_j< \ell_{j+1}\}$, with $\lvert\mathcal{I}^k_n\rvert = {\binom{n}{k}}$, and $u_{I}$ and $\mathrm{~d}x_{I}$ stands respectively for $u_{\ell_1 \ell_2\ldots \ell_k}$ and $\mathrm{~d}x_{\ell_1}\wedge\mathrm{~d}x_{\ell_2}\wedge\ldots\wedge\mathrm{~d}x_{\ell_k}$ whenever $I=(\ell_j)_{j\in\llb 1,k\rrb}$.

One may also notice that such representation of $k$-differential forms with increasing index is possible due to symmetry properties (\textit{i.e.} $\mathrm{~d}x_{\ell}\wedge\mathrm{~d}x_{k} = - \mathrm{~d}x_{k}\wedge\mathrm{~d}x_{\ell}$ for all $k,\ell\in\llb 1,n\rrb$).

In particular, remark that $\Lambda^{0}\simeq \mathbb{C}$, the space of complex scalars, and more generally $\Lambda^k\simeq \mathbb{C}^{{\binom{n}{k}}}$, so that $\Lambda \simeq \mathbb{C}^{2^n}$. We also set $\Lambda^{\ell}=\{0\}$ if $\ell<0$ or $\ell>n$.

On the exterior algebra $\Lambda$, the basic operations are
\begin{enumerate}[label=($\roman*$)]
    \item the exterior product $\wedge: \Lambda^{k} \times \Lambda^{\ell} \rightarrow \Lambda^{k+\ell}$,
    \item the interior product $\lrcorner: \Lambda^{k} \times \Lambda^{\ell} \rightarrow \Lambda^{\ell-k}$,
    \item the Hodge star operator $\star: \Lambda^{\ell} \rightarrow \Lambda^{n-\ell}$,
    \item the inner product $\langle\cdot, \cdot\rangle: \Lambda^{\ell} \times \Lambda^{\ell} \rightarrow \mathbb{C}$.
\end{enumerate}
If $a \in \Lambda^{1}, u \in \Lambda^{\ell}$ and $v \in \Lambda^{\ell+1}$, then
\begin{align*}
\langle a \wedge u, v\rangle=\langle u, a\lrcorner v\rangle.
\end{align*}
For more details, we refer to, e.g., \cite[Section~2]{AxelssonMcIntosh2004} and  \cite[Section~3]{CostabelMcIntosh2010}, noting that both these papers contain some historical background (and being careful that ${\delta}$ has the opposite sign in \cite{AxelssonMcIntosh2004}). One may also consult \cite{DoCarmo1994} for an introduction from the euclidean setting point of view, and \cite[Section~1-3]{Jost2011} for basic and usual properties in the more general Riemannian setting \footnote{Notice that the Riemannian setting presented by Jost deals with compact manifold but a lot of computations remain true in their full generality, due to local behavior of each operation (Hodge star operator, exterior and interior products etc.)}. We recall the relation between $\mathrm{d}$ and ${\delta}$ via the Hodge star operator:
\begin{align*}
    \star {\delta} u=(-1)^{\ell} \mathrm{d}(\star u) \quad\text{ and }\quad \star \mathrm{d} u=(-1)^{\ell-1} {\delta}(\star u) \quad\text{ for an }\ell\text{-form }u.
\end{align*}

In dimension $n=3$, this gives (see \cite{CostabelMcIntosh2010}) for a vector $a \in \mathbb{R}^{3}$ identified with a $1$-form
\begin{itemize}
    \item   $u$ scalar, interpreted as 0 -form: $a \wedge u=u a$, $a\lrcorner u=0$;
    \item   $u$ scalar, interpreted as 3 -form: $a \wedge u=0$, $a\lrcorner u=u a$;
    \item   $u$ vector, interpreted as 1 -form: $a \wedge u=a \times u$, $a\lrcorner u=a \cdot u$;
    \item   $u$ vector, interpreted as 2 -form: $a \wedge u=a \cdot u$, $a\lrcorner u=-a \times u$.
\end{itemize}

From now and until the end of the present paper, if $p\in(1,+\infty)$, $q\in[1,+\infty]$, $s\in\mathbb{R}$, $k\in\llb 0,n \rrb$ and $\mathrm{X}^s\in\{\mathrm{\mathrm{H}}^{s,p},\, \dot{\mathrm{H}}^{s,p},\, \mathrm{B}^{s}_{p,q},\, \dot{\mathrm{B}}^{s}_{p,q}\}$, then $\mathrm{X}^s(\Omega,\Lambda^k)$ stands for $k$-differential forms whose coefficients lies in $\mathrm{X}^s(\Omega)$, \textit{i.e.} $\mathrm{X}^s(\Omega,\Lambda^k)\simeq \mathrm{X}^s(\Omega,\mathbb{C}^{\binom{n}{k}})$. One may also consider similarly $\mathrm{X}^s_0(\Omega,\Lambda^k)$.

Operators $\mathrm{d}$ and ${\delta}$ are differential operators such that $\mathrm{d}^2 = \mathrm{d} \circ \mathrm{d} = 0$ and ${\delta}^2 = {\delta} \circ {\delta} = 0$, and each of them are bounded seen as linear operators $\mathrm{X}^s(\Omega,\Lambda)\longrightarrow \mathrm{X}^{s-1}(\Omega,\Lambda)$.

We recall the following integration by parts formula, for all $u,v\in \eus{S}({\overline{\Omega}},\Lambda)$,
\begin{align}
    \int_{\Omega} \langle \mathrm{d}u, v\rangle = \int_{\Omega} \langle u, {\delta} v\rangle + \int_{\partial\Omega} \langle u, \nu\lrcorner v\rangle \mathrm{~d}\sigma\text{, }\label{eq:IntbyParts1}\\
    \int_{\Omega} \langle {\delta} u, v\rangle = \int_{\Omega} \langle  u, \mathrm{d} v\rangle + \int_{\partial\Omega} \langle u, \nu\wedge v\rangle \mathrm{~d}\sigma\text{, }\label{eq:IntbyParts2}
\end{align}
which are true since we are in the Euclidean setting and where $\nu$ is the outward normal identified as a $1$-form. More generally, for all $T\in \eus{D}'(\Omega,\Lambda^k)\simeq \eus{D}'(\Omega,\mathbb{C}^{\binom{n}{k}})$, $k\in\llb 0,n\rrb$, we define
    \begin{align*}
        \big\langle \mathrm{d}T, \phi \big\rangle_{\Omega} &:= \big\langle T, {\delta}\phi \big\rangle_{\Omega} \text{ for all }\phi \in \mathrm{C}_c^\infty(\Omega,\Lambda^{k+1})\text{, }\\
        \big\langle {\delta}T, \psi \big\rangle_{\Omega} &:= \big\langle T,\mathrm{d}\psi \big\rangle_{\Omega} \text{ for all }\psi \in \mathrm{C}_c^\infty(\Omega,\Lambda^{k-1})\text{. }
\end{align*}

In particular, one may see those operators as unbounded ones and introduce their respective domains on $\mathrm{L}^p(\Omega, \Lambda^{k})$, $k\in\llb 0,n\rrb$, denoted by $\mathrm{D}_p(\mathrm{d},\Lambda^k)$ and $\mathrm{D}_p({\delta},\Lambda^k)$ defined as
\begin{center}
    \resizebox{\textwidth}{!}{
$\mathrm{D}_{p}(\mathrm{d},\Lambda^k):=\left\{u \in \mathrm{L}^p(\Omega,\Lambda^k) \,\left|\, \mathrm{d} u \in \mathrm{L}^p(\Omega,\Lambda^{k+1})\right.\right\} \text { and } \mathrm{D}_{p}({\delta},\Lambda^k):=\left\{u \in \mathrm{L}^p(\Omega, \Lambda^k)\,\left|\, {\delta} u \in \mathrm{L}^p(\Omega, \Lambda^{k-1})\right.\right\}\text{. }$}
\end{center}
We can introduce their corresponding counterparts on homogeneous Sobolev spaces scales, $\dot{\mathrm{D}}^{s}_{p}(\mathrm{d},\Lambda^k)$ on $\dot{\mathrm{H}}^{s,p}$, the same goes for inhomogeneous Sobolev spaces ${\mathrm{D}}^{s}_{p}(\mathrm{d},\Lambda^k)$ on ${\mathrm{H}}^{s,p}$. The same goes with the interior derivative ${\delta}$ instead of $\mathrm{d}$. One may proceed in a similar fashion considering their domains on inhomogeneous and homogeneous Besov spaces.

As an exact sequence of (densely defined but not necessarily closed) unbounded operators, we get:
\begin{center}
    \resizebox{\textwidth}{!}{
$\begin{array}{ccccccccccccccc}
     \mathrm{d}&:& \mathrm{X}^s(\Omega, \Lambda^0) & \overset{}{\longrightarrow} & \mathrm{X}^s(\Omega,\Lambda^1) &\overset{}{\longrightarrow} & \mathrm{X}^s(\Omega,\Lambda^2) & \overset{}{\longrightarrow} & \ldots &\overset{}{\longrightarrow} & \mathrm{X}^s(\Omega,\Lambda^{n-1})& \overset{}{\longrightarrow} & \mathrm{X}^s(\Omega,\Lambda^{n}) & {\longrightarrow} & 0\\
     0 & \longleftarrow & \mathrm{X}^s(\Omega, \Lambda^0) & \overset{}{\longleftarrow} & \mathrm{X}^s(\Omega,\Lambda^1) &\overset{}{\longleftarrow} & \mathrm{X}^s(\Omega,\Lambda^2) & \overset{}{\longleftarrow} & \ldots &\overset{}{\longleftarrow} & \mathrm{X}^s(\Omega,\Lambda^{n-1})& \overset{}{\longleftarrow} & \mathrm{X}^s(\Omega,\Lambda^{n}) & : &{\delta}\text{. }
\end{array}$}
\end{center}
In dimension $n=3$, one can specialized, by the mean of the identification $\Lambda^k\simeq \mathbb{C}^{\binom{3}{k}}$, as
\begin{align*}
\begin{array}{ccccccccccc}
     \mathrm{d}&:& \mathrm{X}^s(\Omega, \mathbb{C}) & \overset{\nabla}{\longrightarrow} & \mathrm{X}^s(\Omega,\mathbb{C}^3) &\overset{\curl}{\longrightarrow} & \mathrm{X}^s(\Omega,\mathbb{C}^3) & \overset{\div}{\longrightarrow} & \mathrm{X}^s(\Omega,\mathbb{C}) & {\longrightarrow} & 0\\
     0 & \longleftarrow & \mathrm{X}^s(\Omega, \mathbb{C}) & \overset{-\div}{\longleftarrow} & \mathrm{X}^s(\Omega,\mathbb{C}^3) &\overset{\curl}{\longleftarrow} & \mathrm{X}^s(\Omega,\mathbb{C}^3)& \overset{-\nabla}{\longleftarrow} & \mathrm{X}^s(\Omega,\mathbb{C}) & : &{\delta}\text{. }
\end{array}
\end{align*}

Thus for arbitrary dimension $n$, the operator $\mathrm{d}$ restricted to its action on $\mathrm{X}^s(\Omega,\Lambda^1)$, with value in $\mathrm{X}^{s-1}(\Omega,\Lambda^2)$, and $\delta$ restricted to its action on $\mathrm{X}^s(\Omega,\Lambda^{n-1})$, with value in $\mathrm{X}^{s-1}(\Omega,\Lambda^{n-2})$, are fair consistent generalizations of the $\curl$ operator on $\mathbb{R}^3$. Since in dimension $n$ higher than $4$, $n-1\neq 2$, we also have to distinguish their dual operators: the operator $\mathrm{d}$ restricted to its action on $\mathrm{X}^s(\Omega,\Lambda^{n-2})$ and the operator $\delta$ restricted to its action on $\mathrm{X}^s(\Omega,\Lambda^{2})$ which are fair consistent generalizations of the dual operator $\prescript{t}{}{\curl}$ (usually fully identified with the $\curl$) on $\mathbb{R}^3$.

We can use \eqref{eq:IntbyParts1} and \eqref{eq:IntbyParts2} to consider adjoints of $\mathrm{d}$ and ${\delta}$ in the sense of maximal adjoint operators in the Hilbert space $\mathrm{L}^2(\Omega,\Lambda)$, so that we will see later, e.g. Lemma \ref{lem:closdenessNdualityDerivativesL2Rn+}, that they have the following exact description of their domains
\begin{align*}
    \mathrm{D}_{2}(\mathrm{d}^\ast,\Lambda^k) = \{\, u\in \mathrm{D}_{2}({\delta},\Lambda^k)\, \,|\, \nu\lrcorner u_{|_{\partial\Omega}} =0\, \} \quad\text{ and }\quad
    \mathrm{D}_{2}({\delta}^\ast,\Lambda^k) = \{\, u\in \mathrm{D}_{2}(\mathrm{d},\Lambda^k)\, \,|\, \nu\wedge u_{|_{\partial\Omega}} =0\, \}\text{. }
\end{align*}
One can also see those adjoints operators through the following $\mathrm{L}^2$-closures of unbounded operators,
\begin{align*}
    (\mathrm{D}_{2}(\mathrm{d}^\ast,\Lambda^k),\mathrm{d}^\ast) = \overline{(\mathrm{C}_c^\infty(\Omega,\Lambda^k),{\delta})}\quad\text{ and }\quad
    (\mathrm{D}_{2}({\delta}^\ast,\Lambda^k),{\delta}^\ast) = \overline{(\mathrm{C}_c^\infty(\Omega,\Lambda^k),\mathrm{d})} \text{. }
\end{align*}

\begin{definition} \begin{enumerate}[label=($\roman*$)]
    \item The \textbf{Hodge-Dirac operator} on $\Omega$ with \textbf{normal boundary conditions} is defined as
    \begin{align*}
        D_\mathfrak{n} := \delta^\ast + \delta \text{.}
    \end{align*}
    Its square denoted by $-\Delta_{\mathcal{H},\mathfrak{n}}:= D_\mathfrak{n}^2 = \delta^\ast\delta + \delta\delta^\ast$, is called the (negative) \textbf{Hodge Laplacian} with \textbf{relative boundary conditions} (also called \textbf{generalised Dirichlet boundary conditions})
    \begin{align*}
        \nu\wedge u _{|_{\partial \Omega}} = 0 \text{, and } \nu\wedge \delta u _{|_{\partial \Omega}} = 0\text{. }
    \end{align*}
    The restriction to scalar functions $u\,:\,\Omega\longrightarrow\Lambda^0$ gives $-\Delta_{\mathcal{H},\mathfrak{n}}u = \delta\delta^\ast u = -\Delta_{\mathcal{D}}u$, where $-\Delta_{\mathcal{D}}$ is the Dirichlet Laplacian.
    \item The \textbf{Hodge-Dirac operator} on $\Omega$ with \textbf{tangential boundary conditions} is defined as
    \begin{align*}
        D_\mathfrak{t} := \mathrm{d}+\mathrm{d}^\ast \text{.}
    \end{align*}
    Its square denoted by $-\Delta_{\mathcal{H},\mathfrak{t}}:= D_\mathfrak{t}^2 = \mathrm{d}\mathrm{d}^\ast + \mathrm{d}^\ast\mathrm{d}$, is called the (negative) \textbf{Hodge Laplacian} with \textbf{absolute boundary conditions} (also called \textbf{generalised Neumann boundary conditions})
    \begin{align*}
        \nu\iprod u _{|_{\partial \Omega}} = 0 \text{, and } \nu\iprod \mathrm{d} u _{|_{\partial \Omega}} = 0\text{. }
    \end{align*}
    The restriction to scalar functions $u\,:\,\Omega\longrightarrow\Lambda^0$ gives $-\Delta_{\mathcal{H},\mathfrak{t}}u = \mathrm{d}^\ast\mathrm{d} u = -\Delta_{\mathcal{N}}u$,  where $-\Delta_{\mathcal{N}}$ is the Neumann Laplacian.
\end{enumerate}
\end{definition}

\begin{notations}When it does not matter $(\mathfrak{d}, D_{\cdot},-\Delta_\mathcal{H})$ will stand either for $(\delta, D_\mathfrak{n},-\Delta_{\mathcal{H},\mathfrak{n}})$ or $(\mathrm{d}, D_\mathfrak{t},-\Delta_{\mathcal{H},\mathfrak{t}})$, just writing
\begin{align*}
        -\Delta_\mathcal{H} = D_{\cdot}^2 = \mathfrak{d}\mathfrak{d}^\ast+\mathfrak{d}^\ast\mathfrak{d}\text{. }
\end{align*}
\end{notations}

\begin{remark}\label{rmk:HodgeLaplacianBC}Let's make two independent remarks:
\begin{itemize}
    \item We recall here that, if $\Omega\subset \mathbb{R}^3$ is an open set with, say at least, Lipschitz boundary, one has formally for $u$ with value in $\Lambda^1\simeq \mathbb{C}^3$ or $\Lambda^2\simeq \mathbb{C}^3$,
\begin{align*}
    -\Delta_{\mathcal{H},\mathfrak{n}}u =-\Delta_{\mathcal{H},\mathfrak{t}} u = \curl \curl u - \nabla \div u
\end{align*}
with either one of the following couple of boundary conditions
\begin{align*}
    [ u\cdot\nu_{|_{\partial\Omega}} = 0 \text{, } \nu\times\curl u_{|_{\partial\Omega}} = 0]&\text{ or }[ u\times\nu_{|_{\partial\Omega}} = 0 \text{, } (\div u)\nu_{|_{\partial\Omega}} = 0] \text{. }
\end{align*}
    \item In the case of $\Omega=\mathbb{R}^n$, notice that no boundary value comes in, hence $\mathrm{d}^\ast = \delta$, $\delta^\ast=\mathrm{d}$, so that
\begin{align*}
    D_\cdot=D_\mathfrak{n}=D_\mathfrak{t} = (\mathrm{d}+\delta)\qquad\text{ and }\qquad -\Delta_{\mathcal{H}}=-\Delta_{\mathcal{H},\mathfrak{n}}=-\Delta_{\mathcal{H},\mathfrak{t}}= (\mathrm{d}+\delta)^2= \mathrm{d}\delta+\delta\mathrm{d}\text{. }
\end{align*}
\end{itemize}
\end{remark}

\begin{definition} \begin{enumerate}[label=($\roman*$)]
    \item The (bounded) orthogonal projector defined on $\mathrm{L}^2(\Omega,\Lambda^k)$ onto $\mathrm{N}_2(\mathrm{d}^\ast,\Lambda^k)$ is denoted by $\mathbb{P}$ and called the \textbf{generalized Helmholtz-Leray} (or \textbf{Leray}) \textbf{projector}.
    \item The (bounded) orthogonal projector defined on $\mathrm{L}^2(\Omega,\Lambda^k)$ onto $\mathrm{N}_2(\delta,\Lambda^k)$ is denoted by $\mathbb{Q}$.
    \item For $p\in(1,+\infty)$, $s\in\mathbb{R}$ and $k\in\llb 0,n\rrb$, we say that $\mathrm{H}^{s,p}(\Omega,\Lambda^k)$ admits a \textbf{Hodge decomposition} if $(\mathrm{D}_p^s(\mathrm{d},\Lambda^k),\mathrm{d})$, $(\mathrm{D}_p^s(\delta,\Lambda^k),\delta)$ and their respective adjoints are closable and
    \begin{align*}\tag{$\mathfrak{H}_p^s$}\label{HodgeDecompHsp}
       \mathrm{H}^{s,p}(\Omega,\Lambda^k) &= \mathrm{N}_p^s(\mathfrak{d},\Lambda^k)\oplus \overline{\mathrm{R}_p^s(\mathfrak{d}^\ast,\Lambda^k)}\text{, }\\
        &=   \overline{\mathrm{R}_p^s(\mathfrak{d},\Lambda^k)} \oplus {\mathrm{N}_p^s(\mathfrak{d}^\ast,\Lambda^k)} \text{, }
\end{align*}
holds in the topological sense. We keep the same definition of the Hodge decomposition on other function spaces replacing $({\mathrm{H}}^{s,p},{\mathrm{D}}_p^s,{\mathrm{R}}_p^s,{\mathrm{N}}_p^s)$ by either $(\dot{\mathrm{H}}^{s,p},\dot{\mathrm{D}}_p^s,\dot{\mathrm{R}}_p^s,\dot{\mathrm{N}}_p^s)$, $({\mathrm{B}}^{s}_{p,q},{\mathrm{D}}_{p,q}^s,{\mathrm{R}}_{p,q}^s,{\mathrm{N}}_{p,q}^s)$, or even by $(\dot{\mathrm{B}}^{s}_{p,q},\dot{\mathrm{D}}_{p,q}^s,\dot{\mathrm{R}}_{p,q}^s,\dot{\mathrm{N}}_{p,q}^s)$, with $q\in[1,+\infty]$.
\end{enumerate}
\end{definition}
One can notice that in the case of vector fields (identified with $1$-forms, \textit{i.e.} $\mathrm{L}^2(\Omega,\Lambda^1)\simeq \mathrm{L}^2(\Omega,\mathbb{C}^n)$), we can identify $\mathbb{P}$ as the usual Helmholtz-Leray projector on divergence free vector fields with null tangential trace at the boundary
\begin{align*}
    \mathbb{P}\,:\, \mathrm{L}^2(\Omega,\mathbb{C}^n) \longrightarrow \mathrm{L}^2_{\sigma}(\Omega)=\{\, u\in \mathrm{L}^2(\Omega,\mathbb{C}^n)\,{|}\, \div u = 0\text{, } u\cdot\nu_{|_{\partial\Omega}} =0\,\}\text{. }
\end{align*}
It gives the following classical orthogonal, topological, Hodge decomposition, see \cite[Chapter~2,~Section~2.5]{SohrBook2001},
\begin{align*}
    \mathrm{L}^2(\Omega,\mathbb{C}^n)= \mathrm{L}^2_{\sigma}(\Omega)\stackrel{\perp}{\oplus}\nabla \dot{\mathrm{H}}^1(\Omega,\mathbb{C})\text{, }
\end{align*}
for any sufficiently nice domains $\Omega$, say for instance with uniformly Lipschitz boundary, see \cite[Lemma~2.5.3]{SohrBook2001}.

Before investigating the Hodge decomposition and the functional analytic properties of the Hodge Laplacian on differential forms on function spaces in $\mathbb{R}^n_+$, we want to know a bit more about the whole space case. In next subsection devoted to the whole space, we gather well known facts and results which lack explicit references in the literature to the best of author's knowledge.

\subsection{The case of the whole space}

On the whole space $\mathbb{R}^n$ the action of the Laplacian and the Hodge decomposition for vector fields is well known in the literature on usual spaces as Lebesgue spaces $\mathrm{L}^p(\mathbb{R}^n,\mathbb{C}^n)$, $p\in(1,+\infty)$, and so is the case of inhomogeneous and homogeneous Sobolev and Besov spaces. Our main goal here is to extend and summarize those results with the formalism of differential forms.

To do so, we introduce an extension of the Fourier transform to differential forms whose coefficients lies in the space of complex valued Schwartz functions $\eus{S}(\mathbb{R}^n,\mathbb{C})$, or in the space of tempered distribution $\eus{S}'(\mathbb{R}^n,\mathbb{C})$.
\begin{itemize}[label=$\bullet$]
    \item For all $u\in \mathrm{L}^1(\mathbb{R}^n,\Lambda^k)\simeq \mathrm{L}^1(\mathbb{R}^n,\mathbb{C}^{\binom{n}{k}})$, $k\in\llb 0,n\rrb$, we define
    \begin{align*}
        \eus{F}u:=\sum\limits_{I\in\mathcal{I}^k_n} \eus{F}u_{I} \,\mathrm{d}\xi_I \in \mathrm{C}^0_{0}(\mathbb{R}^n,\Lambda^k)\text{. }
    \end{align*}
    Hence, as in the scalar valued case, the Fourier transform $\eus{F}$ induces a topological automorphism of $\eus{S}(\mathbb{R}^n,\Lambda^k)\simeq \eus{S}(\mathbb{R}^n,\mathbb{C}^{\binom{n}{k}})$.
    \item For $k\in\llb 0,n\rrb$, we write $\eus{S}'(\mathbb{R}^n,\Lambda^k):= (\eus{S}(\mathbb{R}^n,\Lambda^k))'\simeq \eus{S}'(\mathbb{R}^n,\mathbb{C}^{\binom{n}{k}})$. Similarly, the Fourier transform $\eus{F}$ is an automorphism of $\eus{S}'(\mathbb{R}^n,\Lambda^k)$.
    \item For all $T\in \eus{S}'(\mathbb{R}^n,\Lambda^k)$, $k\in\llb 0,n\rrb$, we define
    \begin{align*}
        \big\langle \mathrm{d}T, \phi \big\rangle_{\mathbb{R}^n} &:= \big\langle T, {\delta}\phi \big\rangle_{\mathbb{R}^n} \text{ for all }\phi \in \eus{S}(\mathbb{R}^{n},\Lambda^{k+1})\text{, }\\
        \big\langle {\delta}T, \psi \big\rangle_{\mathbb{R}^n} &:= \big\langle T,\mathrm{d}\psi \big\rangle_{\mathbb{R}^n} \text{ for all }\psi \in \eus{S}(\mathbb{R}^{n},\Lambda^{k-1})\text{. }
    \end{align*}
\end{itemize}

The following lemma is straightforward and fundamental for our analysis.

\begin{lemma}\label{lem:FundamentalLemmaDiffFormsRn}For all $u\in\eus{S}(\mathbb{R}^n,\Lambda^k)$, $k\in\llb 0,n\rrb$, for all $\xi\in\mathbb{R}^n$,
\begin{align*}
        \eus{F}[\mathrm{d}u](\xi)=i\xi\wedge \eus{F}u(\xi) \quad\text{ and }\quad \eus{F}[\delta u](\xi)=-i\xi\iprod \eus{F}u(\xi) \text{.}
\end{align*}
\end{lemma}

\begin{remark}This is somewhat consistent, when $n=3$, with formulas like
\begin{align*}
    \eus{F}[\curl u](\xi) = i\xi\times\eus{F}u(\xi)\,\text{ and }\,\eus{F}[\div u](\xi) = i\xi\cdot\eus{F}u(\xi)\text{, } u\in\eus{S}(\mathbb{R}^3,\mathbb{C}^3)\text{, } \xi \in\mathbb{R}^3\text{. }
\end{align*}
\end{remark}

From there, the identity for all differential forms $u$ of degree $k$, and all vector $v\in\mathbb{R}^n$, 
\begin{align*}
    v\wedge(v\iprod u) + v\iprod(v\wedge u) = \lvert v\rvert^2 \, u\text{, }
\end{align*}
with the use of Lemma \ref{lem:FundamentalLemmaDiffFormsRn} yields with Remark \ref{rmk:HodgeLaplacianBC} that, for all $u\in\eus{S}'(\mathbb{R}^n,\Lambda^k)$, $k\in\llb 0,n\rrb$,
\begin{align*}
    \eus{F}[-\Delta_{\mathcal{H}} u](\xi)=\eus{F}[(\delta\mathrm{d} + \mathrm{d}\delta) u](\xi) = i\xi\wedge(-i\xi\iprod u) + -i\xi\iprod(i\xi\wedge u) = \lvert \xi\rvert^2 \cdot \eus{F}u(\xi) = \eus{F}[-\Delta u](\xi)\text{. }
\end{align*}
Hence, the Hodge Laplacian on the whole space is nothing but the scalar Laplacian applied separately to each component of a differential form so that its properties are carried over by the scalar Laplacian. We state then a very well known result adapted to our setting.

\begin{theorem}\label{thm:MetaThm1LaplacianRn}Let $p,\Tilde{p}\in(1,+\infty)$, $q,\Tilde{q}\in[1,+\infty]$, $s,\alpha\in\mathbb{R}$, and $k\in\llb 0,n\rrb$. The Hodge Laplacian is an injective operator on $\eus{S}'_h(\mathbb{R}^n,\Lambda^k)$, and satisfies the following properties
\begin{enumerate}[label=($\roman*$)]
    \item For $f\in \eus{S}'_h(\mathbb{R}^n,\Lambda^k)$, consider the problem
    \begin{align*}
         - \Delta u = f\quad\text{ in } \mathbb{R}^n\text{.}
    \end{align*}
        \begin{enumerate}
        \item If $f\in [\dot{\mathrm{H}}^{s,p}\cap \dot{\mathrm{H}}^{\alpha,\tilde{p}}](\mathbb{R}^n,\Lambda^k)$, and $(\mathcal{C}_{s+2,p})$ is satisfied, then there exists a unique solution $u\in [\dot{\mathrm{H}}^{s+2,p}\cap \dot{\mathrm{H}}^{\alpha+2,\tilde{p}}](\mathbb{R}^n,\Lambda^k)$ with estimates,
        \begin{align*}
            \lVert \mathrm{d}\delta u\rVert_{\dot{\mathrm{H}}^{s,p}(\mathbb{R}^n)} + \lVert \delta\mathrm{d} u\rVert_{\dot{\mathrm{H}}^{s,p}(\mathbb{R}^n)} \lesssim_{p,n,s} \lVert \nabla^2 u\rVert_{\dot{\mathrm{H}}^{s,p}(\mathbb{R}^n)} &\lesssim_{p,n,s} \lVert f\rVert_{\dot{\mathrm{H}}^{s,p}(\mathbb{R}^n)}\text{, }\\
            \lVert \mathrm{d}\delta u\rVert_{\dot{\mathrm{H}}^{\alpha,\tilde{p}}(\mathbb{R}^n)} + \lVert \delta\mathrm{d} u\rVert_{\dot{\mathrm{H}}^{\alpha,\tilde{p}}(\mathbb{R}^n)} \lesssim_{\tilde{p},n,\alpha} \lVert \nabla^2 u\rVert_{\dot{\mathrm{H}}^{\alpha,\tilde{p}}(\mathbb{R}^n)} &\lesssim_{\tilde{p},n,\alpha} \lVert f\rVert_{\dot{\mathrm{H}}^{\alpha,\tilde{p}}(\mathbb{R}^n)} \text{. }
        \end{align*}
        In particular, $-\Delta\,:\,[\dot{\mathrm{H}}^{s+2,p}\cap \dot{\mathrm{H}}^{\alpha+2,\tilde{p}}](\mathbb{R}^n,\Lambda^k)\longrightarrow [\dot{\mathrm{H}}^{s,p}\cap \dot{\mathrm{H}}^{\alpha,\tilde{p}}](\mathbb{R}^n,\Lambda^k)$ is an isomorphism of Banach spaces.
        
        \item The result still holds if we replace $( \dot{\mathrm{H}}^{s,p},\dot{\mathrm{H}}^{\alpha,\tilde{p}},\dot{\mathrm{H}}^{s+2,p},\dot{\mathrm{H}}^{\alpha+2,\tilde{p}})$ by $(\dot{\mathrm{B}}^{s}_{p,q}, \dot{\mathrm{B}}^{\alpha}_{\tilde{p},\tilde{q}},\dot{\mathrm{B}}^{s+2}_{p,q}, \dot{\mathrm{B}}^{\alpha+2}_{\tilde{p},\tilde{q}})$.
        \end{enumerate}
    \item For $\mu\in[0,\pi)$, $\lambda\in\Sigma_\mu$, $f\in \eus{S}'(\mathbb{R}^n,\Lambda^k)$, consider the problem
    \begin{align*}
        \lambda u - \Delta u = f\quad\text{ in } \mathbb{R}^n\text{.}
    \end{align*}
        \begin{enumerate}
        \item If $f\in \dot{\mathrm{H}}^{s,p}(\mathbb{R}^n,\Lambda^k)$, and $(\mathcal{C}_{s,p})$ is satisfied, then above resolvent problem admits a unique solution $u\in [\dot{\mathrm{H}}^{s,p}\cap \dot{\mathrm{H}}^{s+2,p}](\mathbb{R}^n,\Lambda^k)$ with estimates,
        \begin{align*}
            \lvert\lambda\rvert\lVert  u\rVert_{\dot{\mathrm{H}}^{s,p}(\mathbb{R}^n)}+\lvert\lambda\rvert^\frac{1}{2}\lVert \nabla u\rVert_{\dot{\mathrm{H}}^{s,p}(\mathbb{R}^n)}+\lVert \nabla^2 u\rVert_{\dot{\mathrm{H}}^{s,p}(\mathbb{R}^n)} &\lesssim_{p,n,s,\mu} \lVert f\rVert_{\dot{\mathrm{H}}^{s,p}(\mathbb{R}^n)} \text{, }\\
            \lvert\lambda\rvert^\frac{1}{2}\lVert (\mathrm{d}+\delta) u\rVert_{\dot{\mathrm{H}}^{s,p}(\mathbb{R}^n)}+\lVert \mathrm{d}\delta u\rVert_{\dot{\mathrm{H}}^{s,p}(\mathbb{R}^n)} + \lVert \delta\mathrm{d} u\rVert_{\dot{\mathrm{H}}^{s,p}(\mathbb{R}^n)} &\lesssim_{p,n,s,\mu} \lVert f\rVert_{\dot{\mathrm{H}}^{s,p}(\mathbb{R}^n)} \text{. }
        \end{align*}
        In particular, $\lambda \mathrm{I}-\Delta\,:\,[\dot{\mathrm{H}}^{s,p}\cap \dot{\mathrm{H}}^{s+2,{p}}](\mathbb{R}^n,\Lambda^k)\longrightarrow \dot{\mathrm{H}}^{s,p}(\mathbb{R}^n,\Lambda^k)$ is an isomorphism of Banach spaces.
        
        Furthermore, the result still holds replacing $(\dot{\mathrm{H}}^{s,p},\dot{\mathrm{H}}^{s,p}\cap \dot{\mathrm{H}}^{s+2,{p}})$ by $({\mathrm{H}}^{s,p},{\mathrm{H}}^{s+2,p})$ with out any restriction on $(s,p)$.
        
        \item If \eqref{AssumptionCompletenessExponents} is satisfied, then above result still holds replacing $(\dot{\mathrm{H}}^{s,p},\dot{\mathrm{H}}^{s,p}\cap \dot{\mathrm{H}}^{s+2,{p}})$ by $(\dot{\mathrm{B}}^{s}_{p,q},\dot{\mathrm{B}}^{s}_{p,q}\cap \dot{\mathrm{B}}^{s+2}_{p,q})$, or even by $({\mathrm{B}}^{s}_{p,q},{\mathrm{B}}^{s+2}_{p,q})$ with out any restriction on $(s,p,q)$.
        \end{enumerate}
        \item For any $\mu\in(0,\pi)$, the operator $-\Delta$ admits a bounded (or $\mathrm{\mathbf{H}}^\infty(\Sigma_\mu)$-)holormophic functional calculus on function spaces: $[\dot{\mathrm{H}}^{s,p}\cap \dot{\mathrm{H}}^{\alpha,\tilde{p}}](\mathbb{R}^n,\Lambda^k)$, $(\mathcal{C}_{s,p})$ being is satisfied, $\dot{\mathrm{B}}^{s}_{p,q}(\mathbb{R}^n,\Lambda^k)$, \eqref{AssumptionCompletenessExponents} being satisfied, and on both ${\mathrm{H}}^{s,p}(\mathbb{R}^n,\Lambda^k)$ and ${\mathrm{B}}^{s}_{p,q}(\mathbb{R}^n,\Lambda^k)$ without any restriction on $(s,p,q)$.
\end{enumerate}
\end{theorem}

\begin{proof} {\textbf{Step 1 :}} the scalar Laplacian is injective on $\eus{S}'_h(\mathbb{R}^n,\mathbb{C})$. For $f\in\eus{S}'(\mathbb{R}^n,\mathbb{C})$, let $u$, $v\in\eus{S}'_h(\mathbb{R}^n,\mathbb{C})$, such that
\begin{align*}
    \forall \phi\in\eus{S}(\mathbb{R}^n,\mathbb{C}),\,\, \big\langle u, -\Delta\phi \big\rangle_{\mathbb{R}^n} = \big\langle f,\phi \big\rangle_{\mathbb{R}^n} = \big\langle v, -\Delta\phi \big\rangle_{\mathbb{R}^n}\text{.}
\end{align*}
Therefore, it follows that $w:=u-v\in\eus{S}'_h(\mathbb{R}^n,\mathbb{C})$ satisfies
\begin{align*}
    \forall \phi\in\eus{S}(\mathbb{R}^n,\mathbb{C}),\,\, \big\langle w, -\Delta\phi \big\rangle_{\mathbb{R}^n} = 0\text{.}
\end{align*}
Hence, one may apply the Fourier transform, to obtain
\begin{align*}
    \forall \phi\in\eus{S}(\mathbb{R}^n,\mathbb{C}),\,\, \big\langle \eus{F}w, \left\lvert\cdot\right\rvert^2\eus{F}^{-1}\phi\big\rangle_{\mathbb{R}^n} = 0\text{,}
\end{align*}
so in particular, for test function in the  form $\phi:=\eus{F}\left[\frac{\psi}{\left\lvert\cdot\right\rvert^2}\right]$, with $\psi\in \mathrm{C}^\infty_c(\mathbb{R}^n\setminus\{0\},\mathbb{C})$ (notice that one can see that $\eus{F}\left[\frac{\psi}{\left\lvert\cdot\right\rvert^2}\right]\in\eus{S}(\mathbb{R}^n,\mathbb{C})$), we deduce
\begin{align*}
    \forall \psi\in \mathrm{C}_c^\infty(\mathbb{R}^n\setminus\{0\},\mathbb{C}),\,\, \big\langle \eus{F}w, \psi \big\rangle_{\mathbb{R}^n} = 0\text{.}
\end{align*}
Following the proof of \cite[Lemma~3.6]{DanchinHieberMuchaTolk2020}, we conclude that $w=0$ (since $w\in\eus{S}'(\mathbb{R}^n,\mathbb{C})$ and $\eus{S}'_h(\mathbb{R}^n,\mathbb{C})$ does not contain any polynomial), so $u=v$ in $\eus{S}'(\mathbb{R}^n,\mathbb{C})$. 

{\textbf{Step 2 :}} For the point \textit{(i)}, it suffices to follow \cite[Lemma~3.1.1]{DanchinMucha2015}.

For the point \textit{(ii)}, see \cite[Example~3.7.6,~Theorem~3.7.11]{ArendtBattyHieberNeubranker2011}.

For the point \textit{(iii)}, the result on $\mathrm{L}^p(\mathbb{R}^n,\mathbb{C})$ is a consequence of a more general one which is \cite[Proposition~8.3.4]{bookHaase2006}.
\end{proof}

Similarly, thanks again to standard Fourier analysis we can introduce appropriate differential form-valued version of Riesz transforms for the Hodge Laplacian. Their boundedness on appropriate function spaces are again carried over by their scalar analogue and a direct consequence will be an explicit formula for our generalized Leray projector $\mathbb{P}$ on $\mathbb{R}^n$.

To do so, we notice that one can write associated Fourier symbols, thanks to Lemma \ref{lem:FundamentalLemmaDiffFormsRn}, to obtain
\begin{align*}
    \mathrm{d}(-\Delta)^{-\frac{1}{2}} = \sum_{k=1}^n R_k \mathfrak{e}_k\wedge \quad\text{ and }\quad \delta(-\Delta)^{-\frac{1}{2}} = -\sum_{k=1}^n R_k \mathfrak{e}_k\iprod\text{ }
\end{align*}
where for $k\in\llb 1,n\rrb$, $R_k$ is the $k$-th Riesz transform on $\mathbb{R}^n$  given by the Fourier symbol $\xi \mapsto \frac{i\xi_k}{\lvert\xi\rvert}$, which is well known to be bounded on $\mathrm{L}^p(\mathbb{R}^n,\mathbb{C})$, $1<p<+\infty$, see \cite[Chapter~2,~Theorem~1~\&~Chapter~3,~Section~1]{Stein1970}. Therefore, the next proposition follows naturally.

\begin{proposition}\label{prop:RiesTransformHodgeRn}Let $p,\Tilde{p}\in(1,+\infty)$, $q,\Tilde{q}\in[1,+\infty]$, $s,\alpha\in\mathbb{R}$, and let $k\in\llb 0,n\rrb$. The operators
\begin{align*}
    \mathrm{d}(-\Delta)^{-\frac{1}{2}}\text{, } \delta(-\Delta)^{-\frac{1}{2}} \text{, }  \mathrm{d}\delta(-\Delta)^{-1} \text{ and } \delta\mathrm{d}(-\Delta)^{-1}
\end{align*}
are all well defined bounded linear operators,
\begin{itemize}
    \item On $[\dot{\mathrm{H}}^{s,p}\cap \dot{\mathrm{H}}^{\alpha,\tilde{p}}](\mathbb{R}^n,\Lambda)$, and on $[\dot{\mathrm{B}}^{s}_{p,q}\cap \dot{\mathrm{B}}^{\alpha}_{\tilde{p},\tilde{q}}](\mathbb{R}^n,\Lambda)$, \eqref{AssumptionCompletenessExponents} being satisfied, with decoupled estimates,
    \begin{align*}
        \lVert T u \rVert_{\mathrm{X}^s(\mathbb{R}^n)} \lesssim_{p,n,s} \lVert u \rVert_{\mathrm{X}^s(\mathbb{R}^n)} \text{ and } \lVert T u \rVert_{\mathrm{X}^\alpha(\mathbb{R}^n)} \lesssim_{p,n,\alpha} \lVert u \rVert_{\mathrm{X}^\alpha(\mathbb{R}^n)}\text{, }\, u\in [\mathrm{X}^s\cap \mathrm{X}^\alpha](\mathbb{R}^n,\Lambda^k)\text{, }
    \end{align*}
    where the operator $T$ is any of above operators, and $(\mathrm{X}^s, \mathrm{X}^\alpha) \in\{ (\dot{\mathrm{H}}^{s,p},\dot{\mathrm{H}}^{\alpha,\tilde{p}}), (\dot{\mathrm{B}}^{s}_{p,q},\dot{\mathrm{B}}^{\alpha}_{\tilde{p},\tilde{q}})\}$.
    \item On ${\mathrm{H}}^{s,p}(\mathbb{R}^n,\Lambda)$, and on ${\mathrm{B}}^{s}_{p,q}(\mathbb{R}^n,\Lambda)$, without any restriction on $(s,p,q)$.
\end{itemize}
Moreover the following identity holds
\begin{align*}
    (\mathrm{d}(-\Delta)^{-\frac{1}{2}}+ \delta(-\Delta)^{-\frac{1}{2}})^2 = \mathrm{I}\text{.}
\end{align*}
\end{proposition}

\begin{theorem}\label{thm:HodgeDecompRn}Let $p\in(1,+\infty)$, $q\in[1,+\infty]$, $s\in\mathbb{R}$, and let $k\in\llb 0,n\rrb$. The following hold
\begin{enumerate}
    \item  The following equality is true, whenever $(\mathcal{C}_{s,p})$ is satisfied,
    \begin{align*}
        \dot{\mathrm{N}}^{s}_{p}(\mathrm{d},\mathbb{R}^n,\Lambda^k) = \overline{\dot{\mathrm{R}}^{s}_{p}(\mathrm{d},\mathbb{R}^n,\Lambda^k)}^{\lVert\cdot\rVert_{\dot{\mathrm{H}}^{s,p}(\mathbb{R}^n)}}\text{, }
    \end{align*}
    and still holds replacing $\mathrm{d}$ by $\delta$.
    
    \item The (generalized) Helmholtz-Leray projector $\mathbb{P}\,:\, \dot{\mathrm{H}}^{s,p}(\mathbb{R}^n,\Lambda^k)\longrightarrow \dot{\mathrm{N}}^{s}_{p}(\delta,\mathbb{R}^n,\Lambda^k)$ is well defined and bounded when $(\mathcal{C}_{s,p})$ is satisfied. Moreover the following identity is true
    \begin{align*}
        \mathbb{P} = \mathrm{I} - \mathrm{d}(-\Delta)^{-1}\delta \text{.}
    \end{align*}

    \item The following Hodge decomposition holds whenever $(\mathcal{C}_{s,p})$ is satisfied,
    \begin{align*}
        \dot{\mathrm{H}}^{s,p}(\mathbb{R}^n,\Lambda^k) = \dot{\mathrm{N}}^{s}_{p}(\delta,\mathbb{R}^n,\Lambda^k) \oplus \dot{\mathrm{N}}^{s}_{p}(\mathrm{d},\mathbb{R}^n,\Lambda^k)\text{. }
    \end{align*}
\end{enumerate}
Everything above still holds replacing $(\dot{\mathrm{H}}^{s,p},\dot{\mathrm{N}}^{s}_{p},\dot{\mathrm{R}}^{s}_{p})$ by either $(\dot{\mathrm{B}}^{s}_{p,q},\dot{\mathrm{N}}^{s}_{p,q},\dot{\mathrm{R}}^{s}_{p,q})$, \eqref{AssumptionCompletenessExponents} being satisfied, $({\mathrm{H}}^{s,p},{\mathrm{N}}^{s}_{p},{\mathrm{R}}^{s}_{p})$ or $({\mathrm{B}}^{s}_{p,q},{\mathrm{N}}^{s}_{p,q},{\mathrm{R}}^{s}_{p,q})$ without any restriction on $(s,p,q)$.\\
In case of Besov spaces with $q=+\infty$, the density result of point \textit{(i)} only holds in the weak${}^\ast$ sense.
\end{theorem}

\begin{remark}On $\Lambda^1$-valued functions identified with vector fields one recover the usual well known formula, \textit{i.e.}
\begin{align*}
    \mathbb{P}=\mathrm{I}+ \nabla (-\Delta)^{-1}\div \text{.}
\end{align*}
\end{remark}

\begin{proof} {\textbf{Step 1 :}} The orthogonal projector $\mathbb{P}$ is originally defined only as an operator
\begin{align*}
    \mathbb{P}\,:\, \mathrm{L}^2(\mathbb{R}^n,\Lambda^k)\longrightarrow\mathrm{N}_2(\delta,\mathbb{R}^n,\Lambda^k)\text{. }
\end{align*}
We claim that $\mathbb{P}$ is equal to the operator formally given by 
\begin{align*}
    \tilde{\mathbb{P}}:=\mathrm{I} - \mathrm{d}(-\Delta)^{-1}\delta \text{. }
\end{align*}
The proof is standard, and works as in the case of vector fields, and then is left to the reader.

{\textbf{Step 2 :}} Previous step and Proposition \ref{prop:RiesTransformHodgeRn} give that $\mathbb{P}$ is bounded $\dot{\mathrm{H}}^{s,p}(\mathbb{R}^n,\Lambda^k)$.

For $u\in \dot{\mathrm{H}}^{s,p}(\mathbb{R}^n,\Lambda^k)$, $v\in\eus{S}(\mathbb{R}^n)$, we regularize with the resolvent, to compute
\begin{align*}
    \big\langle {\mathbb{P}} u, \mathrm{d}v \big\rangle_{\mathbb{R}^n} &= \lim_{ \lambda \rightarrow 0_+}  \big\langle u,\mathrm{d}v \big\rangle_{\mathbb{R}^n} - \big\langle \mathrm{d}(\lambda \mathrm{I}-\Delta)^{-1}\delta u,\mathrm{d}v \big\rangle_{\mathbb{R}^n}\\
    &= \lim_{ \lambda \rightarrow 0_+}  \big\langle u,\mathrm{d}v \big\rangle_{\mathbb{R}^n} - \big\langle u, (\lambda \mathrm{I}-\Delta)^{-1}\mathrm{d}\delta\mathrm{d}v \big\rangle_{\mathbb{R}^n}\\
    &= \lim_{ \lambda \rightarrow 0_+}  \big\langle u,\mathrm{d}v \big\rangle_{\mathbb{R}^n} - \big\langle u, (\lambda \mathrm{I}-\Delta)^{-1}(-\Delta)\mathrm{d}v \big\rangle_{\mathbb{R}^n}\\
    &= \big\langle u,\mathrm{d}v \big\rangle_{\mathbb{R}^n} - \big\langle u,\mathrm{d}v \big\rangle_{\mathbb{R}^n} =0 \text{. }
\end{align*}
Hence $\mathbb{P}\dot{\mathrm{H}}^{s,p}(\mathbb{R}^n,\Lambda^k)\subset \dot{\mathrm{N}}^{s}_{p}(\delta,\mathbb{R}^n,\Lambda^k)$, and we even have $\mathbb{P}_{|_{\dot{\mathrm{N}}^{s}_{p}(\delta,\mathbb{R}^n,\Lambda^k)}}= \mathrm{I}$, so that
\begin{align*}
    \mathbb{P}\dot{\mathrm{H}}^{s,p}(\mathbb{R}^n,\Lambda^k)= \dot{\mathrm{N}}^{s}_{p}(\delta,\mathbb{R}^n,\Lambda^k)\text{. }
\end{align*}
Similarly, $[\mathrm{I}-\mathbb{P}]\dot{\mathrm{H}}^{s,p}(\mathbb{R}^n,\Lambda^k) \subset \dot{\mathrm{N}}^{s}_{p}(\mathrm{d},\mathbb{R}^n,\Lambda^k)$, and $[\mathrm{I}-\mathbb{P}]_{|_{\dot{\mathrm{N}}^{s}_{p}(\mathrm{d},\mathbb{R}^n,\Lambda^k)}}= \mathrm{I}$, which comes from
\begin{align*}
    [\mathrm{I}-\mathbb{P}]u = \lim_{ \lambda \rightarrow 0_+} \mathrm{d}(\lambda \mathrm{I}-\Delta)^{-1} \delta u \text{, }\, u\in \dot{\mathrm{H}}^{s,p}(\mathbb{R}^n,\Lambda^k)\text{. }
\end{align*}
This also gives $\dot{\mathrm{N}}^{s}_{p}(\mathrm{d},\mathbb{R}^n,\Lambda^k) = \overline{\dot{\mathrm{R}}^{s}_{p}(\mathrm{d},\mathbb{R}^n,\Lambda^k)}^{\lVert\cdot\rVert_{\dot{\mathrm{H}}^{s,p}(\mathbb{R}^n)}}$.

The proof is straightforward the same for other function spaces.
\end{proof}

\subsection{The case of the half-space \texorpdfstring{$\mathbb{R}^n_+$}{Rn+}}

\subsubsection{\texorpdfstring{$\mathrm{L}^2$}{L2}-theory for Hodge Laplacians and the Hodge decomposition}

The following lemma is fundamental for the analysis of the $\mathrm{L}^2$ theory of the Hodge Laplacian when one has an explicit access to the boundary, and moreover several proofs presented here do not depend on the open set $\Omega$, here $\Omega=\mathbb{R}^n_+$, and remain valid as long as integration by parts formulas \eqref{eq:IntbyParts1} and \eqref{eq:IntbyParts2} and partial traces results for vector fields are available.

\begin{lemma}\label{lem:closdenessNdualityDerivativesL2Rn+}Let $k\in\llb 0,n\rrb$. We set
\begin{align*}
    \mathrm{D}_2(\underline{\delta},\mathbb{R}^n_+,\Lambda^{k}) &:=\{\,u\in \mathrm{D}_2(\delta,\mathbb{R}^n_+,\Lambda^{k})\,|\, \nu\iprod u_{|_{\partial\mathbb{R}^n_+}}=0\,\}\text{, } \\
    \mathrm{D}_2(\underline{\mathrm{d}},\mathbb{R}^n_+,\Lambda^{k}) &:= \{\,u\in \mathrm{D}_2(\mathrm{d},\mathbb{R}^n_+,\Lambda^{k})\,|\, \nu\wedge u_{|_{\partial\mathbb{R}^n_+}}=0\,\}\text{. }
\end{align*}
The operator $(\mathrm{D}_2(\mathrm{d},\mathbb{R}^n_+,\Lambda^{k}),\mathrm{d})$ is an unbounded densely defined closed operator, with adjoint
\begin{align}
    (\mathrm{D}_2(\mathrm{d}^\ast,\mathbb{R}^n_+,\Lambda^{k+1}),\mathrm{d}^\ast)=(\mathrm{D}_2(\underline{\delta},\mathbb{R}^n_+,\Lambda^{k+1}),\delta)\text{. }\label{eq:L2dualofextderiv.}
\end{align}
Similarly, $(\mathrm{D}_2(\delta,\mathbb{R}^n_+,\Lambda^{k}),\delta)$ is an unbounded densely defined closed operator, with adjoint
\begin{align}
    (\mathrm{D}_2(\delta^\ast,\mathbb{R}^n_+,\Lambda^{k-1}),\delta^\ast)=(\mathrm{D}_2(\underline{\mathrm{d}},\mathbb{R}^n_+,\Lambda^{k-1}),\mathrm{d})\text{. }\label{eq:L2dualofintderiv.}
\end{align}
\end{lemma}

\begin{proof} Closedness and the fact that both are densely defined is straightforward. We just prove the duality identity \eqref{eq:L2dualofextderiv.}, the proof of \eqref{eq:L2dualofintderiv.} is similar. Let $u\in \mathrm{D}_2(\underline{\delta},\mathbb{R}^n_+,\Lambda^{k})$, then for all $v\in\eus{S}_0(\overline{\mathbb{R}^n_+},\Lambda^k)$, we can use Theorem \ref{thm:TracesDifferentialformsInhomSpaces}, to obtain that
\begin{align*}
    \big\langle v, \delta u \big\rangle_{\mathbb{R}^n_+} = \big\langle \mathrm{d}v, u \big\rangle_{\mathbb{R}^n_+} \text{. }
\end{align*}
Thus, by Cauchy-Schwarz inequality
\begin{align*}
    \big\lvert\big\langle \mathrm{d} v,  u \big\rangle_{\mathbb{R}^n_+}\big\rvert \leqslant \lVert \delta u\rVert_{\mathrm{L}^2(\mathbb{R}^n_+)} \lVert v\rVert_{\mathrm{L}^2(\mathbb{R}^n_+)}\text{. }
\end{align*}
Hence, $v\mapsto \big\langle \mathrm{d} v,  u \big\rangle_{\mathbb{R}^n_+}$ extends uniquely as a bounded linear functional on $\mathrm{L}^2(\mathbb{R}^n_+,\Lambda^{k})$, so that necessarily $(\mathrm{D}_2(\underline{\delta},\mathbb{R}^n_+,\Lambda^{k}),\delta)\subset (\mathrm{D}_2(\mathrm{d}^\ast,\mathbb{R}^n_+,\Lambda^{k+1}),\mathrm{d}^\ast)$.

For the reverse inclusion, let $u\in \mathrm{D}_2(\mathrm{d}^\ast,\mathbb{R}^n_+,\Lambda^{k+1})$, for all $v\in \mathrm{D}_2(\mathrm{d},\mathbb{R}^n_+,\Lambda^{k})$, we have
\begin{align*}
    \big\langle v, \mathrm{d}^{\ast} u \big\rangle_{\mathbb{R}^n_+} = \big\langle \mathrm{d}v, u \big\rangle_{\mathbb{R}^n_+} \text{. }
\end{align*}
In particular, for $v\in\mathrm{C}_c^\infty(\mathbb{R}^n_+,\Lambda^k)$, it yields that
\begin{align*}
    \big\langle v, \mathrm{d}^{\ast} u \big\rangle_{\mathbb{R}^n_+} = \big\langle \mathrm{d}v, u \big\rangle_{\mathbb{R}^n_+} = \big\langle v, \delta u \big\rangle_{\mathbb{R}^n_+} \text{. }
\end{align*}
Hence, $\mathrm{d}^{\ast} u = \delta u$ in $\eus{D}'(\mathbb{R}^n_+,\Lambda^{k})$, then in $\mathrm{L}^2(\mathbb{R}^n_+,\Lambda^{k})$, so that for all $v\in \mathrm{D}_2(\mathrm{d},\mathbb{R}^n_+,\Lambda^{k})$,
\begin{align*}
    \big\langle v, \delta u \big\rangle_{\mathbb{R}^n_+} = \big\langle \mathrm{d}v, u \big\rangle_{\mathbb{R}^n_+} \text{. }
\end{align*}
From above equality, we apply Theorem \ref{thm:TracesDifferentialformsInhomSpaces} to check $\nu\iprod u_{|_{\partial\mathbb{R}^n_+}}=0$ and deduce
\begin{align*}
    (\mathrm{D}_2(\mathrm{d}^\ast,\mathbb{R}^n_+,\Lambda^{k+1}),\mathrm{d}^\ast)\subset(\mathrm{D}_2(\underline{\delta},\mathbb{R}^n_+,\Lambda^{k+1}),\delta)\text{, }
\end{align*}
the proof being therefore complete.
\end{proof}

In particular, since $(\mathrm{D}_2(\mathrm{d}^\ast,\mathbb{R}^n_+,\Lambda^{k}),\mathrm{d}^\ast)$ and $(\mathrm{D}_2(\delta,\mathbb{R}^n_+,\Lambda^{k}),\delta)$ are closed operators, both
\begin{align*}
    \mathrm{N}_2(\delta,\mathbb{R}^n_+,\Lambda^{k}) \,\text{ and }\,\mathrm{N}_2(\mathrm{d}^\ast,\mathbb{R}^n_+,\Lambda^{k})
\end{align*}
are closed subspaces of $\mathrm{L}^2(\mathbb{R}^n_+,\Lambda^{k})$. Thus the following orthogonal projection are well defined and bounded
\begin{align*}
\begin{array}{cll}
     &\mathbb{P}\,:\,\mathrm{L}^2(\mathbb{R}^n_+,\Lambda^{k})\longrightarrow \mathrm{N}_2(\mathrm{d}^\ast,\mathbb{R}^n_+,\Lambda^{k}) \,\text{, }\,& [\mathrm{I}-\mathbb{P}]\,:\,\mathrm{L}^2(\mathbb{R}^n_+,\Lambda^{k})\longrightarrow \overline{\mathrm{R}_2(\mathrm{d},\mathbb{R}^n_+,\Lambda^{k})}\text{, }\\
    &\mathbb{Q}\,:\,\mathrm{L}^2(\mathbb{R}^n_+,\Lambda^{k})\longrightarrow \mathrm{N}_2(\delta,\mathbb{R}^n_+,\Lambda^{k}) \,\text{, }\,& [\mathrm{I}-\mathbb{Q}]\,:\,\mathrm{L}^2(\mathbb{R}^n_+,\Lambda^{k})\longrightarrow \overline{\mathrm{R}_2(\delta^\ast,\mathbb{R}^n_+,\Lambda^{k})}\text{, }
\end{array}
\end{align*}
which induce topological Hodge decomposition
\begin{align*}\tag{$\mathfrak{H}_2$}\label{HodgeDecompL2}
       \mathrm{L}^{2}(\mathbb{R}^n_+,\Lambda^k) &= \overline{\mathrm{R}_2(\mathrm{d},\mathbb{R}^n_+,\Lambda^k)}\overset{\perp}{\oplus} \mathrm{N}_2(\mathrm{d}^\ast,\mathbb{R}^n_+,\Lambda^k) \text{, }\\
        &=   \overline{\mathrm{R}_2(\delta^\ast,\mathbb{R}^n_+,\Lambda^k)} \overset{\perp}{\oplus} {\mathrm{N}_2(\delta,\mathbb{R}^n_+,\Lambda^k)} \text{. }
\end{align*}

\begin{lemma}For $k\in\llb0,n\rrb$, the following Hodge-Dirac operators
\begin{align*}
    (\mathrm{D}_2(D_{\mathfrak{t}},\mathbb{R}^n_+,\Lambda^k),D_{\mathfrak{t}})&= (\mathrm{D}_2({\mathrm{d}},\mathbb{R}^n_+,\Lambda^k)\cap\mathrm{D}_2({\mathrm{d}^\ast},\mathbb{R}^n_+,\Lambda^k),\, \mathrm{d}+\mathrm{d}^\ast) \text{, }\\
    (\mathrm{D}_2(D_{\mathfrak{n}},\mathbb{R}^n_+,\Lambda^k),D_{\mathfrak{n}})&= (\mathrm{D}_2({\delta^\ast},\mathbb{R}^n_+,\Lambda^k)\cap\mathrm{D}_2({\delta},\mathbb{R}^n_+,\Lambda^k),\, \delta^\ast+\delta)\text{, }
\end{align*}
are both densely defined closed operators on $\mathrm{L}^2(\mathbb{R}^n_+,\Lambda^k)$.
\end{lemma}

\begin{proof}Let $(u_j)_{j\in\mathbb{N}}\subset \mathrm{D}_2({\mathrm{d}},\mathbb{R}^n_+,\Lambda^k)\cap\mathrm{D}_2({\mathrm{d}^\ast},\mathbb{R}^n_+,\Lambda^k)$, and $(u,v)\in\mathrm{L}^2(\mathbb{R}^n_+, \Lambda^k)\times\mathrm{L}^2(\mathbb{R}^n_+, \Lambda)$ such that it satisfies
\begin{align*}
    u_j \xrightarrow[j\rightarrow +\infty]{}  u \quad\text{ and }\quad D_\mathfrak{t} u_j \xrightarrow[j\rightarrow +\infty]{}  v \quad\text{ in } \mathrm{L}^2(\mathbb{R}^n_+) \text{. }
\end{align*}
By the Hodge decomposition, there exists a unique couple $({v}_{0},{v}_{1})\in \overline{\mathrm{R}_2({\mathrm{d}},\mathbb{R}^n_+,\Lambda)}\times \mathrm{N}_2({\mathrm{d}^\ast},\mathbb{R}^n_+,\Lambda)$ such that $v={v}_{0}+{v}_{1}$.
Since $u_j$ goes to $u$ in $\mathrm{L}^2(\mathbb{R}^n_+,\Lambda^k)$, by continuity of involved projectors, and uniqueness of decomposition, it follows that
\begin{align*}
    \mathrm{d} u_j \xrightarrow[j\rightarrow +\infty]{} v_0 \quad\text{ and }\quad \mathrm{d}^\ast u_j \xrightarrow[j\rightarrow +\infty]{} v_1 \quad\text{ in } \mathrm{L}^2(\mathbb{R}^n_+) \text{. }
\end{align*}
But $(u_j)_{j\in\mathbb{N}}$ converge to $u$ in $\mathrm{L}^2(\mathbb{R}^n_+,\Lambda^k)$, so in particular in distributional sense, thus necessarily $(v_0,v_1)=(\mathrm{d} u,\mathrm{d}^\ast u)$ and $v=D_\mathfrak{t} u$, \textit{ i.e. } $(\mathrm{D}_2(D_{\mathfrak{t}},\mathbb{R}^n_+,\Lambda^k),D_{\mathfrak{t}})$ is closed on $\mathrm{L}^2(\mathbb{R}^n_+,\Lambda^k)$. The proof ends here since one can reproduce all above arguments for $(\mathrm{D}_2(D_{\mathfrak{n}},\mathbb{R}^n_+,\Lambda^k),D_{\mathfrak{n}})$.
\end{proof}

\begin{proposition}\label{prop:HodgeDiracL2Rn+}The Hodge-Dirac operator $(\mathrm{D}_2(D_\cdot,\mathbb{R}^n_+,\Lambda),D_\cdot)$ is an injective self-adjoint $0$-bisectorial operator on $\mathrm{L}^2(\mathbb{R}^n_+,\Lambda)$ so that it satisfies the following bound, for all $\theta\in (0,\frac{\pi}{2})$,
\begin{align}\label{eq:resolvEstDiracL2}
    \forall \mu\in\mathbb{C}\setminus \overline{S}_\theta\text{, } \left\lVert \mu(\mu\mathrm{I}+ D_\cdot)^{-1}\right\rVert_{\mathrm{L}^2(\mathbb{R}^n_+)\rightarrow \mathrm{L}^2(\mathbb{R}^n_+)}\leqslant \frac{1}{\sin(\theta)} \text{.}
\end{align}
Moreover, it admits a bounded ($\mathrm{\mathbf{H}}^\infty(S_\theta)$-)functional calculus on $\mathrm{L}^2(\mathbb{R}^n_+,\Lambda)$ with bound $1$, \textit{i.e.}, for all $f\in\mathrm{\mathbf{H}}^\infty(S_\theta)$, $u\in\mathrm{L}^2(\mathbb{R}^n_+,\Lambda)$,
\begin{align}
    \left\lVert f(D_\cdot)u\right\rVert_{\mathrm{L}^2(\mathbb{R}^n_+)}\leqslant \left\lVert f\right\rVert_{\mathrm{L}^\infty(S_\theta)}\left\lVert u\right\rVert_{\mathrm{L}^2(\mathbb{R}^n_+)}\text{. }
\end{align}
\end{proposition}

\begin{remark}Proposition \ref{prop:HodgeDiracL2Rn+} does not depend on the fact $\Omega = \mathbb{R}^n_+$. See \cite[Section~2]{McintoshMonniaux2018} where the same result is stated for bounded (even weak-)Lipschitz domains.
\end{remark}

\begin{proof}The resolvent bound \eqref{eq:resolvEstDiracL2} is usual since $(\mathrm{D}_2(D_\cdot),D_\cdot)$ is self-adjoint by construction, see \cite[Proposition~C.4.2]{bookHaase2006}.
The fact that it admits a bounded holomorphic functional calculus follows from \cite[Section~10]{McIntosh1986}.
\end{proof}

For $k\in\llb 0,n\rrb$, the Hodge Laplacian $(\mathrm{D}(\Delta_\mathcal{H},\mathbb{R}^n_+,\Lambda^k),-\Delta_\mathcal{H})$ can be realized on $\mathrm{L}^2(\mathbb{R}^n_+,\Lambda^k)$ by the mean of densely defined, symmetric, accretive, continuous, closed, sesquilinear forms on $\mathrm{L}^2(\mathbb{R}^n_+,\Lambda^k)$, for
\begin{align}\label{eq:sesqlinformLaplacianHodge}
    \mathfrak{a}_\mathcal{H}\,:\, \mathrm{D}_2(\mathfrak{a}_\mathcal{H},\Lambda^k)^2\ni(u,v) \longmapsto \int_{\mathbb{R}^n_+}\langle \mathrm{d} u, \overline{\mathrm{d} v}\rangle + \int_{\mathbb{R}^n_+}\langle \delta u, \overline{\delta v}\rangle
\end{align}
with $\mathrm{D}_2(\mathfrak{a}_{\mathcal{H},\mathfrak{t}},\Lambda^k)= \mathrm{D}_2(\mathrm{d},\Lambda^k)\cap\mathrm{D}_2(\mathrm{d}^\ast,\Lambda^k)$, $\mathrm{D}_2(\mathfrak{a}_{\mathcal{H},\mathfrak{n}})= \mathrm{D}_2(\delta^\ast,\Lambda^k)\cap\mathrm{D}_2(\delta,\Lambda^k)$, so that it is easy to see that both are closed, densely defined, non-negative self-adjoint operators on $\mathrm{L}^2(\mathbb{R}^n_+,\Lambda^k)$. See \cite[Chapter~1]{bookOuhabaz2005} for more details about realization of operators via sesquilinear forms on a Hilbert space. The next theorem is a standard consequence.

\begin{theorem}\label{thm:HodgeLapL2Rn+}Let $k\in\llb 0,n\rrb$. The operator $(\mathrm{D}_2(\Delta_\mathcal{H},\mathbb{R}^n_+,\Lambda^k),-\Delta_\mathcal{H})$ is an injective non-negative self-adjoint and $0$-sectorial operator on $\mathrm{L}^2(\mathbb{R}^n_+,\Lambda^k)$, which admits bounded (or $\mathrm{\mathbf{H}}^\infty(\Sigma_\theta)$-) holomorphic functional calculus for all $\theta\in(0,\pi)$.

Moreover, the following hold
\begin{enumerate}
    \item $\mathrm{D}_2(\Delta_\mathcal{H},\mathbb{R}^n_+,\Lambda^k)$ is a closed subspace of $\mathrm{H}^2(\mathbb{R}^n_+,\Lambda)$;
    \item Provided $\mu\in [0,\pi)$, for $\lambda\in\Sigma_\mu$, $f\in\mathrm{L}^2(\mathbb{R}^n_+,\Lambda^k)$, then $u:=(\lambda\mathrm{I}-\Delta_\mathcal{H})^{-1}f $ satisfies
    \begin{align}
        \lvert\lambda\rvert\lVert  u\rVert_{{\mathrm{L}}^{2}(\mathbb{R}^n_+)}+\lvert\lambda\rvert^\frac{1}{2}\lVert D_{\cdot} u\rVert_{{\mathrm{L}}^{2}(\mathbb{R}^n_+)}+\lVert \Delta u\rVert_{{\mathrm{L}}^{2}(\mathbb{R}^n_+)} &\lesssim_{\mu} \lVert f\rVert_{{\mathrm{L}}^{2}(\mathbb{R}^n_+)} \text{ ; }
         \label{eq:ResolvEstHodgeLapL2Rn+} \\ \lvert\lambda\rvert\lVert u\rVert_{{\mathrm{L}}^{2}(\mathbb{R}^n_+)}+\lvert\lambda\rvert^\frac{1}{2}\lVert \nabla u\rVert_{{\mathrm{L}}^{2}(\mathbb{R}^n_+)}+\lVert \nabla^2 u\rVert_{{\mathrm{L}}^{2}(\mathbb{R}^n_+)} &\lesssim_{n,\mu} \lVert f\rVert_{{\mathrm{L}}^{2}(\mathbb{R}^n_+)} \label{eq:H2ResolvEstHodgeLapL2Rn+} \text{ ; }
    \end{align}
    \item The following resolvent identity holds for all $\mu\in [0,\pi)$, $\lambda\in\Sigma_\mu$, $f\in\mathrm{L}^2(\mathbb{R}^n_+,\Lambda^k)$,
    \begin{align*}
        \mathrm{E}_\mathcal{H}(\lambda\mathrm{I}-\Delta_\mathcal{H})^{-1}f = (\lambda\mathrm{I}-\Delta)^{-1}\mathrm{E}_\mathcal{H} f\text{. }
    \end{align*}
\end{enumerate}
\end{theorem}

\begin{remark}In above Theorem \ref{thm:HodgeLapL2Rn+}, points \textit{(i)} and \textit{(iii)}, as well as the estimate \eqref{eq:H2ResolvEstHodgeLapL2Rn+} of point \textit{(ii)} are the only points that relies on the fact that the considered openset is $\Omega=\mathbb{R}^n_+$, but mainly the point \textit{(iii)} is used to deduce previous ones. 
The beginning of the statement, as well as \eqref{eq:ResolvEstHodgeLapL2Rn+}, does not rely on any particular structure, and remains true on any open set $\Omega$. 

Every other results below, in the present subsection about $\mathrm{L}^2$-theory of Hodge Laplacians and the Hodge decomposition, remain true on general domains $\Omega$ as long as one can show that Hodge Laplacians are injective.
\end{remark}

Before proving Theorem \ref{thm:HodgeLapL2Rn+}, following \cite[Section~4]{Gaudin2022} for $\mathcal{J}\in\{\mathcal{D},\mathcal{N}\}$, we introduce the following extension operator defined for any measurable function $u$ on $\mathbb{R}^n_+$, for almost every $x=(x',x_n)\in\mathbb{R}^{n}$:
\begin{align*}
    \mathrm{E}_\mathcal{D}u (x',x_n) &:= \begin{cases}
  u(x',x_n)\,&\text{, if } (x',x_n)\in\mathbb{R}^{n-1}\times\mathbb{R}_+\text{, }\\    
  -u(x',-x_n)\,&\text{, if } (x',x_n)\in\mathbb{R}^{n-1}\times\mathbb{R}_-^*\text{ ; }
\end{cases}\\
    \mathrm{E}_\mathcal{N}u (x',x_n) &:= \begin{cases}
  u(x',x_n)\,&\text{, if } (x',x_n)\in\mathbb{R}^{n-1}\times\mathbb{R}_+\text{, }\\    
  u(x',-x_n)\,&\text{, if } (x',x_n)\in\mathbb{R}^{n-1}\times\mathbb{R}_-^* \text{. }
\end{cases}
\end{align*}
Now, we precise the definition of extension operators $\mathrm{E}_{\mathcal{H},\mathfrak{j}}$, $\mathfrak{j}\in\{\mathfrak{t},\mathfrak{n}\}$, on measurable functions $u\,:\,\mathbb{R}^n_+\longrightarrow \Lambda^k$, provided $k\in\llb 0,n\rrb$, $I\in\mathcal{I}^{k}_{n}$,
\begin{align}\label{eq:HodgeReflexExtOp}
    (\mathrm{E}_{\mathcal{H},\mathfrak{t}}u)_{I} := \begin{cases}
  \mathrm{E}_{\mathcal{D}}u_{I}\,&\text{, if } n\in I\text{, }\\    
  \mathrm{E}_{\mathcal{N}}u_{I}\,&\text{, if } n\notin I\text{ ; } 
    \end{cases}\, \text{ and }\, (\mathrm{E}_{\mathcal{H},\mathfrak{n}}u)_{I} := \begin{cases}
  \mathrm{E}_{\mathcal{N}}u_{I}\,&\text{, if } n\in I\text{, }\\    
  \mathrm{E}_{\mathcal{D}}u_{I}\,&\text{, if } n\notin I\text{ ; } 
    \end{cases} \text{. }
\end{align}
For $u\,:\,\mathbb{R}^n_+\longrightarrow \Lambda^k$, we also set
\begin{align*}
    \tilde{u}_{\mathfrak{j}} := [\mathrm{E}_{\mathcal{H},\mathfrak{j}} u]_{|_{\mathbb{R}^n_{-}}}\text{. }
\end{align*}

By construction, for $\mathfrak{j}\in\{\mathfrak{t},\mathfrak{n}\}$, $s\in (-1+1/p,1/p)$, $p\in(1,+\infty)$, the Proposition \ref{prop:SobolevMultiplier} leads to boundedness of
\begin{align}
    \mathrm{E}_{\mathcal{H},\mathfrak{j}}\,:\, \dot{\mathrm{H}}^{s,p}(\mathbb{R}^n_+,\Lambda)\longrightarrow \dot{\mathrm{H}}^{s,p}(\mathbb{R}^n,\Lambda)  \text{. }
\end{align}
The same result holds replacing $\dot{\mathrm{H}}^{s,p}$ by either ${\mathrm{H}}^{s,p}$, ${\mathrm{B}}^{s}_{p,q}$, or even by $\dot{\mathrm{B}}^{s}_{p,q}$, $q\in[1,+\infty]$.

\begin{lemma}\label{lem:ExtOpRn+DiffFormL2}For all $u\in\mathrm{D}_{2}(\mathrm{d},\mathbb{R}^n_+,\Lambda^{k})$ (resp. $\mathrm{D}_{2}({\mathrm{d}}^\ast,\mathbb{R}^n_+,\Lambda^{k})$) we have
\begin{align*}
     \mathrm{E}_{\mathcal{H},\mathfrak{t}} u \in  \mathrm{D}_{2}(\mathrm{d},\mathbb{R}^n,\Lambda^{k}) \quad\text{(resp. } \mathrm{D}_{2}(\delta,\mathbb{R}^n,\Lambda^{k})\text{ )}
\end{align*}
with formulas
\begin{align*}
    \mathrm{d} \mathrm{E}_{\mathcal{H},\mathfrak{t}} u =  \mathrm{E}_{\mathcal{H},\mathfrak{t}} \mathrm{d} u \qquad \text{(resp. } \delta \mathrm{E}_{\mathcal{H},\mathfrak{t}} u =  \mathrm{E}_{\mathcal{H},\mathfrak{t}} \mathrm{d}^\ast u \text{).}
\end{align*}
\end{lemma}

\begin{proof}Let $u\in\mathrm{D}_{2}(\mathrm{d},\mathbb{R}^n_+,\Lambda^{k})$, for $v\in\eus{S}(\mathbb{R}^n,\Lambda^{k+1})$,
\begin{align*}
    \big\langle \mathrm{E}_{\mathcal{H},\mathfrak{t}} u, \delta v\big\rangle_{\mathbb{R}^n} &=  \big\langle u, \delta v\big\rangle_{\mathbb{R}^n_+} + \big\langle \tilde{u}_{\mathfrak{t}}, \delta v\big\rangle_{\mathbb{R}^n_{-}} \\
    &= \big\langle \mathrm{d} u, v\big\rangle_{\mathbb{R}^n_+} + \big\langle (-\mathfrak{e}_n)\wedge u, v\big\rangle_{\partial\mathbb{R}^n_+} + \big\langle \widetilde{\mathrm{d}u}_{\mathfrak{t}},  v\big\rangle_{\mathbb{R}^n_{-}} + \big\langle (\mathfrak{e}_n)\wedge \tilde{u}_{\mathfrak{t}}, v\big\rangle_{\partial\mathbb{R}^n_-}\\
    &= \big\langle \mathrm{d} u, v\big\rangle_{\mathbb{R}^n_+} + \big\langle \widetilde{\mathrm{d}u}_{\mathfrak{t}},  v\big\rangle_{\mathbb{R}^n_{-}} \\
    &= \big\langle \mathrm{E}_{\mathcal{H},\mathfrak{t}} \mathrm{d} u, v\big\rangle_{\mathbb{R}^n} \text{.}
\end{align*}
Which holds, since $(\mathfrak{e}_n)\wedge \tilde{u}_{\mathfrak{t}}(\cdot,0)=(\mathfrak{e}_n)\wedge {u}(\cdot,0)$.

Now, if $u\in\mathrm{D}_{2}(\mathrm{d}^\ast,\mathbb{R}^n_+,\Lambda^{k})$, for $v\in\eus{S}(\mathbb{R}^n,\Lambda^{k-1})$,
\begin{align*}
    \big\langle \mathrm{E}_{\mathcal{H},\mathfrak{t}} u, \mathrm{d} v\big\rangle_{\mathbb{R}^n} &=  \big\langle u, \mathrm{d} v\big\rangle_{\mathbb{R}^n_+} + \big\langle \tilde{u}_{\mathfrak{t}}, \mathrm{d} v\big\rangle_{\mathbb{R}^n_{-}} \\
    &= \big\langle \delta u, v\big\rangle_{\mathbb{R}^n_+} + \big\langle (-\mathfrak{e}_n)\iprod u, v\big\rangle_{\partial\mathbb{R}^n_+} + \big\langle \widetilde{\delta u}_{\mathfrak{t}},  v\big\rangle_{\mathbb{R}^n_{-}} + \big\langle (\mathfrak{e}_n)\iprod \tilde{u}_{\mathfrak{t}}, v\big\rangle_{\partial\mathbb{R}^n_-}\\
    &= \big\langle \delta u, v\big\rangle_{\mathbb{R}^n_+} + \big\langle \widetilde{\delta u}_{\mathfrak{t}},  v\big\rangle_{\mathbb{R}^n_{-}} \\
    &= \big\langle \mathrm{E}_{\mathcal{H},\mathfrak{t}} \mathrm{d}^\ast u, v\big\rangle_{\mathbb{R}^n} \text{.}
\end{align*}
Which holds, since $(\mathfrak{e}_n)\iprod \tilde{u}_{\mathfrak{t}}(\cdot,0)=-(\mathfrak{e}_n)\iprod {u}(\cdot,0)=0$.
\end{proof}

\begin{proof}[of Theorem \ref{thm:HodgeLapL2Rn+}] By the realization of Hodge Laplacian by the mean of the sesquilinear form \eqref{eq:sesqlinformLaplacianHodge}, we have $(\mathrm{D}_2(\Delta_\mathcal{H},\mathbb{R}^n_+,\Lambda^k),-\Delta_\mathcal{H}) = (\mathrm{D}_2(D_\cdot^2,\mathbb{R}^n_+,\Lambda^k),D_\cdot^2)$. Thus, as the square of a self-adjoint $0$-bisectorial operator, the Hodge Laplacian is a non-negative self-adjoint $0$-sectorial operator on $\mathrm{L}^2(\mathbb{R}^n_+,\Lambda^k)$, and it also admits bounded holomorphic functional calculus, see \cite[Theorem~3.2.20]{EgertPhDThesis2015}. In particular, \eqref{eq:ResolvEstHodgeLapL2Rn+} in point \textit{(ii)} holds.

For now, we only consider the case $(\mathrm{D}_2(\Delta_{\mathcal{H},\mathfrak{t}},\mathbb{R}^n_+,\Lambda^k),-\Delta_{\mathcal{H},\mathfrak{t}})$, the proof could be achieved in similar fashion for $(\mathrm{D}_2(\Delta_{\mathcal{H},\mathfrak{n}},\mathbb{R}^n_+,\Lambda^k),-\Delta_{\mathcal{H},\mathfrak{n}})$.

For $\lambda\in\Sigma_{\mu}$, $\mu\in(0,\pi)$, $f\in\mathrm{L}^2(\mathbb{R}^n_+,\Lambda^k)$, we set $U:=(\lambda \mathrm{I}-\Delta)^{-1}\mathrm{E}_{\mathcal{H}}f\in \mathrm{H}^2(\mathbb{R}^n,\Lambda^k)$. By construction, as in the proof of Theorem \ref{thm:HodgeLapL2Rn+} for $I\in\mathcal{I}^k_n$ we have
\begin{align*}
    {U_I}_{|_{\partial\mathbb{R}^n_+}} &= 0, \text{ provided } n\in I\text{, }\\
    {\partial_{x_n}U_I}_{|_{\partial\mathbb{R}^n_+}} &= 0, \text{ provided } n\notin I\text{. }
\end{align*}
Therefore, we obtain first,
\begin{align*}
    \nu\iprod u_{|_{\partial\mathbb{R}^n_+}}=-\mathfrak{e}_n\iprod u_{|_{\partial\mathbb{R}^n_+}} &= (-1)^k\sum_{1\leqslant \ell_1<\ldots<\ell_{k-1}< n} u_{\ell_1 \ell_2\ldots \ell_{k-1} n}(\cdot,0)\, \mathrm{d}x_{\ell_1}\wedge\ldots\wedge\mathrm{d}x_{\ell_{k-1}}\\
     &= (-1)^k\sum_{I'\in\mathcal{I}^{k-1}_{n-1}} u_{I', n}(\cdot,0)\,\mathrm{d}x_{I'} = 0\text{. }
\end{align*}
Similarly, we get that
\begin{align*}
    \nu\iprod \mathrm{d}u_{|_{\partial\mathbb{R}^n_+}}=-\mathfrak{e}_n\iprod \mathrm{d}u_{|_{\partial\mathbb{R}^n_+}} &= -\sum_{j=1}^{n-1}\sum_{1\leqslant \ell_1<\ldots<\ell_{k-1}< n} \partial_{x_j}u_{\ell_1 \ell_2\ldots \ell_{k-1} n}(\cdot,0)\, \mathrm{d}x_{j}\wedge\mathrm{d}x_{\ell_1}\wedge\ldots\wedge\mathrm{d}x_{\ell_{k-1}}\\
    &\quad -\sum_{1\leqslant \ell_1<\ldots<\ell_{k-1}<\ell_{k}<n} \partial_{x_n}u_{\ell_1 \ell_2\ldots \ell_{k} }(\cdot,0)\,\mathrm{d}x_{\ell_1}\wedge\ldots\wedge\mathrm{d}x_{\ell_{k}}\\
     &= 0\text{. }
\end{align*}
From above calculations, we deduce that $u:=U_{|_{\mathbb{R}^n_+}}$ is such that $u\in\mathrm{H}^2(\mathbb{R}^n_+,\Lambda^k)\cap\mathrm{D}_2(\Delta_{\mathcal{H},\mathfrak{t}},\mathbb{R}^n_+,\Lambda^k)$, and 
\begin{align*}
    \lambda u -\Delta u = f \,\text{ in }\, \mathbb{R}^n_+ \text{.}
\end{align*}
Hence, by uniqueness $(\lambda \mathrm{I}-\Delta_{\mathcal{H}})^{-1}f=[(\lambda \mathrm{I}-\Delta)^{-1}\mathrm{E}_{\mathcal{H}}f]_{\mathbb{R}^n_+}$. One may conclude following the lines of the proof of \cite[Proposition~4.1]{Gaudin2022}, using Lemma \ref{lem:ExtOpRn+DiffFormL2}.
\end{proof}

\begin{lemma}\label{lem:commutrelationL2HodgeRn+}Provided $k\in\llb 0,n\rrb$, $\mu\in(0,\pi)$, $\lambda\in\Sigma_\mu$, the following commutation identities hold,
\begin{enumerate}
    \item $\mathbb{P}(\lambda \mathrm{I} - \Delta_{\mathcal{H},\mathfrak{t}})^{-1}f = (\lambda \mathrm{I} - \Delta_{\mathcal{H},\mathfrak{t}})^{-1}\mathbb{P}f$,\qquad for all $f\in\mathrm{L}^2(\mathbb{R}^n_+,\Lambda^k)$ ;
    \item $\mathrm{d}(\lambda \mathrm{I} - \Delta_{\mathcal{H},\mathfrak{t}})^{-1}f = (\lambda \mathrm{I} - \Delta_{\mathcal{H},\mathfrak{t}})^{-1}\mathrm{d}f$,\qquad for all $f\in\mathrm{D}_2(\mathrm{d},\mathbb{R}^n_+,\Lambda^k)$ ;
    \item $\mathrm{d}^\ast(\lambda \mathrm{I} - \Delta_{\mathcal{H},\mathfrak{t}})^{-1}f = (\lambda \mathrm{I} - \Delta_{\mathcal{H},\mathfrak{t}})^{-1}\mathrm{d}^\ast f$,\qquad for all $f\in\mathrm{D}_2(\mathrm{d}^\ast,\mathbb{R}^n_+,\Lambda^k)$.
\end{enumerate}
Every above identities still hold replacing $(\mathfrak{t},\mathbb{P},\mathrm{d},\mathrm{d}^\ast)$ by $(\mathfrak{n},\mathbb{Q},\delta^\ast,\delta)$.
\end{lemma}

\begin{remark}Lemma \ref{lem:commutrelationL2HodgeRn+} does not depend on the domain $\Omega=\mathbb{R}^n_+$ since its proof only relies on the use of the sesquilinear form associated with the Hodge Laplacian.
\end{remark}

\begin{proof} {\textbf{Step 0 :}} For all $u,v\in\mathrm{D}_2(D_\mathfrak{t},\mathbb{R}^n_+,\Lambda)$,
\begin{align*}
    \mathfrak{a}_{\mathcal{H},\mathfrak{t}}(\mathbb{P}u,v) &= \big\langle \mathrm{d} \mathbb{P}u, \mathrm{d}v \big\rangle_{\mathbb{R}^n_+} + \big\langle \mathrm{d}^\ast \mathbb{P}u, \mathrm{d}^\ast v \big\rangle_{\mathbb{R}^n_+}\\
    &= \big\langle \mathrm{d} u, \mathrm{d}v \big\rangle_{\mathbb{R}^n_+}\\
    &= \big\langle \mathrm{d} u, \mathrm{d}\mathbb{P}v \big\rangle_{\mathbb{R}^n_+}\\
    &= \big\langle \mathrm{d} u, \mathrm{d}\mathbb{P}v \big\rangle_{\mathbb{R}^n_+} + \big\langle \mathrm{d}^\ast u, \mathrm{d}^\ast \mathbb{P}v \big\rangle_{\mathbb{R}^n_+}\\
    &=\mathfrak{a}_{\mathcal{H},\mathfrak{t}}(u,\mathbb{P}v)\text{. }
\end{align*}

{\textbf{Step 1 :}} Let $\mu\in(0,\pi)$, $\lambda\in\Sigma_\mu$, and $f\in\mathrm{L}^2(\mathbb{R}^n_+,\Lambda^k)$, we set $u:=(\lambda \mathrm{I}-\Delta_{\mathcal{H},\mathfrak{t}})^{-1}f$, then for all $v\in\mathrm{D}_2(D_\mathfrak{t},\mathbb{R}^n_+,\Lambda)$,
\begin{align*}
    \lambda \big\langle \mathbb{P}u,v\big\rangle_{\mathbb{R}^n_+} +\mathfrak{a}_{\mathcal{H},\mathfrak{t}}(\mathbb{P}u,v) &= \lambda \big\langle u,\mathbb{P}v\big\rangle_{\mathbb{R}^n_+} +\mathfrak{a}_{\mathcal{H},\mathfrak{t}}(u,\mathbb{P}v) \\
    &= \big\langle f,\mathbb{P}v\big\rangle_{\mathbb{R}^n_+}\\
    &= \big\langle \mathbb{P}f,v\big\rangle_{\mathbb{R}^n_+}\text{. }
\end{align*}
Hence, by uniqueness of the solution to the resolvent problem in $\mathrm{L}^2(\mathbb{R}^n_+,\Lambda)$, we deduce $\mathbb{P}u = (\lambda \mathrm{I}-\Delta_{\mathcal{H},\mathfrak{t}})^{-1}\mathbb{P}f$.

{\textbf{Step 2 :}} We use the same notations as the ones introduced in \textbf{Step 1} but we assume that $f\in\mathrm{D}_2(\mathrm{d},\mathbb{R}^n_+,\Lambda^k)$. For all $v\in\mathrm{D}_2(\Delta_{\mathcal{H},\mathfrak{t}},\mathbb{R}^n_+,\Lambda)$,
\begin{align*}
    \big\langle \mathrm{d}f,v\big\rangle_{\mathbb{R}^n_+} &= \big\langle f,\mathrm{d}^\ast v\big\rangle_{\mathbb{R}^n_+}\\
    &= \lambda \big\langle u,\mathrm{d}^\ast v\big\rangle_{\mathbb{R}^n_+} +\mathfrak{a}_{\mathcal{H},\mathfrak{t}}(u, \mathrm{d}^\ast v)\\
    &= \lambda \big\langle u,\mathrm{d}^\ast v\big\rangle_{\mathbb{R}^n_+} +\big\langle \mathrm{d}u,\mathrm{d}\mathrm{d}^\ast v\big\rangle_{\mathbb{R}^n_+} +\big\langle \mathrm{d}^\ast u,(\mathrm{d}^\ast)^2 v\big\rangle_{\mathbb{R}^n_+}\\
    &= \lambda \big\langle \mathrm{d} u, v\big\rangle_{\mathbb{R}^n_+} +\big\langle \mathrm{d}^\ast\mathrm{d}u,\mathrm{d}^\ast v\big\rangle_{\mathbb{R}^n_+}\\
    &= \lambda \big\langle \mathrm{d} u, v\big\rangle_{\mathbb{R}^n_+} +\big\langle \mathrm{d}\mathrm{d}u,\mathrm{d} v\big\rangle_{\mathbb{R}^n_+}+\big\langle \mathrm{d}^\ast\mathrm{d}u,\mathrm{d}^\ast v\big\rangle_{\mathbb{R}^n_+}\\
    &= \lambda \big\langle \mathrm{d}u, v\big\rangle_{\mathbb{R}^n_+} +\mathfrak{a}_{\mathcal{H},\mathfrak{t}}(\mathrm{d}u, v)\text{. }
\end{align*}
Since the continuous embedding $\mathrm{D}_2(\Delta_{\mathcal{H},\mathfrak{t}},\mathbb{R}^n_+,\Lambda)\hookrightarrow \mathrm{D}_2(D_{\mathfrak{t}},\mathbb{R}^n_+,\Lambda)$ is dense, the above equality still holds for all $v\in \mathrm{D}_2(D_{\mathfrak{t}},\mathbb{R}^n_+,\Lambda)$. Hence, by uniqueness of the solution to the resolvent problem in $\mathrm{L}^2(\mathbb{R}^n_+,\Lambda)$, we obtain $\mathrm{d}u = (\lambda \mathrm{I}-\Delta_{\mathcal{H},\mathfrak{t}})^{-1}\mathrm{d}f$. The proof ends here since all remaining results can be proven similarly.
\end{proof}

\begin{lemma}\label{lem:RiesztransfHodgeL2Rn+}Let $k\in\llb 0,n\rrb$, following operators
\begin{align*}
    \mathrm{d}(-\Delta_{\mathcal{H},\mathfrak{t}})^{-\frac{1}{2}}\,&:\,\mathrm{L}^2(\mathbb{R}^{n}_+,\Lambda^k)\longrightarrow \mathrm{N}_2(\mathrm{d},\mathbb{R}^{n}_+,\Lambda^{k+1})\, ;\\
    \mathrm{d}^\ast(-\Delta_{\mathcal{H},\mathfrak{t}})^{-\frac{1}{2}}\,&:\,\mathrm{L}^2(\mathbb{R}^{n}_+,\Lambda^k)\longrightarrow \mathrm{N}_2(\mathrm{d}^\ast,\mathbb{R}^{n}_+,\Lambda^{k-1})\, ;
\end{align*}
are well defined bounded linear operators on $\mathrm{L}^2(\mathbb{R}^{n}_+,\Lambda)$, and are each-other's adjoint.

Everything still holds replacing $(\mathfrak{t},\mathrm{d},\mathrm{d}^\ast)$ by $(\mathfrak{n},\delta^\ast,\delta)$.
\end{lemma}

\begin{proof} We  prove the $\mathrm{L}^2(\mathbb{R}^n_+,\Lambda)$-boundedness of $\mathrm{d}(-\Delta_{\mathcal{H},\mathfrak{t}})^{-\frac{1}{2}}$ and compute its adjoint.

We use bounded holomorphic functional calculus of $D_\mathfrak{t}$ on $\mathrm{L}^2(\mathbb{R}^n_+,\Lambda)$ provided by Proposition~\ref{prop:HodgeDiracL2Rn+}. By the mean of $z\mapsto \frac{z}{\sqrt{\lambda +z^2}}$, and the boundedness of $\mathbb{P}$, we have, for all $\lambda\geqslant0$, and all  $f\in\mathrm{L}^2(\mathbb{R}^n_+,\Lambda^k)$
\begin{align*}
    \lVert \mathrm{d}(\lambda\mathrm{I}-\Delta_{\mathcal{H},\mathfrak{t}})^{-\frac{1}{2}} f \rVert_{\mathrm{L}^2(\mathbb{R}^n_+)} &\leqslant \lVert f \rVert_{\mathrm{L}^2(\mathbb{R}^n_+)} \text{ ; }\\
    \lVert \mathrm{d}^\ast(\lambda\mathrm{I}-\Delta_{\mathcal{H},\mathfrak{t}})^{-\frac{1}{2}} f \rVert_{\mathrm{L}^2(\mathbb{R}^n_+)} &\leqslant \lVert f \rVert_{\mathrm{L}^2(\mathbb{R}^n_+)} \text{. }
\end{align*}

The adjoint of $\mathrm{d}(\lambda\mathrm{I}-\Delta_{\mathcal{H},\mathfrak{t}})^{-\frac{1}{2}}$ provided $\lambda>0$, is $(\lambda\mathrm{I}-\Delta_{\mathcal{H},\mathfrak{t}})^{-\frac{1}{2}}\mathrm{d}^\ast=\mathrm{d}^\ast(\lambda\mathrm{I}-\Delta_{\mathcal{H},\mathfrak{t}})^{-\frac{1}{2}}$ up to a dense subset of $\mathrm{L}^2(\mathbb{R}^n_+,\Lambda)$ (here $\mathrm{D}_2(\mathrm{d}^\ast,\mathbb{R}^n_+,\Lambda)$). Thanks to Lemma \ref{lem:commutrelationL2HodgeRn+}, we can pass to the limit as $\lambda$ goes to $0$ in the $\mathrm{L}^2$ inner product yielding
\begin{align*}
    (\mathrm{d}(-\Delta_{\mathcal{H},\mathfrak{t}})^{-\frac{1}{2}})^\ast = \mathrm{d}^\ast(-\Delta_{\mathcal{H},\mathfrak{t}})^{-\frac{1}{2}} \text{. }
\end{align*}
\end{proof}

We can summarize
\begin{theorem}\label{thm:HodgeDecompL2Rn+Full} Let $k\in\llb 0,n\rrb$, the following assertions are true
\begin{enumerate}
    \item  The following equality holds
    \begin{align*}
        {\mathrm{N}}_{2}(\mathrm{d},\mathbb{R}^n_+,\Lambda^k) = \overline{{\mathrm{R}}_{2}(\mathrm{d},\mathbb{R}^n_+,\Lambda^k)}^{\lVert\cdot\rVert_{{\mathrm{L}}^{2}(\mathbb{R}^n_+)}}\text{, }
    \end{align*}
    and still holds replacing $\mathrm{d}$ by $\mathrm{d}^\ast$.
    \item The (generalized) Helmholtz-Leray projector $\mathbb{P}\,:\, {\mathrm{L}}^{2}(\mathbb{R}^n_+,\Lambda^k)\longrightarrow {\mathrm{N}}^{2}(\mathrm{d}^\ast,\mathbb{R}^n_+,\Lambda^k)$ satisfies the identity
    \begin{align*}
        \mathbb{P} = \mathrm{I} - \mathrm{d}(-\Delta_{\mathcal{H},\mathfrak{t}})^{-\frac{1}{2}}\mathrm{d}^\ast(-\Delta_{\mathcal{H},\mathfrak{t}})^{-\frac{1}{2}} \text{.}
    \end{align*}
    \item The following Hodge decomposition holds
    \begin{align*}
        {\mathrm{L}}^{2}(\mathbb{R}^n_+,\Lambda^k) = {\mathrm{N}}^{2}(\mathrm{d}^\ast,\mathbb{R}^n_+,\Lambda^k) \overset{\perp}{\oplus} {\mathrm{N}}^{2}(\mathrm{d},\mathbb{R}^n_+,\Lambda^k)\text{. }
    \end{align*}
\end{enumerate}
Moreover, the result remain true if we replace $(\mathfrak{t},\mathrm{d},\mathrm{d}^\ast,\mathbb{P})$ by $(\mathfrak{n},\delta^{\ast},\delta,\mathbb{Q})$.
\end{theorem}

\begin{remark} The Theorem \ref{thm:HodgeDecompL2Rn+Full} and the whole construction of this section mainly depends on the injectivity of the Laplacian : the built is done via resolvent approximation, abstract functional calculus provided by the Hilbertian structure of $\mathrm{L}^2(\mathbb{R}^n_+,\Lambda)$ and the self-adjointness of the Laplacian. Therefore, such construction and the proof do not depend on the open set $\Omega=\mathbb{R}^n_+$.

To be more precise, the above Theorem \ref{thm:HodgeDecompL2Rn+Full} should remain true for all openset $\Omega$, say at least Lipschitz, which have no harmonic forms. In the case of a bounded domain : the theorem remains true whenever all its Betti numbers vanish.
\end{remark}

\begin{proof}{\textbf{Step 1 :}} Identity for $\mathbb{P}$.

From boundedness of above operators we deduce that the new operator $\overline{\mathbb{P}}$ defined for all $f\in\mathrm{L}^2(\mathbb{R}^n_+,\Lambda)$ by 
\begin{align*}
    \overline{\mathbb{P}}f := f-\mathrm{d}(-\Delta_{\mathcal{H},\mathfrak{t}})^{-\frac{1}{2}}\mathrm{d}^\ast(-\Delta_{\mathcal{H},\mathfrak{t}})^{-\frac{1}{2}} f \text{, }
\end{align*}
is well defined and bounded on $\mathrm{L}^2(\mathbb{R}^n_+,\Lambda)$. We are going to check that $\overline{\mathbb{P}}$ is an orthogonal projector, hence, firstly a projector.
\begin{align*}
    \overline{\mathbb{P}}^2 f &= \overline{\mathbb{P}}f - \mathrm{d}(-\Delta_{\mathcal{H},\mathfrak{t}})^{-\frac{1}{2}}\mathrm{d}^\ast(-\Delta_{\mathcal{H},\mathfrak{t}})^{-\frac{1}{2}} \overline{\mathbb{P}}f\\
    &= \overline{\mathbb{P}}f - \mathrm{d}(-\Delta_{\mathcal{H},\mathfrak{t}})^{-\frac{1}{2}}\mathrm{d}^\ast(-\Delta_{\mathcal{H},\mathfrak{t}})^{-\frac{1}{2}}  f + [\mathrm{d}(-\Delta_{\mathcal{H},\mathfrak{t}})^{-\frac{1}{2}}\mathrm{d}^\ast(-\Delta_{\mathcal{H},\mathfrak{t}})^{-\frac{1}{2}}]^2  f\\
    &= \lim_{\lambda\rightarrow 0} \left( \overline{\mathbb{P}}f - \mathrm{d}^\ast\mathrm{d}(\lambda \mathrm{I}-\Delta_{\mathcal{H},\mathfrak{t}})^{-1} f + [\mathrm{d}^\ast\mathrm{d}(\lambda \mathrm{I}-\Delta_{\mathcal{H},\mathfrak{t}})^{-1}]^2 f \right)\\
    &= \lim_{\lambda\rightarrow 0} \left( \overline{\mathbb{P}}f - \lambda \mathrm{d}^\ast\mathrm{d}(\lambda \mathrm{I}-\Delta_{\mathcal{H},\mathfrak{t}})^{-2}f \right)  \\
    &= \overline{\mathbb{P}}f\text{. }
\end{align*}
By construction and by Lemma \ref{lem:RiesztransfHodgeL2Rn+}, $\overline{\mathbb{P}}$ is self-adjoint, hence orthogonal.

For all $f\in \mathrm{D}_2(\mathrm{d}^\ast,\mathbb{R}^n_+,\Lambda)$,
\begin{align*}
    \mathrm{d}^\ast\overline{\mathbb{P}}f &= \lim_{\lambda\rightarrow 0} \left(\mathrm{d}^\ast f - \mathrm{d}^\ast \mathrm{d}(\lambda \mathrm{I}-\Delta_{\mathcal{H},\mathfrak{t}})^{-\frac{1}{2}}\mathrm{d}^\ast(\lambda \mathrm{I}-\Delta_{\mathcal{H},\mathfrak{t}})^{-\frac{1}{2}} f\right)\\
    &=\lim_{\lambda\rightarrow 0} \left( \mathrm{d}^\ast f + \Delta_{\mathcal{H},\mathfrak{t}}(\lambda \mathrm{I}-\Delta_{\mathcal{H},\mathfrak{t}})^{-1} \mathrm{d}^\ast f\right)\\
    &=\lim_{\lambda\rightarrow 0} \left(\mathrm{d}^\ast f -  \mathrm{d}^\ast f\right)=0.
\end{align*}
Thus, since the embedding $\mathrm{D}_2(\mathrm{d}^\ast,\mathbb{R}^n_+,\Lambda)\hookrightarrow \mathrm{L}^2(\mathbb{R}^n_+,\Lambda)$ is dense, it follows that 
\begin{align*}
    \mathrm{R}_2(\overline{\mathbb{P}},\mathbb{R}^n_+,\Lambda)\subset \mathrm{N}_2(\mathrm{d}^\ast,\mathbb{R}^n_+,\Lambda)\text{.}
\end{align*}

For all $f\in \mathrm{N}_2(\mathrm{d}^\ast,\mathbb{R}^n_+,\Lambda)$
\begin{align*}
    \overline{\mathbb{P}}f &= \lim_{\lambda\rightarrow 0}\left( f - \mathrm{d}(\lambda \mathrm{I}-\Delta_{\mathcal{H},\mathfrak{t}})^{-\frac{1}{2}}\mathrm{d}^\ast(\lambda \mathrm{I}-\Delta_{\mathcal{H},\mathfrak{t}})^{-\frac{1}{2}} f\right)\\
    &=\lim_{\lambda\rightarrow 0} \left(f - \mathrm{d}(\lambda \mathrm{I}-\Delta_{\mathcal{H},\mathfrak{t}})^{-1} \mathrm{d}^\ast f\right)\\
    &=\lim_{\lambda\rightarrow 0} \left(f + 0\right)=f.
\end{align*}
Hence, $\mathbb{P}_{|_{\mathrm{N}_2(\mathrm{d}^\ast,\mathbb{R}^n_+,\Lambda)}} =\mathrm{I}$.

By construction, we also have $\mathrm{R}_2(\mathrm{I}-\overline{\mathbb{P}})=\mathrm{R}_2(\mathrm{d}(-\Delta_{\mathcal{H},\mathfrak{t}})^{-\frac{1}{2}}\mathrm{d}^\ast(-\Delta_{\mathcal{H},\mathfrak{t}})^{-\frac{1}{2}})\subset \overline{\mathrm{R}_2(\mathrm{d},\mathbb{R}^n_+,\Lambda)}$ and obviously $\overline{\mathbb{P}}^\ast = \overline{\mathbb{P}}$, so that by uniqueness of the orthogonal projection on $\mathrm{N}_2(\mathrm{d}^\ast,\mathbb{R}^n_+,\Lambda)$, $\overline{\mathbb{P}}=\mathbb{P}$.

{\textbf{Step 2 :}}  We notice first that the inclusion $\overline{\mathrm{R}_2(\mathrm{d},\mathbb{R}^n_+,\Lambda^k)}^{\lVert\cdot\rVert_{\mathrm{L}^2(\mathbb{R}^n_+)}} \subset \mathrm{N}_2(\mathrm{d},\mathbb{R}^n_+,\Lambda^k)$ is true.

Now, for the reverse inclusion let $f\in \mathrm{N}_2(\mathrm{d},\mathbb{R}^n_+,\Lambda^k)$, we have 
\begin{align*}
    [\mathrm{I}-\mathbb{P}]f &= \lim_{\lambda \rightarrow 0} \mathrm{d}\mathrm{d}^\ast(\lambda \mathrm{I}-\Delta_{\mathcal{H},\mathfrak{t}})^{-1} f\\
    &= \lim_{\lambda \rightarrow 0} -\Delta_{\mathcal{H},\mathfrak{t}}(\lambda \mathrm{I}-\Delta_{\mathcal{H},\mathfrak{t}})^{-1} f\\
    &= f\text{. }
\end{align*}
By construction, for all $\lambda>0$, we have $\mathrm{d}\mathrm{d}^\ast(\lambda \mathrm{I}-\Delta_{\mathcal{H},\mathfrak{t}})^{-1} f\in \mathrm{R}_2(\mathrm{d},\mathbb{R}^n_+,\Lambda^{k-1})$, so that the reverse inclusion $\mathrm{N}_2(\mathrm{d},\mathbb{R}^n_+,\Lambda^k)\subset \overline{\mathrm{R}_2(\mathrm{d},\mathbb{R}^n_+,\Lambda^k)}^{\lVert\cdot\rVert_{\mathrm{L}^2(\mathbb{R}^n_+)}} $ holds.
\end{proof}

\subsubsection {\texorpdfstring{$\dot{\mathrm{H}}^{s,p}$}{Hsp} and \texorpdfstring{$\dot{\mathrm{B}}^{s}_{p,q}$}{Bspq}-theory for Hodge Laplacians and the Hodge decomposition}

We start this new subsection claiming about closedness of the exterior and interior derivatives with and without $0$ boundary conditions. The two following lemmas are straightforward.
\begin{lemma}Let $p\in(1,+\infty)$, $s\in(-1+1/p,1/p)$, $k\in\llb 0,n \rrb$. With the same notations as in Lemma \ref{lem:closdenessNdualityDerivativesL2Rn+}, the operators
\begin{align*}
    (\dot{\mathrm{D}}^{s}_{p}(\mathrm{d},\mathbb{R}^n_+,\Lambda^k),\mathrm{d}) \text{ and } (\dot{\mathrm{D}}^{s}_{p}(\underline{\delta},\mathbb{R}^n_+,\Lambda^k),{\delta})
\end{align*}
are densely defined closed operators on $\dot{\mathrm{H}}^{s,p}(\mathbb{R}^n_+,\Lambda)$.

Moreover, 
\begin{itemize}
    \item the result still holds replacing $(\dot{\mathrm{H}}^{s,p},\dot{\mathrm{D}}^{s}_{p})$ by either $({\mathrm{H}}^{s,p},{\mathrm{D}}^{s}_{p})$, $(\mathrm{B}^{s}_{p,q},\mathrm{D}^{s}_{p,q})$ or $(\dot{\mathrm{B}}^{s}_{p,q},\dot{\mathrm{D}}^{s}_{p,q})$, with $q\in[1,+\infty)$ ;

    \item in case of $(\mathrm{B}^{s}_{p,\infty},\mathrm{D}^{s}_{p,\infty})$ and $(\dot{\mathrm{B}}^{s}_{p,\infty},\dot{\mathrm{D}}^{s}_{p,\infty})$ above operators are only weak${}^\ast$ densely defined, strongly closed operators;
    
    \item all the above results remain true exchanging the roles of $\mathrm{d}$ and $\delta$.
\end{itemize}
\end{lemma}

\begin{lemma}Let $p\in(1,+\infty)$, $s\in(-1+1/p,1/p)$, $k\in\llb 0,n \rrb$. With the same notations as in Lemma \ref{lem:closdenessNdualityDerivativesL2Rn+}, the dual operator of $(\dot{\mathrm{D}}^{s}_{p}(\mathrm{d},\mathbb{R}^n_+,\Lambda),\mathrm{d})$ on $\dot{\mathrm{H}}^{s,p}(\mathbb{R}^n_+,\Lambda)$, is
\begin{align*}
     (\dot{\mathrm{D}}^{-s}_{p'}({\mathrm{d}}^\ast,\mathbb{R}^n_+,\Lambda),{\mathrm{d}}^\ast) = (\dot{\mathrm{D}}^{-s}_{p'}(\underline{\delta},\mathbb{R}^n_+,\Lambda),{\delta}) 
\end{align*}
as an operator on $\dot{\mathrm{H}}^{-s,p'}(\mathbb{R}^n_+,\Lambda)$.

Moreover,
\begin{itemize}
    \item the result still holds replacing $(\dot{\mathrm{H}}^{s,p},\dot{\mathrm{D}}^{s}_{p},\dot{\mathrm{H}}^{-s,p'},\dot{\mathrm{D}}^{-s}_{p'})$ by $(\dot{\mathrm{B}}^{s}_{p,q},\dot{\mathrm{D}}^{s}_{p,q},\dot{\mathrm{B}}^{-s}_{p',q'},\dot{\mathrm{D}}^{-s}_{p',q'})$ with $q\in[1,+\infty)$ ;
    
    \item we may replace $(\dot{\mathrm{D}},\dot{\mathrm{H}},\dot{\mathrm{B}})$ by $(\mathrm{D},\mathrm{H},\mathrm{B})$.
    
    \item all the above results remain true exchanging the roles of $\mathrm{d}$ and $\delta$.
\end{itemize}
\end{lemma}

\begin{remark} Notice that talking about $\mathrm{D}^{s}_{p}(\underline{\mathrm{d}},\mathbb{R}^n_+,\Lambda^k)$ in above lemmas, with respect to notation introduced in Lemma \ref{lem:closdenessNdualityDerivativesL2Rn+}, in particular the involved $0$-boundary condition, actually makes sense, thanks to Theorem \ref{thm:TracesDifferentialformsHomSpaces}.
\end{remark}

Before we start our investigation of Hodge Laplacians and the Hodge decomposition, we need to show the closedness of Hodge-Dirac operators. In order to verify such a property, the next result will be of paramount importance, to reproduce the behavior obtained in the $\mathrm{L}^2$ setting on other scales of function spaces. We mention that many results presented here will strongly depends on the fact that the considered openset is $\mathbb{R}^n_+$ (mainly Lemma \ref{lem:ExtOpRn+DiffFormHspBspq}, and point (ii) of Theorem \ref{thm:MetaThm1HodgeLaplacianRn+} which are widely used to construct other results of the present section).

The proof of the next lemma is identical to the one of Lemma \ref{lem:ExtOpRn+DiffFormL2}.

\begin{lemma}\label{lem:ExtOpRn+DiffFormHspBspq}Let $p\in(1,+\infty)$, $s\in(-1+1/p,1/p)$, $k\in\llb 0,n \rrb$. For all $u\in\dot{\mathrm{D}}^{s}_{p}(\mathrm{d},\mathbb{R}^n_+,\Lambda^{k})$ (resp. $\dot{\mathrm{D}}^{s}_{p}({\mathrm{d}}^\ast,\mathbb{R}^n_+,\Lambda^{k})$) we have
\begin{align*}
     \mathrm{E}_{\mathcal{H},\mathfrak{t}} u \in  \dot{\mathrm{D}}^{s}_{p}(\mathrm{d},\mathbb{R}^n,\Lambda^{k}) \quad\text{(resp. } \dot{\mathrm{D}}^{s}_{p}(\delta,\mathbb{R}^n,\Lambda^{k})\text{ )}
\end{align*}
with formulas
\begin{align*}
    \mathrm{d} \mathrm{E}_{\mathcal{H},\mathfrak{t}} u =  \mathrm{E}_{\mathcal{H},\mathfrak{t}} \mathrm{d} u \qquad \text{(resp. } \delta \mathrm{E}_{\mathcal{H},\mathfrak{t}} u =  \mathrm{E}_{\mathcal{H},\mathfrak{t}} \mathrm{d}^\ast u \text{).}
\end{align*}

Moreover,
\begin{itemize}
    \item the result still holds replacing $\dot{\mathrm{D}}^{s}_{p}$ by $\dot{\mathrm{D}}^{s}_{p,q}$ with $q\in[1,+\infty]$ ;
    
    \item we may replace $\dot{\mathrm{D}}$ by $\mathrm{D}$.
    
    \item all the above results remain true exchanging the roles of $\mathrm{d}$ and $\delta$, and replacing $\mathfrak{t}$ by $\mathfrak{n}$.
\end{itemize}
\end{lemma}

\begin{proposition}Let $p\in(1,+\infty)$, $s\in(-1+1/p,1/p)$, $k\in\llb 0,n \rrb$. The Hodge-Dirac operator
\begin{align*}
    (\dot{\mathrm{D}}^{s}_{p}(D_\mathfrak{t},\mathbb{R}^n_+,\Lambda^k),D_\mathfrak{t})=(\dot{\mathrm{D}}^{s}_{p}(\mathrm{d},\mathbb{R}^n_+,\Lambda^k)\cap \dot{\mathrm{D}}^{s}_{p}(\mathrm{d}^\ast,\mathbb{R}^n_+,\Lambda^k),\mathrm{d}+\mathrm{d}^\ast)
\end{align*}
is a densely defined closed operator $\dot{\mathrm{H}}^{s,p}(\mathbb{R}^n_+,\Lambda)$.

Moreover,
\begin{itemize}
    \item the result still holds replacing $(\dot{\mathrm{H}}^{s,p},\dot{\mathrm{D}}^{s}_{p})$ by either $({\mathrm{H}}^{s,p},{\mathrm{D}}^{s}_{p})$, $(\mathrm{B}^{s}_{p,q},\mathrm{D}^{s}_{p,q})$ or $(\dot{\mathrm{B}}^{s}_{p,q},\dot{\mathrm{D}}^{s}_{p,q})$, with $q\in[1,+\infty)$ ;

    \item in case of $(\mathrm{B}^{s}_{p,\infty},\mathrm{D}^{s}_{p,\infty})$ and $(\dot{\mathrm{B}}^{s}_{p,\infty},\dot{\mathrm{D}}^{s}_{p,\infty})$ above Hodge-Dirac operator is only weak${}^\ast$ densely defined, and strongly closed;
    
    \item all above results remain true replacing $(\mathfrak{t},\mathrm{d})$ by $(\mathfrak{n},\delta)$.
\end{itemize}
\end{proposition}

\begin{proof}Let $(u_j)_{j\in\mathbb{N}}\subset \dot{\mathrm{D}}^{s}_{p}(\mathrm{d},\mathbb{R}^n_+,\Lambda^k)\cap \dot{\mathrm{D}}^{s}_{p}(\mathrm{d}^\ast,\mathbb{R}^n_+,\Lambda^k)$, and $(u,v)\in\dot{\mathrm{H}}^{s,p}(\mathbb{R}^n_+,\Lambda^k)\times\dot{\mathrm{H}}^{s,p}(\mathbb{R}^n_+, \Lambda)$ satisfying
\begin{align*}
    u_j \xrightarrow[j\rightarrow +\infty]{} u \quad\text{ and }\quad D_\mathfrak{t} u_j \xrightarrow[j\rightarrow +\infty]{} v \text{ in }  \dot{\mathrm{H}}^{s,p}(\mathbb{R}^n_+)\text{.} 
\end{align*}
We set for all $j\in\mathbb{N}$, $U_j:=\mathrm{E}_{\mathcal{H},\mathfrak{t}}u_j$, $U:=\mathrm{E}_{\mathcal{H},\mathfrak{t}}u$. By Lemma \ref{lem:ExtOpRn+DiffFormHspBspq}, we have for all $j\in\mathbb{Z}$ $U_j\in\dot{\mathrm{D}}^{s}_{p}(\mathrm{d},\mathbb{R}^n,\Lambda^k)\cap \dot{\mathrm{D}}^{s}_{p}(\delta,\mathbb{R}^n,\Lambda^k)$
\begin{align*}
    D U_j = \mathrm{E}_{\mathcal{H},\mathfrak{t}}D_\mathfrak{t} u_j\text{. }
\end{align*}
We also have,
\begin{align*}
    D U_j \xrightarrow[j\rightarrow +\infty]{}  V:=\mathrm{E}_{\mathcal{H},\mathfrak{t}}v\text{ in }  \dot{\mathrm{H}}^{s,p}(\mathbb{R}^n)\text{.} 
\end{align*}
By the Hodge decomposition on $\mathbb{R}^n$, check Theorem \ref{thm:HodgeDecompRn}, there exists a unique couple $({V}_{0},{V}_{1})\in \overline{\dot{\mathrm{R}}^{s}_{p}({\mathrm{d}},\mathbb{R}^n,\Lambda)}\times \dot{\mathrm{N}}^{s}_{p}(\delta,\mathbb{R}^n,\Lambda)$ such that $V={V}_{0}+{V}_{1}$.
Since $U_j$ goes to $U$ in $\dot{\mathrm{H}}^{s,p}(\mathbb{R}^n,\Lambda^k)$, by continuity of involved projectors, and uniqueness of decomposition, it follows that
\begin{align*}
    \mathrm{d} U_j \xrightarrow[j\rightarrow +\infty]{}  V_0  \quad\text{ and }\quad \delta U_j \xrightarrow[j\rightarrow +\infty]{} V_1   \text{ in }  \dot{\mathrm{H}}^{s,p}(\mathbb{R}^n)\text{.} 
\end{align*}
In particular, if we set $v_\ell:= {V_\ell}_{|_{\mathbb{R}^n_+}}$ for $\ell\in\{0,1\}$, we necessarily have by restriction
\begin{align*}
    \mathrm{d} u_j \xrightarrow[j\rightarrow +\infty]{}  v_0  \quad\text{ and }\quad \delta u_j \xrightarrow[j\rightarrow +\infty]{} v_1 \text{ in }  \dot{\mathrm{H}}^{s,p}(\mathbb{R}^n_+)\text{.} 
\end{align*}
But $(u_j)_{j\in\mathbb{N}}$ converge to $u$ in $\dot{\mathrm{H}}^{s,p}(\mathbb{R}^n_+,\Lambda^k)$, so in particular in distributional sense. Thus necessarily $(v_0,v_1)=(\mathrm{d} u,\delta u)$ and $v=D u$. By continuity of trace provided by Theorem \ref{thm:TracesDifferentialformsInhomSpaces}, we also have $\nu\iprod u_{|_{\partial\mathbb{R}^n_+}} =0$, \textit{ i.e. } $(\dot{\mathrm{D}}^{s}_p(D_{\mathfrak{t}},\mathbb{R}^n_+,\Lambda^k),D_{\mathfrak{t}})$ is a closed operator on $\dot{\mathrm{H}}^{s,p}(\mathbb{R}^n_+,\Lambda^k)$. The proof ends here since one can reproduce all above arguments for $(\dot{\mathrm{D}}^{s}_p(D_{\mathfrak{n}},\mathbb{R}^n_+,\Lambda^k),D_{\mathfrak{n}})$, and also for all other kind of function spaces.
\end{proof}

The next result about closedness of Hodge Laplacian admits a similar proof

\begin{proposition}\label{prop:ExtOpHodgeLapRn+}Let $p\in(1,+\infty)$, $s\in(-1+1/p,1/p)$, $k\in\llb 0,n \rrb$. The Hodge Laplacian
\begin{align*}
    (\dot{\mathrm{D}}^{s}_{p}(\Delta_{\mathcal{H},\mathfrak{t}},\mathbb{R}^n_+,\Lambda^k),-\Delta_{\mathcal{H},\mathfrak{t}})=( \dot{\mathrm{D}}^{s}_{p}(D_\mathfrak{t}^2,\mathbb{R}^n_+,\Lambda^k),D_\mathfrak{t}^2)
\end{align*}
is a densely defined closed injective operator on $\dot{\mathrm{H}}^{s,p}(\mathbb{R}^n_+,\Lambda)$. For all $u\in \dot{\mathrm{D}}^{s}_{p}(\Delta_{\mathcal{H},\mathfrak{t}},\mathbb{R}^n_+,\Lambda^k)$, the following formula holds
\begin{align*}
    -\Delta \mathrm{E}_{\mathcal{H},\mathfrak{t}} u =  \mathrm{E}_{\mathcal{H},\mathfrak{t}}[ -\Delta_{\mathcal{H},\mathfrak{t}} u]\text{. }
\end{align*}

Moreover,
\begin{itemize}
    \item the result still holds replacing $(\dot{\mathrm{H}}^{s,p},\dot{\mathrm{D}}^{s}_{p})$ by either $({\mathrm{H}}^{s,p},{\mathrm{D}}^{s}_{p})$, $(\mathrm{B}^{s}_{p,q},\mathrm{D}^{s}_{p,q})$ or $(\dot{\mathrm{B}}^{s}_{p,q},\dot{\mathrm{D}}^{s}_{p,q})$, with $q\in[1,+\infty)$ ;

    \item in case of $(\mathrm{B}^{s}_{p,\infty},\mathrm{D}^{s}_{p,\infty})$ and $(\dot{\mathrm{B}}^{s}_{p,\infty},\dot{\mathrm{D}}^{s}_{p,\infty})$ the Hodge Laplacian is only weak${}^\ast$ densely defined, and strongly closed;
    
    \item all above results remain true replacing $\mathfrak{t}$ by $\mathfrak{n}$.
\end{itemize}
\end{proposition}

From there, the whole context has been established in order to be able to claim the next theorem.

\begin{theorem}\label{thm:MetaThm1HodgeLaplacianRn+}Let $p\in(1,+\infty)$, $q\in[1,+\infty]$, $s\in (-1+1/p,1/p)$, and $k\in\llb 0,n\rrb$.
\begin{enumerate}[label=($\roman*$)]
    \item For $\mu\in[0,\pi)$, $\lambda\in\Sigma_\mu$, if $f\in \dot{\mathrm{H}}^{s,p}(\mathbb{R}^n_+,\Lambda^k)$ then the following resolvent problem
    \begin{align*}
        \lambda u - \Delta_{\mathcal{H}} u = f\quad\text{ in } \mathbb{R}^n_+\text{, }
    \end{align*}
         admits a unique solution $u\in \dot{\mathrm{D}}^{s}_{p}(\Delta_{\mathcal{H}},\mathbb{R}^n_+,\Lambda^k) \subset[\dot{\mathrm{H}}^{s,p}\cap \dot{\mathrm{H}}^{s+2,p}](\mathbb{R}^n_+,\Lambda^k)$ with estimates,
        \begin{align*}
            \lvert\lambda\rvert\lVert  u\rVert_{\dot{\mathrm{H}}^{s,p}(\mathbb{R}^n_+)}+\lvert\lambda\rvert^\frac{1}{2}\lVert \nabla u\rVert_{\dot{\mathrm{H}}^{s,p}(\mathbb{R}^n_+)}+\lVert \nabla^2 u\rVert_{\dot{\mathrm{H}}^{s,p}(\mathbb{R}^n_+)} &\lesssim_{p,n,s,\mu} \lVert f\rVert_{\dot{\mathrm{H}}^{s,p}(\mathbb{R}^n_+)} \text{, }\\
            \lvert\lambda\rvert^\frac{1}{2}\lVert (\mathrm{d}+\delta) u\rVert_{\dot{\mathrm{H}}^{s,p}(\mathbb{R}^n_+)}+\lVert \mathrm{d}\delta u\rVert_{\dot{\mathrm{H}}^{s,p}(\mathbb{R}^n_+)} + \lVert \delta\mathrm{d} u\rVert_{\dot{\mathrm{H}}^{s,p}(\mathbb{R}^n_+)} &\lesssim_{p,n,s,\mu} \lVert f\rVert_{\dot{\mathrm{H}}^{s,p}(\mathbb{R}^n_+)} \text{. }
        \end{align*}
        In particular, $\lambda \mathrm{I}-\Delta_{\mathcal{H}}\,:\,\dot{\mathrm{D}}^{s}_{p}(\Delta_{\mathcal{H}},\mathbb{R}^n_+,\Lambda^k)\longrightarrow \dot{\mathrm{H}}^{s,p}(\mathbb{R}^n_+,\Lambda^k)$ is an isomorphism of Banach spaces.
        
        Furthermore, the result still holds replacing $(\dot{\mathrm{H}}^{s,p},\dot{\mathrm{H}}^{s,p}\cap \dot{\mathrm{H}}^{s+2,{p}},\dot{\mathrm{D}}^{s}_{p})$ by $({\mathrm{H}}^{s,p},{\mathrm{H}}^{s+2,p},{\mathrm{D}}^{s}_{p})$, $(\dot{\mathrm{B}}^{s}_{p,q},\dot{\mathrm{B}}^{s}_{p,q}\cap \dot{\mathrm{B}}^{s+2}_{p,q},\dot{\mathrm{D}}^{s}_{p,q})$, or even by $({\mathrm{B}}^{s}_{p,q},{\mathrm{B}}^{s+2}_{p,q},{\mathrm{D}}^{s}_{p,q})$.
        
        \item For any $\mu\in(0,\pi)$, the operator $-\Delta_{\mathcal{H}}$ admits a bounded (or $\mathrm{\mathbf{H}}^\infty(\Sigma_\mu)$-) holomorphic functional calculus on function spaces: $\dot{\mathrm{H}}^{s,p}(\mathbb{R}^n_+,\Lambda^k)$,  $\dot{\mathrm{B}}^{s}_{p,q}(\mathbb{R}^n_+,\Lambda^k)$, ${\mathrm{H}}^{s,p}(\mathbb{R}^n_+,\Lambda^k)$ and ${\mathrm{B}}^{s}_{p,q}(\mathbb{R}^n_+,\Lambda^k)$.
        
        Moreover, the following resolvent identity holds on any previously mentioned function spaces,
        \begin{align*}
            \mathrm{E}_{\mathcal{H}}(\lambda \mathrm{I} - \Delta_{\mathcal{H}})^{-1} = (\lambda \mathrm{I} - \Delta)^{-1}\mathrm{E}_{\mathcal{H}}\text{. }
        \end{align*}
\end{enumerate}
\end{theorem}

\begin{proof}For $\mu\in[0,\pi)$, $\lambda\in\Sigma_\mu$, if $f\in \dot{\mathrm{H}}^{s,p}(\mathbb{R}^n_+,\Lambda^k)$, we have 
\begin{align*}
    (\lambda \mathrm{I}-\Delta)^{-1}\mathrm{E}_{\mathcal{H}}f \in [\dot{\mathrm{H}}^{s,p}\cap\dot{\mathrm{H}}^{s+2,p}](\mathbb{R}^n,\Lambda^k).
\end{align*}
Thus $u:= [(\lambda \mathrm{I}-\Delta)^{-1}\mathrm{E}_{\mathcal{H}}f]_{|_{\mathbb{R}^n_+}}\in [\dot{\mathrm{H}}^{s,p}\cap\dot{\mathrm{H}}^{s+2,p}](\mathbb{R}^n_+,\Lambda^k)$, satisfies
\begin{align*}
    \lambda u -\Delta u =f \,\text{ in }\, \mathbb{R}^n_+ \text{, }
\end{align*}
with estimates
\begin{align*}
        \lvert\lambda\rvert\lVert  u\rVert_{\dot{\mathrm{H}}^{s,p}(\mathbb{R}^n_+)}+\lvert\lambda\rvert^\frac{1}{2}\lVert \nabla u\rVert_{\dot{\mathrm{H}}^{s,p}(\mathbb{R}^n_+)}+\lVert \nabla^2 u\rVert_{\dot{\mathrm{H}}^{s,p}(\mathbb{R}^n_+)} &\lesssim_{p,n,s,\mu} \lVert f\rVert_{\dot{\mathrm{H}}^{s,p}(\mathbb{R}^n_+)} \text{, }\\
        \lvert\lambda\rvert^\frac{1}{2}\lVert (\mathrm{d}+\delta) u\rVert_{\dot{\mathrm{H}}^{s,p}(\mathbb{R}^n_+)}+\lVert \mathrm{d}\delta u\rVert_{\dot{\mathrm{H}}^{s,p}(\mathbb{R}^n_+)} + \lVert \delta\mathrm{d} u\rVert_{\dot{\mathrm{H}}^{s,p}(\mathbb{R}^n_+)} &\lesssim_{p,n,s,\mu} \lVert f\rVert_{\dot{\mathrm{H}}^{s,p}(\mathbb{R}^n_+)} \text{. }
\end{align*}
By density of $[\mathrm{L}^2\cap\dot{\mathrm{H}}^{s,p}](\mathbb{R}^n_+,\Lambda^k)$ in $\dot{\mathrm{H}}^{s,p}(\mathbb{R}^n_+,\Lambda^k)$ one may use Theorem \ref{thm:HodgeLapL2Rn+} and continuity of traces provided by Theorem \ref{thm:TracesDifferentialformsHomSpaces} to show that necessarily $\nu\iprod u_{|_{\partial\mathbb{R}^n_+}}=0$ and $\nu\iprod \mathrm{d} u_{|_{\partial\mathbb{R}^n_+}}=0$ (or resp. $\nu\wedge u_{|_{\partial\mathbb{R}^n_+}}=0$ and $\nu\wedge \delta u_{|_{\partial\mathbb{R}^n_+}}=0$). Hence $u\in \dot{\mathrm{D}}^{s}_{p}(\Delta_{\mathcal{H}},\mathbb{R}^n_+,\Lambda^k)$.

Now assume $v\in \dot{\mathrm{D}}^{s}_{p}(\Delta_{\mathcal{H}},\mathbb{R}^n_+,\Lambda^k)$ satisfies
\begin{align*}
    \lambda v -\Delta_\mathcal{H} v =f \,\text{ in }\, \mathbb{R}^n_+\text{.}
\end{align*}
We apply Proposition \ref{prop:ExtOpHodgeLapRn+} to claim that $V:=\mathrm{E}_\mathcal{H}v$ must satisfy
\begin{align*}
    \lambda V -\Delta V = \mathrm{E}_{\mathcal{H}}f \,\text{ in }\, \mathbb{R}^n\text{.}
\end{align*}
Thus, necessarily $\mathrm{E}_{\mathcal{H}}(\lambda \mathrm{I} - \Delta_{\mathcal{H}})^{-1}f = (\lambda \mathrm{I} - \Delta)^{-1}\mathrm{E}_{\mathcal{H}}f$.

This resolvent identity leads to the construction of bounded ($\mathrm{\mathbf{H}}^\infty(\Sigma_\mu)$-)holomorphic functional calculus, given by the following identity for all $\Psi\in \mathrm{\mathbf{H}}^{\infty}(\Sigma_\mu)$, $\mu\in(0,\pi)$:
\begin{align*}
    \mathrm{E}_\mathcal{H}\Psi(-\Delta_{\mathcal{H}}) = \Psi(-\Delta)\mathrm{E}_\mathcal{H} \text{.}
\end{align*}
The result for homogeneous Besov spaces $\dot{\mathrm{B}}^{s}_{p,q}$, $q<+\infty$, and other similar inhomogeneous function spaces may be achieved in a similar manner. The case of inhomogeneous and homogeneous Besov spaces with $q=+\infty$ follows from real interpolation.
\end{proof}

The goal for now is to prove the Hodge decomposition. The idea is to prove that the representation formula of $\mathbb{P}$ (resp. $\mathbb{Q}$) proved in Lemma \ref{lem:RiesztransfHodgeL2Rn+} still makes sense on $\dot{\mathrm{H}}^{s,p}(\mathbb{R}^n_+,\Lambda)$, $\dot{\mathrm{B}}^{s}_{p,q}(\mathbb{R}^n_+,\Lambda)$, and their inhomogeneous counterparts. To do so we adapt Lemma \ref{lem:commutrelationL2HodgeRn+} in the present setting.

\begin{lemma}\label{lem:commutrelationHspBspqHodgeRn+}Let $p\in(1,+\infty)$, $s\in(-1+1/p,1/p)$, $\mu\in[0,\pi)$, $\lambda\in\Sigma_\mu$, $t\geqslant0$,  $k\in\llb 0,n\rrb$. The following commutation identities hold,
\begin{enumerate}
    \item $\mathrm{d}(\lambda \mathrm{I} - \Delta_{\mathcal{H},\mathfrak{t}})^{-1}f = (\lambda \mathrm{I} - \Delta_{\mathcal{H},\mathfrak{t}})^{-1}\mathrm{d}f$,\qquad for all $f\in\dot{\mathrm{D}}^s_{p}(\mathrm{d},\mathbb{R}^n_+,\Lambda^k)$ ;
    \item $\mathrm{d}^\ast(\lambda \mathrm{I} - \Delta_{\mathcal{H},\mathfrak{t}})^{-1}f = (\lambda \mathrm{I} - \Delta_{\mathcal{H},\mathfrak{t}})^{-1}\mathrm{d}^\ast f$,\qquad for all $f\in\dot{\mathrm{D}}^s_{p}(\mathrm{d}^\ast,\mathbb{R}^n_+,\Lambda^k)$.
    \item $\mathrm{d}e^{t\Delta_{\mathcal{H},\mathfrak{t}}}f = e^{t\Delta_{\mathcal{H},\mathfrak{t}}}\mathrm{d}f$,\qquad for all $f\in\dot{\mathrm{D}}^s_{p}(\mathrm{d},\mathbb{R}^n_+,\Lambda^k)$ ;
    \item $\mathrm{d}^\ast e^{t\Delta_{\mathcal{H},\mathfrak{t}}}f = e^{t\Delta_{\mathcal{H},\mathfrak{t}}}\mathrm{d}^\ast f$,\qquad for all $f\in\dot{\mathrm{D}}^s_{p}(\mathrm{d}^\ast,\mathbb{R}^n_+,\Lambda^k)$.
\end{enumerate}
Every above identities still hold replacing $(\mathfrak{t},\mathrm{d},\mathrm{d}^\ast)$ by $(\mathfrak{n},\delta^\ast,\delta)$, and $\dot{\mathrm{D}}^s_{p}$ by either ${\mathrm{D}}^s_{p}$, ${\mathrm{D}}^s_{p,q}$ or even by $\dot{\mathrm{D}}^s_{p,q}$, with $q\in[1,+\infty]$.
\end{lemma}

\begin{proof} Let $f\in\dot{\mathrm{D}}^s_{p}(\mathrm{d},\mathbb{R}^n_+,\Lambda^k)\subset \dot{\mathrm{H}}^{s,p}(\mathbb{R}^n_+,\Lambda^k)$, then by Theorem \ref{thm:MetaThm1HodgeLaplacianRn+} there exists a unique $u\in \dot{\mathrm{D}}^s_{p}(\Delta_{\mathcal{H},\mathfrak{t}},\mathbb{R}^n_+,\Lambda^k)\subset \dot{\mathrm{D}}^s_{p}(D_{\mathfrak{t}},\mathbb{R}^n_+,\Lambda^k)\subset \dot{\mathrm{D}}^s_{p}(\mathrm{d},\mathbb{R}^n_+,\Lambda^k)$ such that
\begin{align*}
    \lambda u -\Delta_{\mathcal{H},\mathfrak{t}} u =f \text{. }
\end{align*}
Since $u,f \in \dot{\mathrm{D}}^s_{p}(\mathrm{d},\mathbb{R}^n_+,\Lambda^k)$, we deduce that $\Delta_{\mathcal{H},\mathfrak{t}}u \in \dot{\mathrm{D}}^s_{p}(\mathrm{d},\mathbb{R}^n_+,\Lambda^k)$ and we use $\mathrm{d}^2=0$ to deduce
\begin{align*}
    \lambda \mathrm{d}u -\Delta_{\mathcal{H},\mathfrak{t}} \mathrm{d}u =\mathrm{d}f \text{. }
\end{align*}
But, we have $\mathrm{d}u \in \dot{\mathrm{H}}^{s,p}(\mathbb{R}^n_+,\Lambda^{k+1})$ solution of
\begin{align*}
    \lambda v -\Delta_{\mathcal{H},\mathfrak{t}} v = \mathrm{d}f \text{. }
\end{align*}
Thus, uniqueness of the solution yields $\mathrm{d}u = (\lambda \mathrm{I} - \Delta_{\mathcal{H},\mathfrak{t}})^{-1}\mathrm{d}f$. If it holds for resolvents, then it holds for semigroups.
\end{proof}

\begin{proposition}\label{prop:RieszTransformHodgeHspBspqRn+} Let $p\in(1,+\infty)$, $s\in(-1+1/p,1/p)$,  $k\in\llb 0,n\rrb$. For any $\lambda\geqslant 0$, following operators are well defined and uniformly bounded with respect to $\lambda$
\begin{align*}
    \mathrm{d}(\lambda \mathrm{I}-\Delta_{\mathcal{H},\mathfrak{t}})^{-\frac{1}{2}}\,&:\,\dot{\mathrm{H}}^{s,p}(\mathbb{R}^{n}_+,\Lambda^k)\longrightarrow \dot{\mathrm{N}}^{s}_{p}(\mathrm{d},\mathbb{R}^{n}_+,\Lambda^{k+1})\, ;\\
    \mathrm{d}^\ast(\lambda \mathrm{I}-\Delta_{\mathcal{H},\mathfrak{t}})^{-\frac{1}{2}}\,&:\,\dot{\mathrm{H}}^{s,p}(\mathbb{R}^{n}_+,\Lambda^k)\longrightarrow \dot{\mathrm{N}}^{s}_{p}(\mathrm{d}^\ast,\mathbb{R}^{n}_+,\Lambda^{k-1})\,.
\end{align*}
Moreover, following identities also hold for all $\lambda>0$
\begin{itemize}
    \item $\mathrm{d}(\lambda \mathrm{I}-\Delta_{\mathcal{H},\mathfrak{t}})^{-\frac{1}{2}}f =  (\lambda \mathrm{I}-\Delta_{\mathcal{H},\mathfrak{t}})^{-\frac{1}{2}} \mathrm{d} f$ \qquad for all $f\in\dot{\mathrm{D}}^s_{p}(\mathrm{d},\mathbb{R}^n_+,\Lambda^k)$ ;
    \item $\mathrm{d}^\ast(\lambda \mathrm{I}-\Delta_{\mathcal{H},\mathfrak{t}})^{-\frac{1}{2}}f =  (\lambda \mathrm{I}-\Delta_{\mathcal{H},\mathfrak{t}})^{-\frac{1}{2}} \mathrm{d}^\ast f$ \qquad for all $f\in\dot{\mathrm{D}}^s_{p}(\mathrm{d}^\ast,\mathbb{R}^n_+,\Lambda^k)$.
\end{itemize}
Everything still holds replacing $(\mathfrak{t},\mathrm{d},\mathrm{d}^\ast)$ by $(\mathfrak{n},\delta^\ast,\delta)$, and replacing $(\dot{\mathrm{H}}^{s,p},\dot{\mathrm{N}}^{s}_{p})$ by either $(\dot{\mathrm{B}}^{s}_{p,q},\dot{\mathrm{N}}^{s}_{p,q})$, $({\mathrm{H}}^{s,p},{\mathrm{N}}^{s}_{p})$ or even by $({\mathrm{B}}^{s}_{p,q},{\mathrm{N}}^{s}_{p,q})$ with $q\in[1,+\infty]$.
\end{proposition}

\begin{proof} For $\lambda\geqslant 0$, we introduce the representation formula,
\begin{align}\label{eq:repformulanegSqrt}
    (\lambda \mathrm{I}-\Delta_{\mathcal{H},\mathfrak{t}})^{-\frac{1}{2}} f = \frac{1}{\sqrt{\pi}}\int_{0}^{+\infty} e^{-\tau\lambda}e^{\tau\Delta_{\mathcal{H},\mathfrak{t}}}f \frac{\mathrm{d}\tau}{\sqrt{\tau}} \text{. }
\end{align}
This representation formula makes sense thanks to holomorphic functional calculus, and the integral is absolutely convergent for every for $f\in \dot{\mathrm{R}}^s_{p}(\Delta_{\mathcal{H},\mathfrak{t}},\mathbb{R}^n_+,\Lambda^k)$.

We use the definition of function spaces by restriction and the bounded holomorphic functional calculus, with the identity provided by point \textit{(ii)} of Theorem \ref{thm:MetaThm1HodgeLaplacianRn+}, \textit{i.e.} $\mathrm{E}_{\mathcal{H},\mathfrak{t}}e^{\tau\Delta_{\mathcal{H},\mathfrak{t}}} = e^{\tau\Delta}\mathrm{E}_{\mathcal{H},\mathfrak{t}}$, to obtain
\begin{align*}
    \lVert \mathrm{d}(\lambda \mathrm{I}-\Delta_{\mathcal{H},\mathfrak{t}})^{-\frac{1}{2}} f\rVert_{\dot{\mathrm{H}}^{s,p}(\mathbb{R}^{n}_+)} + \lVert \mathrm{d}^\ast(\lambda \mathrm{I}-\Delta_{\mathcal{H},\mathfrak{t}})^{-\frac{1}{2}} f\rVert_{\dot{\mathrm{H}}^{s,p}(\mathbb{R}^{n}_+)} &\lesssim_{k,n} \lVert \nabla (\lambda \mathrm{I}-\Delta)^{-\frac{1}{2}} \mathrm{E}_{\mathcal{H},\mathfrak{t}} f\rVert_{\dot{\mathrm{H}}^{s,p}(\mathbb{R}^{n})}\\
    &\lesssim_{n,k,s,p} \lVert  f\rVert_{\dot{\mathrm{H}}^{s,p}(\mathbb{R}^{n}_+)}\text{. }
\end{align*}
Therefore, the boundedness follows by density of $\dot{\mathrm{R}}^s_{p}(\Delta_{\mathcal{H},\mathfrak{t}},\mathbb{R}^n_+,\Lambda^k)$ in $\dot{\mathrm{H}}^{s,p}(\mathbb{R}^n_+,\Lambda^k)$.
Commutations relations when $\lambda>0$ follow from Lemma \ref{lem:commutrelationHspBspqHodgeRn+} and the representation formula \eqref{eq:repformulanegSqrt}.
The boundedness on the Besov scale follows from real interpolation.
\end{proof}

According to more convenient and usual notations with respect to the field of partial differential equations we set new symbols.
\begin{notations}\label{def:newNotationsSoleinodalSpaces} We introduce the following notations
\begin{align*}
    {\mathrm{H}}^{s,p}_{{\mathfrak{t}},\sigma} :={\mathrm{N}}^{s}_{p}(\mathrm{d}^\ast) \text{, } {\mathrm{H}}^{s,p}_{\gamma}:={\mathrm{N}}^{s}_{p}(\mathrm{d}) &\text{ and } {\mathrm{H}}^{s,p}_{\sigma} :={\mathrm{N}}^{s}_{p}(\delta) \text{, } {\mathrm{H}}^{s,p}_{{\mathfrak{n}},\gamma}:={\mathrm{N}}^{s}_{p}(\delta^\ast)\text{; }\\
    {\mathrm{B}}^{s,\sigma}_{p,q,{\mathfrak{t}}} :={\mathrm{N}}^{s}_{p,q}(\mathrm{d}^\ast) \text{, } {\mathrm{B}}^{s,\gamma}_{p,q} :={\mathrm{N}}^{s}_{p,q}(\mathrm{d}) &\text{ and }
    {\mathrm{B}}^{s,\sigma}_{p,q} :={\mathrm{N}}^{s}_{p}(\delta) \text{, } {\mathrm{B}}^{s,\gamma}_{p,q,{\mathfrak{n}}} :={\mathrm{N}}^{s}_{p,q}(\delta^\ast)\text{. }
\end{align*}
\end{notations}

Then we are able to obtain the following result.
\begin{theorem}\label{thm:HodgeDecompRn+}Let $p\in(1,+\infty)$, $q\in[1,+\infty]$, $s\in(-1+1/p,1/p)$, and let $k\in\llb 0,n\rrb$. It holds that
\begin{enumerate}
    \item  The following equality holds with equivalence of norms,
    \begin{align*}
        \dot{\mathrm{N}}^{s}_{p}(\mathrm{d},\mathbb{R}^n_+,\Lambda^k) = \overline{\dot{\mathrm{R}}^{s}_{p}(\mathrm{d},\mathbb{R}^n_+,\Lambda^k)}^{\lVert\cdot\rVert_{\dot{\mathrm{H}}^{s,p}(\mathbb{R}^n_+)}}\text{, }
    \end{align*}
    and still holds replacing $\mathrm{d}$ by $\mathrm{d}^\ast$.
    
    \item The (generalized) Helmholtz-Leray projector $\mathbb{P}\,:\, \dot{\mathrm{H}}^{s,p}(\mathbb{R}^n_+,\Lambda^k)\longrightarrow \dot{\mathrm{H}}^{s,p}_{{\mathfrak{t}},\sigma}(\mathbb{R}^n_+,\Lambda^k)$ is well defined and bounded. Moreover the following identity is true
    \begin{align*}
        \mathbb{P} = \mathrm{I} - \mathrm{d}(-\Delta_{\mathcal{H},\mathfrak{t}})^{-\frac{1}{2}}\mathrm{d}^\ast(-\Delta_{\mathcal{H},\mathfrak{t}})^{-\frac{1}{2}} \text{.}
    \end{align*}

    \item The following Hodge decomposition holds
    \begin{align*}
        \dot{\mathrm{H}}^{s,p}(\mathbb{R}^n_+,\Lambda^k) = \dot{\mathrm{H}}^{s,p}_{{\mathfrak{t}},\sigma}(\mathbb{R}^n_+,\Lambda^k) \oplus \dot{\mathrm{H}}^{s,p}_{\gamma}(\mathbb{R}^n_+,\Lambda^k)\text{. }
    \end{align*}
\end{enumerate}
Moreover, the result remains true if we replace
\begin{itemize}
    \item $\dot{\mathrm{H}}^{s,p}$ by $\dot{\mathrm{B}}^{s}_{p,q}$ ;
    \item  $(\dot{\mathrm{H}},\dot{\mathrm{B}})$ by $({\mathrm{H}},{\mathrm{B}})$ ;
    \item $(\mathfrak{t},\mathrm{d},\mathrm{d}^\ast,\mathbb{P},\sigma,\gamma)$ by $(\mathfrak{n},\delta^\ast,\delta,\mathbb{Q},\gamma,\sigma)$.
\end{itemize}
In case of Besov spaces with $q=+\infty$, the density result of point \textit{(i)} only holds in the weak${}^\ast$ sense.
\end{theorem}

\begin{proof}One may reproduce the proofs of Theorem \ref{thm:HodgeDecompL2Rn+Full}, thanks to above Proposition \ref{prop:RieszTransformHodgeHspBspqRn+}.
\end{proof}

The following corollary is a direct consequence of the given expression for the Helmholtz-Leray projection in Theorem \ref{thm:HodgeDecompRn+}.

\begin{corollary}\label{cor:commutrelationHodgeProjectorHspBspqRn+}Let $p\in(1,+\infty)$, $s\in(-1+1/p,1/p)$, $\mu\in[0,\pi)$, $\lambda\in\Sigma_\mu$,  $k\in\llb 0,n\rrb$. The following commutation identities hold for all $f\in\dot{\mathrm{H}}^{s,p}(\mathbb{R}^n_+,\Lambda^k)$, for all $t\geqslant 0$,
\begin{align*}
    (\lambda \mathrm{I}-\Delta_{\mathcal{H},\mathfrak{t}})^{-1} \mathbb{P} f &= \mathbb{P}(\lambda \mathrm{I}-\Delta_{\mathcal{H},\mathfrak{t}})^{-1}  f \text{, }\\
    e^{t\Delta_{\mathcal{H},\mathfrak{t}}} \mathbb{P} f &= \mathbb{P}e^{t\Delta_{\mathcal{H},\mathfrak{t}}}  f \text{. }
\end{align*}
Above identity still holds replacing $(\mathfrak{t},\mathbb{P})$ by $(\mathfrak{n},\mathbb{Q})$, and $\dot{\mathrm{H}}^{s,p}$ by either ${\mathrm{H}}^{s,p}$, ${\mathrm{B}}^s_{p,q}$ or even by $\dot{\mathrm{B}}^s_{p,q}$, with $q\in[1,+\infty]$.
\end{corollary}

\subsubsection{Hodge-Stokes and Hodge-Maxwell operators}

The present subsection is about discussing properties of Hodge-Stokes and Hodge-Maxwell operators. First, one can define \textbf{\textit{Hodge-Stokes operator}}'s domain, for all $p\in(1,+\infty)$, $s\in(-1+1/p,1/p)$, $k\in\llb 0,n\rrb$, by
\begin{align}\label{eq:domainHodgeStokesAbsBC}
    \dot{\mathrm{D}}^{s}_{p}(\mathbb{A}_{\mathcal{H},\mathfrak{t}},\mathbb{R}^n_+,\Lambda^k):= \dot{\mathrm{H}}^{s,p}_{\mathfrak{t},\sigma}(\mathbb{R}^n_+,\Lambda^k)\cap\dot{\mathrm{D}}^{s}_{p}(\Delta_{\mathcal{H},\mathfrak{t}},\mathbb{R}^n_+,\Lambda^k)\text{, } 
\end{align}
and for all $u\in \dot{\mathrm{D}}^{s}_{p}(\mathbb{A}_{\mathcal{H},\mathfrak{t}},\mathbb{R}^n_+,\Lambda^k)$
\begin{align}\label{eq:defHodgeStokesAbsBC}
    \mathbb{A}_{\mathcal{H},\mathfrak{t}} u := \mathrm{d}^\ast\mathrm{d} u =  -\mathbb{P}\Delta u = -\Delta \mathbb{P} u =-\Delta u \text{. }
\end{align}
Above operator $\mathbb{A}_{\mathcal{H},\mathfrak{t}}$ is called the \textbf{\textit{Hodge-Stokes operator}} with \textbf{\textit{absolute boundary conditions}} which is a closed densely defined operator on $\dot{\mathrm{H}}^{s,p}_{{\mathfrak{t}},\sigma}(\mathbb{R}^n_+,\Lambda^k)$.

Similarly one can treat the case of \textbf{\textit{Hodge-Maxwell operators}},
\begin{align}\label{eq:domainHodgeMaxwellPerfectConducBC}
    \dot{\mathrm{D}}^{s}_{p}(\mathbb{M}_{\mathcal{H},\mathfrak{t}},\mathbb{R}^n_+,\Lambda^k):= \dot{\mathrm{H}}^{s,p}_{\gamma}(\mathbb{R}^n_+,\Lambda^k)\cap\dot{\mathrm{D}}^{s}_{p}(\Delta_{\mathcal{H},\mathfrak{t}},\mathbb{R}^n_+,\Lambda^k)\text{, } 
\end{align}
and for all $u\in \dot{\mathrm{D}}^{s}_{p}(\mathbb{M}_{\mathcal{H},\mathfrak{t}},\mathbb{R}^n_+,\Lambda^k)$
\begin{align}\label{eq:defHodgeMaxwellPerfectConducBC}
    \mathbb{M}_{\mathcal{H},\mathfrak{t}} u := \mathrm{d}\mathrm{d}^\ast u =  -[\mathrm{I}-\mathbb{P}]\Delta u = -\Delta [\mathrm{I}-\mathbb{P}] u =-\Delta u \text{. }
\end{align}
The operator $\mathbb{M}_{\mathcal{H},\mathfrak{t}}$ defined as above is called the \textbf{\textit{Hodge-Maxwell operator}} with \textbf{\textit{perfectly conductive wall boundary conditions}} which is a closed densely defined operator on $\dot{\mathrm{H}}^{s,p}_{\gamma}(\mathbb{R}^n_+,\Lambda^k)$.

Similarly, one may replace $(\mathfrak{t},\mathrm{d},\mathbb{P})$ by $(\mathfrak{n},\delta,\mathbb{Q})$, respectively in, \eqref{eq:domainHodgeStokesAbsBC} and \eqref{eq:defHodgeStokesAbsBC}, and in \eqref{eq:domainHodgeMaxwellPerfectConducBC} and \eqref{eq:defHodgeMaxwellPerfectConducBC}. This leads to the construction of
\begin{align}
    (\dot{\mathrm{D}}^{s}_{p}(\mathbb{A}_{\mathcal{H},\mathfrak{n}},\mathbb{R}^n_+,\Lambda^k), \mathbb{A}_{\mathcal{H},\mathfrak{n}})\quad \text{and}\quad\text(\dot{\mathrm{D}}^{s}_{p}(\mathbb{M}_{\mathcal{H},\mathfrak{n}},\mathbb{R}^n_+,\Lambda^k), \mathbb{M}_{\mathcal{H},\mathfrak{n}})
\end{align}
called respectively the \textbf{\textit{Hodge-Stokes operator}} with \textbf{\textit{relative boundary conditions}} and  the \textbf{\textit{Hodge-Maxwell operator}} with \textbf{\textit{relative boundary conditions}} which are both closed densely defined operator on $\dot{\mathrm{H}}^{s,p}_{\sigma}(\mathbb{R}^n_+,\Lambda^k)$ and $\dot{\mathrm{H}}^{s,p}_{\mathfrak{n},\gamma}(\mathbb{R}^n_+,\Lambda^k)$, respectively.

Those Hodge-Stokes and Hodge-Maxwell operators still have sense on other scale of function spaces replacing $(\dot{\mathrm{H}}^{s,p},\dot{\mathrm{D}}^{s}_{p})$ by $(\dot{\mathrm{B}}^{s}_{p,q},\dot{\mathrm{D}}^{s}_{p,q})$, then $(\dot{\mathrm{H}},\dot{\mathrm{B}},\dot{\mathrm{D}})$ by $({\mathrm{H}},{\mathrm{B}},{\mathrm{D}})$.

Notice again the exception of Besov spaces, homogeneous and inhomogeneous, where the domains of any Hodge-Stokes and Hodge-Maxwell operators are only weak${}^\ast$ dense in the case $q=+\infty$.

With the above definitions, Theorem \ref{thm:MetaThm1HodgeLaplacianRn+}, Corollary \ref{cor:commutrelationHodgeProjectorHspBspqRn+} and Theorem \ref{thm:HodgeDecompRn+}, we obtain for free the next theorem.
\begin{theorem}\label{thm:HinftyFuncCalcHodgeStokesMaxwell} Let $p\in(1,+\infty)$, $s\in(-1+1/p,1/p)$. For all $\mu\in (0,\pi)$, the operator $\mathbb{A}_{\mathcal{H},\mathfrak{t}}$ (resp. $\mathbb{M}_{\mathcal{H},\mathfrak{t}}$) admits a bounded (or $\mathrm{\mathbf{H}}^{\infty}(\Sigma_\mu)$-)holomorphic functional calculus on $\dot{\mathrm{H}}^{s,p}_{{\mathfrak{t}},\sigma}(\mathbb{R}^n_+,\Lambda)$ (resp. $\dot{\mathrm{H}}^{s,p}_{\gamma}(\mathbb{R}^n_+,\Lambda)$ ).\\
Moreover, the result remains true if we replace
\begin{itemize}
    \item $\dot{\mathrm{H}}^{s,p}$ by $\dot{\mathrm{B}}^{s}_{p,q}$, $q\in[1,+\infty]$ ;
    \item  $(\dot{\mathrm{H}},\dot{\mathrm{B}})$ by $({\mathrm{H}},{\mathrm{B}})$ ;
    \item $(\mathfrak{t},\sigma,\gamma,\mathbb{A},\mathbb{M})$ by $(\mathfrak{n},\gamma,\sigma,\mathbb{M},\mathbb{A})$.
\end{itemize}
\end{theorem}


\section{\texorpdfstring{$\mathrm{L}^q$}{Lq}-maximal regularities with global-in-time estimates} \label{Sect:HomInterpDomainsandHomMaxReg}

In order to motivate the results in next Sections \ref{sec:homogeneousinterp} and \ref{sec:maxRegMaxHodgeStokes}, we provide short reminders about recent advances for maximal regularity in the Sobolev framework provided by \cite[Chapter~2]{DanchinHieberMuchaTolk2020} and \cite{Gaudin2023}. We are going to follow the presentation from \cite[Section~2]{Gaudin2023}.

First, let us consider $(\mathrm{D}(A),A)$ a densely defined closed operator on a Banach space $X$. It is known, see \cite[Theorem~3.7.11]{ArendtBattyHieberNeubranker2011}, that the two following assertions are equivalent:\textit{
\begin{enumerate}
    \item $A$ is $\omega$-sectorial on $X$, with $\omega\in [0,\tfrac{\pi}{2})$;
    \item $-A$ generates a bounded holomorphic $\mathrm{C}_0$-semigroup on $X$, denoted by $(e^{-tA})_{t\geqslant0}$.
\end{enumerate}}

Thus, provided that $A$ is $\omega$-sectorial on $X$ for some $\omega\in [0,\tfrac{\pi}{2})$, for $T\in(0,+\infty]$, we look at the following abstract Cauchy problem, 
\begin{align}\tag{ACP}\label{ACP}
    \left\{\begin{array}{rl}
            \partial_t u(t) +Au(t)  =& f(t) \,\text{, } 0<t<T \text{, }\\
            u(0) =& u_0\text{. }
    \end{array}
    \right.\text{, }
\end{align}
where $f\in \mathrm{L}^1_{\mathrm{loc}}((0,T),X)$, $u_0 \in Y$, $Y$ being some normed vector space depending on $X$ and $\mathrm{D}(A)$.

We want to look at global-in-time maximal regularity results for \eqref{ACP}. To obtain estimates that are uniform in time, we require the involved function spaces to be \textbf{homogeneous}. This keypoint was captured in the work of Danchin, Hieber, Mucha and Tolksdorf \cite[Chapter~2]{DanchinHieberMuchaTolk2020} to build an \textbf{homogeneous} version of the Da~Prato-Grivard theorem for injective sectorial operators under some additional assumptions on $A$. We are going to present briefly their construction to motivate the next section.

Before that,  we introduce two quantities for $v\in X + \mathrm{D}(A)$,

\begin{align*}
    \lVert v\rVert_{\mathring{\eus{D}}_{A}(\theta,q)} := \left( \int_{0}^{+\infty} (t^{1-\theta}\lVert Ae^{-tA} v \rVert_{X})^q \frac{\mathrm{d}t}{t}\right)^\frac{1}{q}\text{, and }  \lVert v\rVert_{{\eus{D}}_{A}(\theta,q)} := \lVert v \rVert_X + \lVert v\rVert_{\mathring{\eus{D}}_{A}(\theta,q)}\text{, }
\end{align*}
where $\theta \in(0,1)$, $q\in[1,+\infty]$. This leads to the construction of the vector space
\begin{align*}
    {\eus{D}}_{A}(\theta,q) := \{ v\in X\,|\, \lVert v\rVert_{\mathring{\eus{D}}_{A}(\theta,q)}<+\infty \} \text{.}
\end{align*}
The vector space ${\eus{D}}_{A}(\theta,q)$ is known to be a Banach space under the norm $\lVert \cdot\rVert_{{\eus{D}}_{A}(\theta,q)}$ and moreover it satisfies the following equality with equivalence of norms
\begin{align}\label{eq:InterpolationXDomainAIntro}
    {\eus{D}}_{A}(\theta,q) = (X,\mathrm{D}(A))_{\theta,q}\text{, }
\end{align}
see \cite[Theorem~6.2.9]{bookHaase2006}. If moreover, $0\in \rho(A)$ it has been proved, \cite[Corollary~6.5.5]{bookHaase2006}, that $\lVert \cdot\rVert_{\mathring{\eus{D}}_{A}(\theta,q)}$ and $\lVert \cdot\rVert_{{\eus{D}}_{A}(\theta,q)}$ are two equivalent norms on ${\eus{D}}_{A}(\theta,q)$. So we restrict ourselves to the case of injective operators.

\begin{assumption}\label{asmpt:homogeneousdomaindef} The operator $(\mathrm{D}(A),A)$ is injective on $X$, and there exists a normed vector space $(Y, \left\lVert \cdot \right\rVert_{Y})$, such that for all $x\in \mathrm{D}(A)$,
\begin{align}
        \left\lVert Ax\right\rVert_{X} \sim \left\lVert x \right\rVert_{Y} \text{. }
\end{align}
\end{assumption}

The idea is to construct an homogeneous version of $A$ denoted $\mathring{A}$, defining first its domain
\begin{align*}
    \mathrm{D}(\mathring{A}) := \{\, y\in Y\, |\, \exists (x_n)_{n\in\mathbb{N}}\subset \mathrm{D}(A),\, \left\lVert y-x_n\right\rVert_Y\underset{n\rightarrow +\infty}{\longrightarrow} 0   \,\}\text{. }
\end{align*}
Then, for all $y\in \mathrm{D}(\mathring{A})$,
\begin{align*}
    \mathring{A}y := \lim_{n\rightarrow +\infty} Ax_n \text{. }
\end{align*}
Constructed this way, the operator $\mathring{A}$ is then injective on $\mathrm{D}(\mathring{A})$. We notice that $\mathrm{D}(\mathring{A})$ is a normed vector space, but not necessarily complete. We also need the existence of a Hausdorff topological vector space $Z$, such that $X,Y\subset Z$, and to consider the following assumption 
\begin{assumption}\label{asmpt:homogeneousdomainintersect} The operator $(\mathrm{D}(A),A)$ and the normed vector space $Y$ are such that
\begin{align}
        X\cap\mathrm{D}(\mathring{A}) = \mathrm{D}({A}) \text{. }
\end{align}
\end{assumption}
As a consequence of all above assumptions, we can extend naturally, see \cite[Remark~2.7]{DanchinHieberMuchaTolk2020}, $(e^{-tA})_{t\geqslant0}$ to a $\mathrm{C}_0$-semigroup,
\begin{align*}
        e^{-tA}\,:\, X+\mathrm{D}(\mathring{A}) \longrightarrow X+\mathrm{D}(\mathring{A})\text{ , } t\geqslant 0 \text{, }
\end{align*}
so that, one can fully make sense of the following vector space,
\begin{align*}
    \mathring{\eus{D}}_{A}(\theta,q) := \{ v\in X+\mathrm{D}(\mathring{A})\,|\, \lVert v\rVert_{\mathring{\eus{D}}_{A}(\theta,q)}<+\infty \} \text{.}
\end{align*}
Similarly to what happens for ${\eus{D}}_{A}(\theta,q)$ in \eqref{eq:InterpolationXDomainAIntro}, it has been proved in \cite[Proposition~2.12]{DanchinHieberMuchaTolk2020}, that the following equality holds with equivalence of norms,
\begin{align}
    \mathring{\eus{D}}_{A}(\theta,q) = (X,\mathrm{D}(\mathring{A}))_{\theta,q}\text{, }
\end{align}
but the lack of completeness of $\mathrm{D}(\mathring{A})$ implies that $\mathring{\eus{D}}_{A}(\theta,q)$ is not necessarily complete. This has consequences on how to consider the forcing term $f$ in \eqref{ACP}, choosing $f\in \mathrm{L}^q((0,T),{\eus{D}}_{A}(\theta,q))$ instead of $f\in \mathrm{L}^q((0,T),\mathring{\eus{D}}_{A}(\theta,q))$ to avoid definition issues, the later choice being possible when $\mathring{\eus{D}}_{A}(\theta,q)$ is a Banach space.

\begin{theorem}[{\cite[Theorem~2.20]{DanchinHieberMuchaTolk2020}} ] \label{thm:DaPratoGrisvardHom2020} Let $\omega\in [0,\frac{\pi}{2})$, $(\mathrm{D}(A),A)$ an $\omega$-sectorial operator on a Banach space $X$ such that Assumptions \ref{asmpt:homogeneousdomaindef} and \ref{asmpt:homogeneousdomainintersect} are satisfied. Let $q\in [1,+\infty)$, $\theta\in(0,\tfrac{1}{q})$, $\theta_q := \theta +1-1/q$, and let $T\in(0,+\infty]$.

For $f\in \mathrm{L}^q((0,T),{\eus{D}}_{A}(\theta,q))$ and $u_0\in \mathring{\eus{D}}_{A}(\theta_q,q)$, the problem \eqref{ACP} admits a unique mild solution
\begin{align*}
    u\in \mathrm{C}^0_{b}([0,T],\mathring{\eus{D}}_{A}(\theta_q,q))\text{,}
\end{align*}
such that $\partial_t u$, $Au\in \mathrm{L}^q((0,T),\mathring{\eus{D}}_{A}(\theta,q))$ with estimates,
\begin{align}
    \lVert u\rVert_{\mathrm{L}^\infty((0,T),\mathring{\eus{D}}_{A}(\theta_q,q))} + \lVert (\partial_t u, Au)\rVert_{\mathrm{L}^q((0,T),\mathring{\eus{D}}_{A}(\theta,q))} \lesssim_{A} \lVert f\rVert_{\mathrm{L}^q((0,T),\mathring{\eus{D}}_{A}(\theta,q))} + \lVert u_0\rVert_{\mathring{\eus{D}}_{A}(\theta_q,q)}\text{. } \label{estimateLqMaxRegDaPratoGrisvardHom}
\end{align}
In case $q=+\infty$, we assume in addition that $u_0\in \mathrm{D}(A^2)$ and then for each $\theta\in (0,1)$,
\begin{align}
    \lVert (\partial_t u, Au)\rVert_{\mathrm{L}^\infty((0,T),\mathring{\eus{D}}_{A}(\theta,\infty))} \lesssim_{A} \lVert f\rVert_{\mathrm{L}^\infty((0,T),{\eus{D}}_{A}(\theta,\infty))} + \lVert A u_0\rVert_{\mathring{\eus{D}}_{A}(\theta,\infty)}\text{. }
\end{align}
\end{theorem}

One also have the following result.
\begin{theorem}[{\cite[Theorem~4.7]{Gaudin2023}} ]\label{thm:LqMaxRegUMDHomogeneous}Let $\omega\in [0,\frac{\pi}{2})$, $(\mathrm{D}(A),A)$ an $\omega$-sectorial operator on a UMD Banach space $X$, such that it has BIP on $X$ of type $\theta_A<\frac{\pi}{2}$, and satisfies assumptions \eqref{asmpt:homogeneousdomaindef} and \eqref{asmpt:homogeneousdomainintersect}. Let $q\in(1,+\infty)$, $\alpha\in(-1+1/q,1/q)$ and $T\in(0,+\infty]$. We set $\alpha_q:= 1+\alpha-1/q$.

For $f\in\dot{\mathrm{H}}^{\alpha,q}((0,T),X)$, $u_0\in \mathring{\eus{D}}_{A}(\alpha_q,q)$, the problem \eqref{ACP} admits a unique mild solution $u\in \mathrm{C}^0_b([0,T],\mathring{\eus{D}}_{A}(\alpha_q,q))$ such that $\partial_t u$, $Au \in \dot{\mathrm{H}}^{\alpha,q}((0,T),X)$ with estimate
\begin{align}\label{eq:BoundLqMaxReghomogeneous}
    \lVert u \rVert_{\mathrm{L}^\infty([0,T],\mathring{\eus{D}}_{A}(\alpha_q,q))} \lesssim_{A,q,\alpha} \lVert (\partial_t u, Au)\rVert_{\dot{\mathrm{H}}^{\alpha,q}((0,T),X)} \lesssim_{A,q,\alpha} \lVert f\rVert_{\dot{\mathrm{H}}^{\alpha,q}((0,T),X)} + \lVert u_0\rVert_{\mathring{\eus{D}}_{A}(\alpha_q,q)}\text{. }
\end{align}
Moreover, if $u_0\in {\eus{D}}_{A}(\alpha_q,q)$ for all $\beta\in[0,1]$,
\begin{align}\label{eq:mixedDerivEstimate}
    \lVert (-\partial_t)^{\beta} A^{1-\beta} u\rVert_{\dot{\mathrm{H}}^{\alpha,q}((0,T),X)} \lesssim_{A,q,\alpha} \lVert f\rVert_{\dot{\mathrm{H}}^{\alpha,q}((0,T),X)} + \lVert u_0\rVert_{\mathring{\eus{D}}_{A}(\alpha_q,q)} \text{. }
\end{align}
\end{theorem}

\begin{remark} $\bullet$ In Theorem \ref{thm:LqMaxRegUMDHomogeneous}, assumptions \eqref{asmpt:homogeneousdomaindef} and \eqref{asmpt:homogeneousdomainintersect} are assumed here in order to ensure that $\mathring{\eus{D}}_{A}(\theta,q)$ is a well defined, even if not complete, normed vector space. The estimate \eqref{eq:mixedDerivEstimate} still holds for $u_0\in \mathring{\eus{D}}_{A}(1+\alpha-1/q,q)$ whenever this space is complete.

$\bullet$ If $u_0=0$, the estimate \eqref{eq:mixedDerivEstimate} remains valid if we replace the operator $(-\partial_t)^{1-\beta}$ by $(\partial_t)^{1-\beta}$.

$\bullet$ If one asks instead the initial data $u_0$ to be in the smaller, but complete, space ${\eus{D}}_{A}(\theta,q)$ then one can drop assumptions \eqref{asmpt:homogeneousdomaindef} and \eqref{asmpt:homogeneousdomainintersect}, and the estimate \eqref{eq:BoundLqMaxReghomogeneous} still holds.
\end{remark}

Before going further, we want to simplify notations. From now on, we will only consider function spaces on $\mathbb{R}^n_+$ and no longer on $\mathbb{R}^n$, so that we drop the mention of the open set in domains of operators, we all also drop the mention of degree of differential forms, except when it is necessary. Any discussion involving Dirichlet and Neumann Laplacians will always contain the implicit information that their domains are made of scalar valued functions, whereas talking about Hodge Laplacians and their derived operators will always contain the implicit information we are talking about general differential forms valued functions of any degree, unless it is explicitly stated.

A first aim of this section is about to give an explicit description of homogeneous interpolation spaces, provided $\theta\in (0,1)$, $q\in [1,+\infty]$,
\begin{align}\label{eq:hominterpspaceXDA}
    (\mathrm{X},\mathrm{D}(\mathring{A}))_{\theta,q} = \mathring{\eus{D}}_{A}(\theta,q)\text{, }
\end{align}
where $X=\dot{\mathrm{H}}^{s,p},\dot{\mathrm{B}}^{s}_{p,r}$ and $A\in\{ -\Delta_{\mathcal{H}},\mathbb{A}_{\mathcal{H}},\mathbb{M}_\mathcal{H}\}$, with $p\in(1,+\infty)$, $-1+1/p<s<1/p$, $r\in[1,+\infty)$. The main task here will be to compute \eqref{eq:hominterpspaceXDA} to compute above space for $A=-\Delta_{\mathcal{H}}$. Indeed, for the related Hodge-Stokes operator, due to the commutations relations between the Hodge Laplacian and its Helmholtz-Leray projection, see \eqref{eq:defHodgeStokesAbsBC}, \eqref{eq:defHodgeMaxwellPerfectConducBC} and Corollary~\ref{cor:commutrelationHodgeProjectorHspBspqRn+}, we should have (at least formally or up to a dense subset)
\begin{align*}
    \mathring{\eus{D}}^{s,p}_{\mathbb{A}_{\mathcal{H},\mathfrak{t}}}(\theta,q) = (\dot{\mathrm{H}}^{s,p}_{\mathfrak{t},\sigma}(\mathbb{R}^n_+),\dot{\mathrm{D}}^{s}_{p}(\mathring{\mathbb{A}}_{\mathcal{H},\mathfrak{t}}))_{\theta,q} = \mathbb{P}(\dot{\mathrm{H}}^{s,p}(\mathbb{R}^n_+),\dot{\mathrm{D}}^{s}_{p}(\mathring{\Delta}_{\mathcal{H},\mathfrak{t}}))_{\theta,q} = \mathbb{P}\mathring{\eus{D}}^{s,p}_{-{\Delta}_{\mathcal{H},\mathfrak{t}}}(\theta,q)\text{. }
\end{align*}
Obviously similar identities can be obtained with $(\mathfrak{n},\mathbb{Q})$ instead of $(\mathfrak{t},\mathbb{P})$, but also for the Hodge-Maxwell operators up to appropriate changes.

Secondly, we will aim to recover global-in-time $\mathrm{L}^q$-maximal regularity estimates for the abstract Cauchy problem \eqref{ACP}, provided $T\in(0,+\infty]$, for $A\in\{ -\Delta_{\mathcal{H}},\mathbb{A}_{\mathcal{H}},\mathbb{M}_\mathcal{H}\}$, so that we will apply Theorems~\ref{thm:DaPratoGrisvardHom2020}~and~\ref{thm:LqMaxRegUMDHomogeneous}.

\subsection{Interpolation of homogeneous \texorpdfstring{$\dot{\mathrm{H}}^{s,p}$}{Hsp}-domains of operators}\label{sec:homogeneousinterp}

We start this section claiming that one can reduce the problem to the computation of interpolation spaces
\begin{align*}
    \mathring{\eus{D}}^{s,p}_{-{\Delta}_{\mathcal{D}}}(\theta,q) \,\text{ and }\, \mathring{\eus{D}}^{s,p}_{-{\Delta}_{\mathcal{N}}}(\theta,q)\text{. }
\end{align*}
We recall here for convenience that $-{\Delta}_{\mathcal{D}}$ and $-{\Delta}_{\mathcal{N}}$ stands respectively for the (negative) Dirichlet and the Neumann Laplacian on the half-space for which a wide review of their properties in homogeneous function spaces was achieved by the author, see \cite[Section~4]{Gaudin2022}.

To see it, one may stare at Theorem \ref{thm:HodgeLapL2Rn+}, \cite[Propositions~4.3~\&~4.6]{Gaudin2022}, formula \eqref{eq:HodgeReflexExtOp}, Lemma  \ref{lem:ExtOpRn+DiffFormHspBspq} and Theorem \ref{thm:MetaThm1HodgeLaplacianRn+} until the next lemma becomes quite clear

\begin{lemma}\label{lem:Hs+2pestimateHodgeLap}Let $p\in(1,+\infty)$, $s\in(-1+1/p,1/p)$, $k\in \llb 0,n\rrb$. For all $u \in\dot{\mathrm{D}}^{s}_{p}({{\Delta}}_{\mathcal{H},\mathfrak{t}},\Lambda^k)$, we have
\begin{align*}
    {{\Delta}}_{\mathcal{H},\mathfrak{t}} u = \sum_{I\in\mathcal{I}^{k}_{n-1}} \Delta_{\mathcal{N}} u_{I}\,\mathrm{d}{x_I} + \sum_{I'\in\mathcal{I}^{k-1}_{n-1}} \Delta_{\mathcal{D}} u_{I',n}\,\mathrm{d}{x_{I'}}\wedge\mathrm{d}x_n  \text{. }
\end{align*}
We also have estimates,
\begin{align*}
    \lVert \delta \mathrm{d} u \rVert_{\dot{\mathrm{H}}^{s,p}(\mathbb{R}^n_+)} + \lVert \mathrm{d}\delta  u \rVert_{\dot{\mathrm{H}}^{s,p}(\mathbb{R}^n_+)} &\sim_{p,s,n} \sum_{I\in\mathcal{I}^{k}_{n-1}} \lVert \Delta_{\mathcal{N}} u_{I} \rVert_{\dot{\mathrm{H}}^{s,p}(\mathbb{R}^n_+)} + \sum_{I'\in\mathcal{I}^{k-1}_{n-1}} \lVert \Delta_{\mathcal{D}} u_{I',n} \rVert_{\dot{\mathrm{H}}^{s,p}(\mathbb{R}^n_+)}\\
    & \sim_{p,s,n} \lVert \nabla^2 u \rVert_{\dot{\mathrm{H}}^{s,p}(\mathbb{R}^n_+)}\\
    & \sim_{p,s,n} \lVert u \rVert_{\dot{\mathrm{H}}^{s+2,p}(\mathbb{R}^n_+)}\text{. }
\end{align*}
The result still holds replacing $(\mathfrak{t},\mathcal{N},\mathcal{D})$ by $(\mathfrak{n},\mathcal{D},\mathcal{N})$.
\end{lemma}

And for the same reasons, one has more generally,
\begin{lemma}\label{lem:FractionalSobestimateHodgeLap}Let $p\in(1,+\infty)$, $s\in(-1+1/p,1/p)$, $\alpha\in[0,2]$, such that $s+\alpha\neq 1/p,1+1/p$, $k\in \llb 0,n\rrb$. For all $u \in\dot{\mathrm{D}}^{s}_{p}((-{{\Delta}}_{\mathcal{H}})^\frac{\alpha}{2},\Lambda)$, we have
\begin{align*}
    \lVert (-{{\Delta}}_{\mathcal{H}})^\frac{\alpha}{2} u \rVert_{\dot{\mathrm{H}}^{s,p}(\mathbb{R}^n_+)} \sim_{p,s,\alpha,n} \lVert u \rVert_{\dot{\mathrm{H}}^{s+\alpha,p}(\mathbb{R}^n_+)}\sim_{p,s,\alpha,n} \lVert (-{{\Delta}}_{\mathcal{H}})^\frac{s+\alpha}{2} u \rVert_{\mathrm{L}^{p}(\mathbb{R}^n_+)}\text{. }
\end{align*}
Recalling that in particular, ${\Delta_{\mathcal{H},\mathfrak{t}}}_{|_{\Lambda^0}}=\Delta_{\mathcal{N}}$ and  ${\Delta_{\mathcal{H},\mathfrak{n}}}_{|_{\Lambda^0}}=\Delta_{\mathcal{D}}$.
\end{lemma}

In general, explicit description for interpolation spaces with boundary condition may be quite tedious. We mention the work of Guidetti \cite{Guidetti1991Ell,Guidetti1991Interp}, where such investigation is done. Guidetti's results were used to make a extensive treatment of elliptic boundary value problem with general Lopatinskii-Shapiro boundary conditions in inhomogeneous Besov spaces on the half-space and on bounded domains with smooth boundary.

Thanks to Lemmas \ref{lem:Hs+2pestimateHodgeLap} and \ref{lem:FractionalSobestimateHodgeLap}, the current work will be reduced to Dirichlet and Neumann boundary conditions in the homogeneous case which is unknown to the author's knowledge yet.

For $p\in(1,+\infty)$, $q\in[1,+\infty]$, $s\in(-1+1/p,2+1/p)$, such that \eqref{AssumptionCompletenessExponents} is satisfied, we set:
\begin{align*}
    &\dot{\mathrm{B}}^{s}_{p,q,\mathcal{D}}(\mathbb{R}^n_+) := \begin{cases} \dot{\mathrm{B}}^{s}_{p,q}(\mathbb{R}^n_+) &\text{, if (} s < \frac{1}{p}\text{),}\\
    \left\{ \,u\in \dot{\mathrm{B}}^{s}_{p,q}(\mathbb{R}^n_+)\, \Big{|}\, u_{|_{\partial\mathbb{R}^n_+}}=0 \,\right\} &\text{, if (}s > \frac{1}{p}\text{) or (}s = \frac{1}{p}\text{ and }q=1\text{). }
    \end{cases}\\
    &\dot{\mathrm{B}}^{s}_{p,q,\mathcal{N}}(\mathbb{R}^n_+) := \begin{cases} \dot{\mathrm{B}}^{s}_{p,q}(\mathbb{R}^n_+) &\text{, if (} s < 1+ \frac{1}{p}\text{), }\\
    \left\{ \,u\in \dot{\mathrm{B}}^{s}_{p,q}(\mathbb{R}^n_+)\, \Big{|}\, \partial_{\nu}u_{|_{\partial\mathbb{R}^n_+}}=0\,\right\} &\text{, if (}s >1+ \frac{1}{p}\text{) or (}s = 1+\frac{1}{p}\text{ and }q=1\text{). }
    \end{cases}
\end{align*}
and similarly $(\mathcal{C}_{s,p})$ is satisfied, we also set:
\begin{align*}
    &\dot{\mathrm{H}}^{s,p}_{\mathcal{D}}(\mathbb{R}^n_+) := \begin{cases} \dot{\mathrm{H}}^{s,p}(\mathbb{R}^n_+) &\text{, if } s < \frac{1}{p}\text{, }\\
    \left\{ \,u\in \dot{\mathrm{H}}^{s,p}(\mathbb{R}^n_+)\, \Big{|}\, u_{|_{\partial\mathbb{R}^n_+}}=0 \,\right\} &\text{, if }s > \frac{1}{p}\text{. }
    \end{cases}\\
    &\dot{\mathrm{H}}^{s,p}_{\mathcal{N}}(\mathbb{R}^n_+) := \begin{cases} \dot{\mathrm{H}}^{s,p}(\mathbb{R}^n_+) &\text{, if } s < 1+ \frac{1}{p}\text{, }\\
    \left\{ \,u\in \dot{\mathrm{H}}^{s,p}(\mathbb{R}^n_+)\, \Big{|}\, \partial_{\nu}u_{|_{\partial\mathbb{R}^n_+}}=0\,\right\} &\text{, if }s >1+ \frac{1}{p}\text{. }
    \end{cases}
\end{align*}

And then, for $\mathcal{J}\in\{\mathcal{D},\mathcal{N}\}$, we introduce the following subspace
\begin{align*}
    \mathrm{Y}_{\mathcal{J}} := \bigcap_{\substack{s\in(-1+1/p,1/p) \\ p\in(1,+\infty)\\ q\in[1,+\infty]}} [{\mathrm{D}}^{s}_{p}\cap\dot{\mathrm{D}}^{s}_{p}\cap{\mathrm{D}}^{s}_{p,q}\cap\dot{\mathrm{D}}^{s}_{p,q}]({{\Delta}}_{\mathcal{J}}) \text{. }
\end{align*}

\begin{proposition}\label{prop:DensityResultHomNeumDirBesov} Let $p\in(1,+\infty)$, $q\in[1,+\infty)$, $s\in(-1+1/p,2+1/p)$, such that \eqref{AssumptionCompletenessExponents} is satisfied, we have that
\begin{align*}
    \dot{\mathrm{B}}^{s}_{p,q,\mathcal{J}}(\mathbb{R}^n_+) = \overline{\mathrm{Y}_{\mathcal{J}}}^{\lVert\cdot\rVert_{\dot{\mathrm{B}}^{s}_{p,q}(\mathbb{R}^n_+)}} \text{, }
\end{align*}
whenever
\begin{itemize}
    \item $\mathcal{J}=\mathcal{D}$, $s\neq 1/p,1+1/p$ ;
    \item $\mathcal{J}=\mathcal{N}$, $s\neq 1+1/p$.
\end{itemize}
When $q=+\infty$, we still have weak${}^\ast$ density.
\end{proposition}

\begin{proof} We recall that for all $\tilde{p}\in(1,+\infty)$, $\tilde{q}\in[1,+\infty)$, $\tilde{s}\in\mathbb{R}$, such that \eqref{AssumptionCompletenessExponents} is satisfied, we have that
\begin{align*}
    \dot{\mathrm{B}}^{\tilde{s}}_{\tilde{p},\tilde{q}}(\mathbb{R}^n_+) = \overline{\eus{S}_0(\overline{\mathbb{R}^n_+})}^{\lVert\cdot\rVert_{\dot{\mathrm{B}}^{\tilde{s}}_{\tilde{p},\tilde{q}}(\mathbb{R}^n_+)}} \text{. }
\end{align*}
\begin{itemize}
    \item First, assume that $s\in(1+1/p,2+1/p)$, for $u\in\dot{\mathrm{B}}^{s}_{p,q,\mathcal{J}}(\mathbb{R}^n_+)$, we set $f:=-\Delta_{\mathcal{J}}u \in\dot{\mathrm{B}}^{s-2}_{p,q}(\mathbb{R}^n_+)$. For $(f_j)_{j\in\mathbb{N}} \subset \eus{S}_0(\overline{\mathbb{R}^n_+})$ such that
\begin{align*}
    f_j \xrightarrow[j\rightarrow +\infty]{} f\text{, }\quad \text{ in } \dot{\mathrm{B}}^{s-2}_{p,q}(\mathbb{R}^n_+)\text{.}
\end{align*}
We set for all $\lambda>0$,
\begin{align*}
    u_{\lambda,j} := (\lambda\mathrm{I}-\Delta_{\mathcal{J}})^{-1}f_j \in \mathrm{Y}_{\mathcal{J}}
\end{align*}
where belonging to space $\mathrm{Y}_{\mathcal{J}}$ is a consequence of \cite[Propositions~4.4~\&~4.6]{Gaudin2022}. For $\mu,\lambda>0$, by Lemma \ref{lem:Hs+2pestimateHodgeLap}, we have
\begin{align*}
    \left\lVert u_{j,\lambda}-u_{j,\mu} \right\rVert_{\dot{\mathrm{B}}^{s}_{p,q}(\mathbb{R}^n_+)} &\lesssim_{s,p,q,n} \left\lVert -\Delta_{\mathcal{J}}(\lambda\mathrm{I}-\Delta_{\mathcal{J}})^{-1}f_j + \Delta_{\mathcal{J}}(\mu\mathrm{I}-\Delta_{\mathcal{J}})^{-1}f_j \right\rVert_{\dot{\mathrm{B}}^{s-2}_{p,q}(\mathbb{R}^n_+)}\\
    &\lesssim_{s,p,q,n} \left\lVert -\Delta_{\mathcal{J}}(\lambda\mathrm{I}-\Delta_{\mathcal{J}})^{-1}f_j -f_j\right\rVert_{\dot{\mathrm{B}}^{s-2}_{p,q}(\mathbb{R}^n_+)}\\
    &\qquad\quad +\left\lVert f_j + \Delta_{\mathcal{J}}(\mu\mathrm{I}-\Delta_{\mathcal{J}})^{-1}f_j\right\rVert_{\dot{\mathrm{B}}^{s-2}_{p,q}(\mathbb{R}^n_+)}
    \xrightarrow[\lambda,\mu\rightarrow 0]{} 0 \text{. }
\end{align*}
By uniqueness of the solution for the Neumann (resp. Dirichlet) problem provided by \cite[Propositions~4.8]{Gaudin2022} (resp. \cite[Propositions~4.6]{Gaudin2022}), $(u_{\mu,j})_{\mu>0}$ is a Cauchy net that admits a limit that must be the unique solution $u_j$. The strong continuity of $(-\Delta_\mathcal{J})^{-1}$ concludes the case of $s\in(1+1/p,2+1/p)$.

    \item For $s\in (1/p,1+1/p)$, we consider first the case $\mathcal{J}=\mathcal{N}$. Again, for $u\in\dot{\mathrm{B}}^{s}_{p,q,\mathcal{N}}(\mathbb{R}^n_+)=\dot{\mathrm{B}}^{s}_{p,q}(\mathbb{R}^n_+)$, we can introduce $(u_j)_{j\in\mathbb{N}} \subset \eus{S}_0(\overline{\mathbb{R}^n_+})$ such that
\begin{align*}
    u_j \xrightarrow[j\rightarrow +\infty]{\dot{\mathrm{B}}^{s}_{p,q}} u\text{. }
\end{align*}
We set for all $\tau>0$,
\begin{align*}
    u_{\tau,j} := (\mathrm{I}-\tau\Delta_{\mathcal{J}})^{-1}u_j \in \mathrm{Y}_{\mathcal{J}}
\end{align*}
where belonging to space $\mathrm{Y}_{\mathcal{J}}$ is a consequence of \cite[Propositions~4.3~\&~4.6]{Gaudin2022}. It is direct to see that
\begin{align*}
    u_{\tau,j}\xrightarrow[\tau\rightarrow 0]{} u_{j}\xrightarrow[j\rightarrow +\infty]{} u \quad \text{ in } \dot{\mathrm{B}}^{s}_{p,q}(\mathbb{R}^n_+)\text{.}
\end{align*}

For the case $\mathcal{J}=\mathcal{D}$, since $\dot{\mathrm{B}}^{s}_{p,q,\mathcal{D}}(\mathbb{R}^n_+)=\dot{\mathrm{B}}^{s}_{p,q,0}(\mathbb{R}^n_+)$ and $\mathrm{C}_c^\infty(\mathbb{R}^n_+) \subset \mathrm{Y}_{\mathcal{\mathcal{D}}}$, the result follows from \cite[Lemma~2.32]{Gaudin2022}.

This argument still works for $s\in(-1+1/p,1+1/p)$, when $\mathcal{J}=\mathcal{N}$.

\item Finally $s\in (-1+1/p,1/p)$, we notice that $\dot{\mathrm{B}}^{s}_{p,q,\mathcal{J}}(\mathbb{R}^n_+)=\dot{\mathrm{B}}^{s}_{p,q,0}(\mathbb{R}^n_+)=\dot{\mathrm{B}}^{s}_{p,q}(\mathbb{R}^n_+)$. Since $\mathrm{C}_c^\infty(\mathbb{R}^n_+) \subset \mathrm{Y}_{\mathcal{\mathcal{J}}}$, the result follows from \cite[Corollary~2.34]{Gaudin2022}.
\end{itemize}
\end{proof}

The next result has a similar proof, left to the reader.

\begin{proposition}\label{prop:DensityResultHomNeumÎntersecSobol} Let $p\in(1,+\infty)$, $s_0,s_1\in(1/p,2+1/p)$, $\mathcal{J}\in\{\mathcal{D},\mathcal{N}\}$ such that $(\mathcal{C}_{s_0,p})$ is satisfied, we have
\begin{align*}
    [\dot{\mathrm{H}}^{s_0,p}_{\mathcal{J}}\cap{\dot{\mathrm{H}}^{s_1,p}}](\mathbb{R}^n_+) = \overline{\mathrm{Y}_{\mathcal{J}}}^{\lVert\cdot\rVert_{[\dot{\mathrm{H}}^{s_0,p}\cap{\dot{\mathrm{H}}^{s_1,p}}](\mathbb{R}^n_+)}} \text{, }
\end{align*}
whenever $1/p<s_0,s_1<2+1/p$, except for $s=1+1/p$ when $\mathcal{J}=\mathcal{N}$.
\end{proposition}

The next lemma is inspired from \cite[Lemma~2.4]{Guidetti1991Interp}.
\begin{lemma}\label{lem:ProjectDirNeuBC}Let $p_j\in(1,+\infty)$, $q_j\in[1,+\infty)$, and $s_j>1/p_j$, for $j\in\{ 0,1\}$. Let $T$ be the map
\begin{align*}
 T\,:\,f   &   \longmapsto    \left[    (x',x_n)\mapsto e^{-x_n(-\Delta')^\frac{1}{2}}f(x') \right] \text{.}
\end{align*}

\begin{enumerate}
    \item Assume $s_j\in( 1/p_j,1+2/p_j)$, for $j\in\{0,1\}$. Then the operator defined formally by
    \begin{align*}
        \mathcal{P}_{\mathcal{D}} u := u - T[u_{|_{\partial\mathbb{R}^n_+}}]\text{, }
    \end{align*}
    is such that
    \begin{enumerate}
        \item If $(\mathcal{C}_{s_0,p_0})$ is satisfied, then $\mathcal{P}_{\mathcal{D}} \,:\,[\dot{\mathrm{H}}^{s_0,p_0}\cap\dot{\mathrm{H}}^{s_1,p_1}](\mathbb{R}^n_+) \longrightarrow [\dot{\mathrm{H}}^{s_0,p_0}_{\mathcal{D}}\cap\dot{\mathrm{H}}^{s_1,p_1}](\mathbb{R}^n_+)$ is a well defined linear and bounded projection. For all $u\in [\dot{\mathrm{H}}^{s_0,p_0}\cap\dot{\mathrm{H}}^{s_1,p_1}](\mathbb{R}^n_+)$ the following estimate is true
        \begin{align*}
         \lVert \mathcal{P}_{\mathcal{D}} u \rVert_{\dot{\mathrm{H}}^{s_j,p_j}(\mathbb{R}^n_+)} \lesssim_{p_j,s_j,n} \lVert  u \rVert_{\dot{\mathrm{H}}^{s_j,p_j}(\mathbb{R}^n_+)} \quad\text{, } j\in\{0,1\} \text{. }
        \end{align*}
        \item If instead $(\mathcal{C}_{s_0,p_0,q_0})$ is satisfied, then the above statement still holds with $(\dot{\mathrm{B}}^{s_0}_{p_0,q_0},\dot{\mathrm{B}}^{s_1}_{p_1,q_1})$ replacing $(\dot{\mathrm{H}}^{s_0,p_0},\dot{\mathrm{H}}^{s_1,p_1})$. 
        
        We also have that $\mathcal{P}_{\mathcal{D}}\,:\,\dot{\mathrm{B}}^{s_0}_{p_0,\infty}(\mathbb{R}^n_+)\longrightarrow \dot{\mathrm{B}}^{s_0}_{p_0,\infty,\mathcal{D}}(\mathbb{R}^n_+)$ is also well defined linear and bounded.
    \end{enumerate}
    \item Assume $s_j\in( 1+1/p_j,1+2/p_j)$, for $j\in\{0,1\}$. Then the operator defined formally by
    \begin{align*}
        \mathcal{P}_{\mathcal{N}} u := u + (-\Delta')^{-\frac{1}{2}}T[\partial_{x_n}u_{|_{\partial\mathbb{R}^n_+}}]\text{, }
    \end{align*}
    satisfies points (i)(a) and (i)(b) with $\mathcal{N}$ instead of $\mathcal{D}$.
\end{enumerate}
\end{lemma}

\begin{proof} This a direct consequence of \cite[Proposition~B.2,~Corollary~B.3]{Gaudin2022}.
\end{proof}

\begin{proposition}\label{prop:assumptionsareCheckedforLap} Let $p\in(1,+\infty)$, $s\in(-1+1/p,1/p)$, $\mathcal{J}\in\{\mathcal{D},\mathcal{N},\mathcal{H}\}$, then $(\dot{\mathrm{D}}^{s}_{p}(\Delta_\mathcal{J}),-\Delta_\mathcal{J})$ satisfies Assumptions \eqref{asmpt:homogeneousdomaindef} and \eqref{asmpt:homogeneousdomainintersect}. In other words, $-\Delta_\mathcal{J}$ is injective on $\dot{\mathrm{H}}^{s,p}(\mathbb{R}^n_+)$, and we can define
\begin{align*}
    \dot{\mathrm{D}}^{s}_{p}(\mathring{\Delta}_\mathcal{J}) := \{ u\in \dot{\mathrm{H}}^{s+2,p}(\mathbb{R}^n_+)\,|\,\exists (u_j)_{j\in\mathbb{N}}\subset\dot{\mathrm{D}}^{s}_{p}(\Delta_\mathcal{J}),\,\lVert u-u_j \rVert_{\dot{\mathrm{H}}^{s+2,p}(\mathbb{R}^n_+)} \xrightarrow[j\rightarrow +\infty]{} 0  \}
\end{align*}
such that it also satisfies
\begin{align}\label{eq:IntersectHomogeneousDomainsLap}
    \dot{\mathrm{H}}^{s,p}(\mathbb{R}^n_+)\cap\dot{\mathrm{D}}^{s}_{p}(\mathring{\Delta}_\mathcal{J}) = \dot{\mathrm{D}}^{s}_{p}({\Delta}_\mathcal{J})\text{. }
\end{align}
The result holds yields with either $\mathbb{A}_{\mathcal{H},\mathfrak{t}}$ (resp. $\mathbb{M}_{\mathcal{H},\mathfrak{t}}$) on $\dot{\mathrm{H}}^{s,p}_{\mathfrak{t},\sigma}(\mathbb{R}^n_+)$ (resp. $\dot{\mathrm{H}}^{s,p}_{\gamma}(\mathbb{R}^n_+)$), and similarly replacing $(\mathfrak{t},\sigma,\gamma,\mathbb{A},\mathbb{M})$ by $(\mathfrak{n},\gamma,\sigma,\mathbb{M},\mathbb{A})$.
\end{proposition}

\begin{proof} We only show \eqref{eq:IntersectHomogeneousDomainsLap}. The following inclusion is clear 
\begin{align*}
    \dot{\mathrm{D}}^{s}_{p}({\Delta}_\mathcal{J}) \subset\dot{\mathrm{H}}^{s,p}(\mathbb{R}^n_+)\cap\dot{\mathrm{D}}^{s}_{p}(\mathring{\Delta}_\mathcal{J}) \text{. }
\end{align*}
Now, let $u\in \dot{\mathrm{H}}^{s,p}(\mathbb{R}^n_+)\cap\dot{\mathrm{D}}^{s}_{p}(\mathring{\Delta}_\mathcal{J})$, then obviously
\begin{align*}
    u\in \dot{\mathrm{H}}^{s,p}(\mathbb{R}^n_+)\cap\dot{\mathrm{H}}^{s+2,p}(\mathbb{R}^n_+) \text{. }
\end{align*}
It suffices to show that $u$ has appropriate boundary conditions. We assume here that $\mathcal{J}=\mathcal{D}$, other cases would be achieved similarly. Let $(u_j)_{j\in\mathbb{N}}\subset\dot{\mathrm{D}}^{s}_{p}(\Delta_\mathcal{D})$, such that
\begin{align*}
    \lVert u-u_j \rVert_{\dot{\mathrm{H}}^{s+2,p}(\mathbb{R}^n_+)} \xrightarrow[j\rightarrow +\infty]{} 0 \text{. }
\end{align*}
Since $u-u_j\in \dot{\mathrm{H}}^{s,p}(\mathbb{R}^n_+)\cap\dot{\mathrm{H}}^{s+2,p}(\mathbb{R}^n_+)$, one may apply \cite[Proposition~3.8]{Gaudin2022}, and use ${u_j}_{|_{\partial\mathbb{R}^n_+}}=0$ to obtain
\begin{align*}
    \lVert u_{|_{\partial\mathbb{R}^n_+}}\rVert_{\dot{\mathrm{B}}^{s+2-{1/p}}_{p,p}(\mathbb{R}^{n-1})}\lesssim_{s,p,n}\lVert u-u_j \rVert_{\dot{\mathrm{H}}^{s+2,p}(\mathbb{R}^n_+)} \xrightarrow[j\rightarrow +\infty]{} 0\text{.}
\end{align*}
Therefore $u_{|_{\partial\mathbb{R}^n_+}}=0$ so that $u\in\dot{\mathrm{D}}^{s}_{p}({\Delta}_\mathcal{D})$.
\end{proof}

The Proposition \ref{prop:assumptionsareCheckedforLap} tells us that, for all $p\in(1,+\infty)$, $s\in(-1+1/p,1/p)$, it makes sense to consider the semigroup,
\begin{align*}
    e^{t\Delta_\mathcal{H}}\,:\,\dot{\mathrm{H}}^{s,p}(\mathbb{R}^n_+)+\dot{\mathrm{D}}^{s}_{p}(\mathring{\Delta}_\mathcal{H})\longrightarrow\dot{\mathrm{H}}^{s,p}(\mathbb{R}^n_+)+\dot{\mathrm{D}}^{s}_{p}(\mathring{\Delta}_\mathcal{H}),
\end{align*}
thanks to \cite[Chapter~2,~Section~1]{DanchinHieberMuchaTolk2020}.

For convenience of notations, and for later use, one may think about Lemma \ref{lem:Hs+2pestimateHodgeLap}, we also set for all $p\in(1,+\infty)$, $q\in[1,+\infty]$, $s\in(-1+1/p,1+2/p)$, such that \eqref{AssumptionCompletenessExponents} is satisfied, and for $k\in\llb 0,n\rrb$,
\begin{align}\label{eq:defhomBesovwithCompCond}
    \dot{\mathrm{B}}^{s}_{p,q,\mathcal{H}_\mathfrak{t}}(\mathbb{R}^n_+,\Lambda^k) := \begin{cases} \dot{\mathrm{B}}^{s}_{p,q}(\mathbb{R}^n_+,\Lambda^k) & \text{, }s <  \frac{1}{p}\text{, }\\
    \left\{ \,u\in \dot{\mathrm{B}}^{s}_{p,q}(\mathbb{R}^n_+,\Lambda^k)\, \big{|}\, \mathfrak{e}_n\iprod u_{|_{\partial\mathbb{R}^n_+}}=0 \,\right\} &\text{, }\frac{1}{p} < s < 1+\frac{1}{p}\text{,  }\\
    \left\{ \,u\in \dot{\mathrm{B}}^{s}_{p,q}(\mathbb{R}^n_+,\Lambda^k)\, \big{|}\,  \mathfrak{e}_n\iprod u_{|_{\partial\mathbb{R}^n_+}},\,\mathfrak{e}_n\iprod\mathrm{d} u_{|_{\partial\mathbb{R}^n_+}}=0\right\} &\text{, }1+\frac{1}{p} < s < 2+\frac{1}{p}\text{. }
    \end{cases}
\end{align}
It is not difficult to see from point \textit{(iii)} of Theorem \ref{thm:TracesDifferentialformsHomSpaces}, that,
\begin{align*}
    \dot{\mathrm{B}}^{s}_{p,q,\mathcal{H}_\mathfrak{t}}(\mathbb{R}^n_+,\Lambda^k) \simeq \dot{\mathrm{B}}^{s}_{p,q,\mathcal{D}}(\mathbb{R}^n_+)^{\binom{n-1}{k-1}}\times\dot{\mathrm{B}}^{s}_{p,q,\mathcal{N}}(\mathbb{R}^n_+)^{\binom{n-1}{k}} \text{, }
\end{align*}
for which one may check for instance the \textbf{Step 3} of Theorem \ref{thm:TracesDifferentialformsHomSpaces}'s proof.

One can also build in the same fashion $\dot{\mathrm{B}}^{s}_{p,q,\mathcal{H}_\mathfrak{n}}(\mathbb{R}^n_+,\Lambda^k)$, with boundary conditions $\nu\wedge u_{|_{\partial\mathbb{R}^n_+}}=0$ and $\nu\wedge \delta u_{|_{\partial\mathbb{R}^n_+}}=0$, so that
\begin{align*}
    \dot{\mathrm{B}}^{s}_{p,q,\mathcal{H}_\mathfrak{n}}(\mathbb{R}^n_+,\Lambda^k) \simeq \dot{\mathrm{B}}^{s}_{p,q,\mathcal{N}}(\mathbb{R}^n_+)^{\binom{n-1}{k-1}}\times\dot{\mathrm{B}}^{s}_{p,q,\mathcal{D}}(\mathbb{R}^n_+)^{\binom{n-1}{k}} \text{. }
\end{align*}
We denote by $\dot{\mathrm{B}}^{s}_{p,q,\mathcal{H}}(\mathbb{R}^n_+)$, either $\dot{\mathrm{B}}^{s}_{p,q,\mathcal{H}_\mathfrak{t}}(\mathbb{R}^n_+,\Lambda)$ or $\dot{\mathrm{B}}^{s}_{p,q,\mathcal{H}_\mathfrak{n}}(\mathbb{R}^n_+,\Lambda)$.

\begin{proposition}\label{prop:InterpHomDomainLaplacians} Let $p\in(1,+\infty)$, $q\in[1,+\infty]$, and $s\in(-1+1/p,1/p)$. For all $\theta\in(0,1)$ such that $(\mathcal{C}_{s+2\theta,p,q})$ is satisfied, provided $\mathcal{J}\in\{\mathcal{D},\mathcal{N},\mathcal{H}\}$, one has
\begin{align*}
    (\dot{\mathrm{H}}^{s,p}(\mathbb{R}^n_+),\dot{\mathrm{D}}^{s}_{p}(\mathring{\Delta}_\mathcal{J}))_{\theta,q} = \dot{\mathrm{B}}^{s+2\theta}_{p,q,\mathcal{J}}(\mathbb{R}^n_+)\text{, }
\end{align*}
with equivalence of norms, whenever
\begin{itemize}
    \item $s+2\theta\neq 1/p$ if $\mathcal{J}=\mathcal{D}$,
    \item $s+2\theta\neq 1+1/p$ if $\mathcal{J}=\mathcal{N}$,
    \item $s+2\theta\neq 1/p,1+1/p$ if $\mathcal{J}=\mathcal{H}$.
\end{itemize}
\end{proposition}

The proof is heavily inspired from the one of \cite[Proposition~4.12]{DanchinHieberMuchaTolk2020}.

\begin{proof}{\textbf{Step 1 :}}  We start applying \cite[Proposition~3.33]{Gaudin2022} which yields the embedding
\begin{align*}
    (\dot{\mathrm{H}}^{s,p}(\mathbb{R}^n_+),\dot{\mathrm{D}}^{s}_{p}(\mathring{\Delta}_\mathcal{J}))_{\theta,q} \hookrightarrow (\dot{\mathrm{H}}^{s,p}(\mathbb{R}^n_+),\dot{\mathrm{H}}^{s+2,p}(\mathbb{R}^n_+))_{\theta,q} = \dot{\mathrm{B}}^{s+2\theta}_{p,q}(\mathbb{R}^n_+) \text{. }
\end{align*}
Now, if $q<+\infty$, we recall that $\dot{\mathrm{H}}^{s,p}(\mathbb{R}^n_+)\cap\dot{\mathrm{D}}^{s}_{p}(\mathring{\Delta}_\mathcal{J}) = \dot{\mathrm{D}}^{s}_{p}({\Delta}_\mathcal{J})$ is dense in $(\dot{\mathrm{H}}^{s,p}(\mathbb{R}^n_+),\dot{\mathrm{D}}^{s}_{p}(\mathring{\Delta}_\mathcal{J}))_{\theta,q}$ by \cite[Theorem~3.4.2]{BerghLofstrom1976}, so that by continuity of traces,
\begin{align*}
    (\dot{\mathrm{H}}^{s,p}(\mathbb{R}^n_+),\dot{\mathrm{D}}^{s}_{p}(\mathring{\Delta}_\mathcal{J}))_{\theta,q} \hookrightarrow \dot{\mathrm{B}}^{s+2\theta}_{p,q,\mathcal{J}}(\mathbb{R}^n_+) \text{. }
\end{align*}
The case $q=+\infty$ will be done in later steps.

{\textbf{Step 2 :}} The reverse embedding when $s+2\theta\in(-1+1/p,1/p)$. Let $f\in \dot{\mathrm{D}}^{s}_{p}({{\Delta}}_{\mathcal{J}})$, then for all $t>0$
\begin{align*}
    f = e^{t\Delta_{\mathcal{J}}}f + \mathring{\Delta}_{\mathcal{J}}\int_{0}^{t} \tau e^{\tau\Delta_{\mathcal{J}}}f \frac{\mathrm{d}\tau}{\tau} =: b + a 
\end{align*}
with obviously $f\in \dot{\mathrm{D}}^{s}_{p}({{\Delta}}_{\mathcal{J}})\subset \dot{\mathrm{H}}^{s,p}(\mathbb{R}^n_+)+\dot{\mathrm{D}}^{s}_{p}(\mathring{\Delta}_\mathcal{J})$ and by definition of the $K$-functional, we obtain
\begin{align*}
    K(t,f,\dot{\mathrm{H}}^{s,p}(\mathbb{R}^n_+),\dot{\mathrm{D}}^{s}_{p}(\mathring{\Delta}_\mathcal{J})) \leqslant \lVert a \rVert_{\dot{\mathrm{H}}^{s,p}(\mathbb{R}^n_+)} + t\lVert {\Delta}_\mathcal{J} b \rVert_{\dot{\mathrm{H}}^{s,p}(\mathbb{R}^n_+)} \text{. }
\end{align*}
So as in  the proof of \cite[Proposition~4.12]{DanchinHieberMuchaTolk2020}, we may apply \cite[Lemma~2.11]{DanchinHieberMuchaTolk2020} so that
\begin{align}\label{eq:referenceEstimate}
    \lVert f \rVert_{(\dot{\mathrm{H}}^{s,p}(\mathbb{R}^n_+),\dot{\mathrm{D}}^{s}_{p}(\mathring{\Delta}_\mathcal{J}))_{\theta,q}} \leqslant \frac{1+\theta}{\theta} \left( \int_{0}^{+\infty} \lVert t^{1-\theta} \Delta_{\mathcal{J}}e^{t\Delta_{\mathcal{J}}} f\rVert_{\dot{\mathrm{H}}^{s,p}(\mathbb{R}^n_+)}^q \frac{\mathrm{d}t}{t} \right)^\frac{1}{q} \text{.}
\end{align}
Now, for $0<\theta<\varepsilon$, such that $s+2\theta<s+2\varepsilon<1/p$ we want to bound the $\mathrm{L}^q_\ast$-norm of $\lVert t^{1-\theta} \Delta_{\mathcal{J}}e^{t\Delta_{\mathcal{J}}} f\rVert_{\dot{\mathrm{H}}^{s,p}(\mathbb{R}^n_+)}$ by the $\mathrm{L}^q_\ast$-norm of the $K$-functional associated with the real interpolation space $(\dot{\mathrm{H}}^{s,p}(\mathbb{R}^n_+),\dot{\mathrm{H}}^{s+2\varepsilon,p}(\mathbb{R}^n_+))_{\frac{\theta}{\varepsilon},q}$.

Let $(\tilde{a},\tilde{b})\in\dot{\mathrm{H}}^{s,p}(\mathbb{R}^n_+)\times\dot{\mathrm{H}}^{s+2\varepsilon,p}(\mathbb{R}^n_+)$, such that $f=\tilde{a}+\tilde{b}$, the fact that $f\in \dot{\mathrm{D}}^{s}_{p}({{\Delta}}_{\mathcal{J}})$ implies  $\tilde{a},\tilde{b}\in [\dot{\mathrm{H}}^{s,p}\cap\dot{\mathrm{H}}^{s+2\varepsilon,p}](\mathbb{R}^n_+)$. By Lemma \ref{lem:FractionalSobestimateHodgeLap}, since $s+2\varepsilon<1/p$ we have
\begin{align*}
   \tilde{b} \in \dot{\mathrm{D}}^{s}_{p}((-{{\Delta}}_{\mathcal{J}})^\varepsilon)
\end{align*}
Therefore, since the semigroup $(e^{t\Delta_\mathcal{J}})_{t>0}$ is analytic, by the use of Lemma \ref{lem:FractionalSobestimateHodgeLap}, we have
\begin{align*}
    \lVert t^{1-\theta} \Delta_{\mathcal{J}}e^{t\Delta_{\mathcal{J}}} f\rVert_{\dot{\mathrm{H}}^{s,p}(\mathbb{R}^n_+)} &\leqslant \lVert t^{1-\theta} \Delta_{\mathcal{J}}e^{t\Delta_{\mathcal{J}}} \tilde{a}\rVert_{\dot{\mathrm{H}}^{s,p}(\mathbb{R}^n_+)} + \lVert t^{1-\theta} \Delta_{\mathcal{J}}e^{t\Delta_{\mathcal{J}}} \tilde{b}\rVert_{\dot{\mathrm{H}}^{s,p}(\mathbb{R}^n_+)}\\
    &\lesssim_{n,p,s} t^{-\theta} \left( \lVert \tilde{a}\rVert_{\dot{\mathrm{H}}^{s,p}(\mathbb{R}^n_+)} + t^{\varepsilon} \lVert \tilde{b}\rVert_{\dot{\mathrm{H}}^{s+2\varepsilon,p}(\mathbb{R}^n_+)}\right)\text{. }
\end{align*}
Taking the infimum of all such $\tilde{a},\tilde{b}$, yields
\begin{align*}
    \lVert t^{1-\theta} \Delta_{\mathcal{J}}e^{t\Delta_{\mathcal{J}}} f\rVert_{\dot{\mathrm{H}}^{s,p}(\mathbb{R}^n_+)} \lesssim_{p,n,s} t^{-\theta}K(t^{\varepsilon},f,\dot{\mathrm{H}}^{s,p}(\mathbb{R}^n_+),\dot{\mathrm{H}}^{s+2\varepsilon,p}(\mathbb{R}^n_+)) \text{. }
\end{align*}
Therefore one may take the $\mathrm{L}^q_{\ast}$-norm on both sides, and use \cite[Proposition~3.33]{Gaudin2022}, to obtain for all $f\in \dot{\mathrm{D}}^{s}_{p}({{\Delta}}_{\mathcal{J}})$,
\begin{align*}
    \lVert f \rVert_{(\dot{\mathrm{H}}^{s,p}(\mathbb{R}^n_+),\dot{\mathrm{D}}^{s}_{p}(\mathring{\Delta}_\mathcal{J}))_{\theta,q}}\sim_{p,s,n,\theta,\varepsilon} \lVert f \rVert_{(\dot{\mathrm{H}}^{s,p}(\mathbb{R}^n_+),\dot{\mathrm{H}}^{s+2\varepsilon}(\mathbb{R}^n_+))_{\frac{\theta}{\varepsilon},q}} \sim_{p,s,n,\theta,\varepsilon} \lVert f \rVert_{\dot{\mathrm{B}}^{s+2\theta}_{p,q}(\mathbb{R}^n_+)}\text{. }
\end{align*}
If $q\in[1,+\infty)$, the result follows from \cite[Corollary~2.34]{Gaudin2022}, since $\mathrm{C}_c^\infty(\mathbb{R}^n_+)\subset \dot{\mathrm{D}}^{s}_{p}({{\Delta}}_{\mathcal{J}})$. The case $q=+\infty$ is obtained via the reiteration theorem \cite[Theorem~3.5.3]{BerghLofstrom1976}.

{\textbf{Step 3 :}} The reverse embedding when $s+2\theta\in(1/p,2+1/p)$, $\mathcal{J}=\mathcal{D}$.
Provided $f\in \mathrm{Y}_\mathcal{D}$, as introduced in Proposition \ref{prop:DensityResultHomNeumDirBesov}, we may reproduce above \textbf{Step 2} up to \eqref{eq:referenceEstimate}. From there, for $0<\eta<\theta$ such that $1/p<s+2\eta<s+2\theta$, we want to prove that one can bound \eqref{eq:referenceEstimate} by the $\mathrm{L}^q_\ast$-norm of the $K$-functional associated with the real interpolation space $(\dot{\mathrm{H}}^{s+2\eta,p}(\mathbb{R}^n_+),\dot{\mathrm{H}}^{s+2,p}(\mathbb{R}^n_+))_{\frac{\theta-\eta}{1-\eta},q}$.

Since $f\in\mathrm{Y}_\mathcal{D}\subset\dot{\mathrm{H}}^{s+2\eta,p}(\mathbb{R}^n_+)+\dot{\mathrm{H}}^{s+2,p}(\mathbb{R}^n_+)$, for $a,b\in\dot{\mathrm{H}}^{s+2\eta,p}(\mathbb{R}^n_+)\times\dot{\mathrm{H}}^{s+2,p}(\mathbb{R}^n_+)$ such that $f=a+b$, we get
\begin{align*}
    a=f-b \in \dot{\mathrm{H}}^{s+2\eta,p}(\mathbb{R}^n_+)\cap(\dot{\mathrm{H}}^{s+2,p}(\mathbb{R}^n_+)+\mathrm{Y}_\mathcal{D}) \subset \dot{\mathrm{H}}^{s+2\eta,p}(\mathbb{R}^n_+)\cap\dot{\mathrm{H}}^{s+2,p}(\mathbb{R}^n_+)\text{, }
\end{align*}
and the same arguments leads to $b \in \dot{\mathrm{H}}^{s+2\eta,p}(\mathbb{R}^n_+)\cap\dot{\mathrm{H}}^{s+2,p}(\mathbb{R}^n_+)$. From Lemma \ref{lem:ProjectDirNeuBC}, we have
\begin{align*}
    f = \mathcal{P}_\mathcal{D}f = \mathcal{P}_\mathcal{D}a + \mathcal{P}_\mathcal{D}b
\end{align*}
where $\mathcal{P}_\mathcal{D}a,\mathcal{P}_\mathcal{D}b\in \dot{\mathrm{H}}^{s+2\eta,p}_\mathcal{D}(\mathbb{R}^n_+)\cap\dot{\mathrm{H}}^{s+2,p}(\mathbb{R}^n_+)$ with estimates
\begin{align*}
    \lVert \mathcal{P}_\mathcal{D}a\rVert_{\dot{\mathrm{H}}^{s+2\eta,p}(\mathbb{R}^n_+)}\lesssim_{p,s,\eta,n} \lVert a \rVert_{\dot{\mathrm{H}}^{s+2\eta,p}(\mathbb{R}^n_+)}\,\text{ and }\,\lVert \mathcal{P}_\mathcal{D}b\rVert_{\dot{\mathrm{H}}^{s+2,p}(\mathbb{R}^n_+)}\lesssim_{p,s,n} \lVert b \rVert_{\dot{\mathrm{H}}^{s+2,p}(\mathbb{R}^n_+)}\text{. }
\end{align*}
Therefore, by above estimate, analyticity of the semigroup $(e^{t\Delta_{\mathcal{D}}})_{t>0}$, and Lemma \ref{lem:Hs+2pestimateHodgeLap}, we are able to obtain
\begin{align*}
    \lVert t^{1-\theta} \Delta_{\mathcal{D}}e^{t\Delta_{\mathcal{D}}} f\rVert_{\dot{\mathrm{H}}^{s,p}(\mathbb{R}^n_+)} &\leqslant \lVert t^{1-\theta} \Delta_{\mathcal{D}}e^{t\Delta_{\mathcal{D}}} \mathcal{P}_\mathcal{D} {a}\rVert_{\dot{\mathrm{H}}^{s,p}(\mathbb{R}^n_+)} + \lVert t^{1-\theta} \Delta_{\mathcal{D}}e^{t\Delta_{\mathcal{D}}} \mathcal{P}_\mathcal{D} {b}\rVert_{\dot{\mathrm{H}}^{s,p}(\mathbb{R}^n_+)}\\
    &\lesssim_{n,p,s} t^{-(\theta-\eta)} \lVert \mathcal{P}_\mathcal{D} {a}\rVert_{\dot{\mathrm{H}}^{s+2\eta,p}(\mathbb{R}^n_+)} + t^{1-\theta}\lVert   \mathcal{P}_\mathcal{D} {b}\rVert_{\dot{\mathrm{H}}^{s+2,p}(\mathbb{R}^n_+)} \\
    &\lesssim_{n,p,s} t^{-(\theta-\eta)} \left( \lVert {a}\rVert_{\dot{\mathrm{H}}^{s+2\eta,p}(\mathbb{R}^n_+)} + t^{1-\eta} \lVert {b}\rVert_{\dot{\mathrm{H}}^{s+2,p}(\mathbb{R}^n_+)}\right)\text{. }
\end{align*}
Taking the infimum of all such couple $({a},{b})$, yields
\begin{align*}
    \lVert t^{1-\theta} \Delta_{\mathcal{D}}e^{t\Delta_{\mathcal{D}}} f\rVert_{\dot{\mathrm{H}}^{s,p}(\mathbb{R}^n_+)} \lesssim_{p,n,s} t^{-(\theta-\eta)}K(t^{1-\eta},f,\dot{\mathrm{H}}^{s+2\eta,p}(\mathbb{R}^n_+),\dot{\mathrm{H}}^{s+2,p}(\mathbb{R}^n_+)) \text{. }
\end{align*}
As in the \textbf{Step 2}, one may take the $\mathrm{L}^q_{\ast}$-norm on both sides, and use \cite[Proposition~3.33]{Gaudin2022}, to obtain for all $f\in \mathrm{Y}_\mathcal{D}$,
\begin{align*}
    \lVert f \rVert_{(\dot{\mathrm{H}}^{s,p}(\mathbb{R}^n_+),\dot{\mathrm{D}}^{s}_{p}(\mathring{\Delta}_\mathcal{D}))_{\theta,q}}\sim_{p,s,n,\theta,\eta}\lVert f \rVert_{(\dot{\mathrm{H}}^{s+2\eta,p}(\mathbb{R}^n_+),\dot{\mathrm{H}}^{s+2,p}(\mathbb{R}^n_+)))_{\frac{\theta-\eta}{1-\eta},q}}\sim_{p,s,n,\theta,\eta} \lVert f \rVert_{\dot{\mathrm{B}}^{s+2\theta}_{p,q}(\mathbb{R}^n_+)}\text{. }
\end{align*}
If $q\in[1,+\infty)$, the result follows from Proposition \ref{prop:DensityResultHomNeumDirBesov}. The case $q=+\infty$ is obtained via the application of the reiteration theorem \cite[Theorem~3.5.3]{BerghLofstrom1976}, by the mean of Lemma \ref{lem:ProjectDirNeuBC}.

{\textbf{Step 4 :}} The reverse embedding $s+2\theta\in[1/p,1+1/p)$, $\mathcal{J}=\mathcal{N}$.
One may pick $f\in \mathrm{Y}_\mathcal{N}$ so that, as before, we can reproduce above \textbf{Step 2} up to \eqref{eq:referenceEstimate}. From there, for $0<\eta<\theta<\varepsilon$ such that $1/p<s+2\eta<s+2\theta<s+2\varepsilon<1+1/p$, we want to prove that one can bound \eqref{eq:referenceEstimate} by the $\mathrm{L}^q_\ast$-norm of the $K$-functional associated with the interpolation space $(\dot{\mathrm{H}}^{s+2\eta,p}(\mathbb{R}^n_+),\dot{\mathrm{H}}^{s+2\varepsilon,p}(\mathbb{R}^n_+))_{\frac{\theta-\eta}{\varepsilon-\eta},q}$.

Since $f\in\mathrm{Y}_\mathcal{D}\subset\dot{\mathrm{H}}^{s+2\eta,p}(\mathbb{R}^n_+)+\dot{\mathrm{H}}^{s+2\varepsilon,p}(\mathbb{R}^n_+)$, for $a,b\in\dot{\mathrm{H}}^{s+2\eta,p}(\mathbb{R}^n_+)\times\dot{\mathrm{H}}^{s+2\varepsilon,p}(\mathbb{R}^n_+)$ such that $f=a+b$, we get
\begin{align*}
    b=f-a \in \dot{\mathrm{H}}^{s+2\eta,p}(\mathbb{R}^n_+)\cap(\dot{\mathrm{H}}^{s+2\varepsilon,p}(\mathbb{R}^n_+)+\mathrm{Y}_\mathcal{N}) \subset \dot{\mathrm{H}}^{s+2\eta,p}(\mathbb{R}^n_+)\cap\dot{\mathrm{H}}^{s+2\varepsilon,p}(\mathbb{R}^n_+)\text{. }
\end{align*}
By Proposition \ref{prop:DensityResultHomNeumÎntersecSobol}, there exists sequences $(a_j)_{j\in\mathbb{N}},(b_j)_{j\in\mathbb{N}}\subset Y_{\mathcal{N}}$ such that
\begin{align*}
    \lVert a_j - a\rVert_{\dot{\mathrm{H}}^{s+2\eta,p}(\mathbb{R}^n_+)} + \lVert b_j - b\rVert_{[\dot{\mathrm{H}}^{s+2\eta,p}\cap\dot{\mathrm{H}}^{s+2\varepsilon,p}](\mathbb{R}^n_+)} \xrightarrow[j\rightarrow +\infty]{} 0 \text{. }
\end{align*}
Therefore, the analyticity of the semigroup $(e^{t\Delta_{\mathcal{N}}})_{t>0}$ and Lemma \ref{lem:Hs+2pestimateHodgeLap} works together to deliver
\begin{align*}
    \lVert t \Delta_{\mathcal{N}}e^{t\Delta_{\mathcal{N}}} a_j\rVert_{\dot{\mathrm{H}}^{s,p}(\mathbb{R}^n_+)} \lesssim_{p,s,n,\eta} t^{\eta} \lVert  a_j\rVert_{\dot{\mathrm{H}}^{s+2\eta,p}(\mathbb{R}^n_+)} \,\text{, }\\
    \lVert t \Delta_{\mathcal{N}}e^{t\Delta_{\mathcal{N}}} b_j\rVert_{\dot{\mathrm{H}}^{s,p}(\mathbb{R}^n_+)} \lesssim_{p,s,n,\varepsilon} t^{\varepsilon} \lVert  b_j\rVert_{\dot{\mathrm{H}}^{s+2\varepsilon,p}(\mathbb{R}^n_+)}\,\text{, }
\end{align*}
so that taking limits, it yields
\begin{align}
    \lVert t \Delta_{\mathcal{N}}e^{t\Delta_{\mathcal{N}}} a\rVert_{\dot{\mathrm{H}}^{s,p}(\mathbb{R}^n_+)} &\lesssim_{p,s,n,\eta} t^{\eta} \lVert  a\rVert_{\dot{\mathrm{H}}^{s+2\eta,p}(\mathbb{R}^n_+)} \text{, } \label{eq:analyticEstimSemigrpNeuForInterp}\\
    \lVert t \Delta_{\mathcal{N}}e^{t\Delta_{\mathcal{N}}} b\rVert_{\dot{\mathrm{H}}^{s,p}(\mathbb{R}^n_+)} &\lesssim_{p,s,n,\varepsilon} t^{\varepsilon} \lVert  b\rVert_{\dot{\mathrm{H}}^{s+2\varepsilon,p}(\mathbb{R}^n_+)}\text{. }\nonumber
\end{align}

Therefore, by the estimates \eqref{eq:analyticEstimSemigrpNeuForInterp}, the following holds
\begin{align*}
    \lVert t^{1-\theta} \Delta_{\mathcal{N}}e^{t\Delta_{\mathcal{N}}} f\rVert_{\dot{\mathrm{H}}^{s,p}(\mathbb{R}^n_+)} &\lesssim_{n,p,s,\eta,\varepsilon} t^{-(\theta-\eta)} \left( \lVert {a}\rVert_{\dot{\mathrm{H}}^{s+2\eta,p}(\mathbb{R}^n_+)} + t^{\varepsilon-\eta} \lVert {b}\rVert_{\dot{\mathrm{H}}^{s+2\varepsilon,p}(\mathbb{R}^n_+)}\right)\text{. }
\end{align*}
From there, we can take the infimum of all such couple $({a},{b})$, and we see that
\begin{align*}
    \lVert t^{1-\theta} \Delta_{\mathcal{N}}e^{t\Delta_{\mathcal{N}}} f\rVert_{\dot{\mathrm{H}}^{s,p}(\mathbb{R}^n_+)} \lesssim_{p,n,s} t^{-(\theta-\eta)}K(t^{\varepsilon-\eta},f,\dot{\mathrm{H}}^{s+2\eta,p}(\mathbb{R}^n_+),\dot{\mathrm{H}}^{s+2\varepsilon,p}(\mathbb{R}^n_+)) \text{. }
\end{align*}
As in the \textbf{Step 2}, one may take the $\mathrm{L}^q_{\ast}$-norm on both sides, and use \cite[Proposition~3.33]{Gaudin2022}, to obtain for all $f\in \mathrm{Y}_\mathcal{N}$,
\begin{align*}
    \lVert f \rVert_{(\dot{\mathrm{H}}^{s,p}(\mathbb{R}^n_+),\dot{\mathrm{D}}^{s}_{p}(\mathring{\Delta}_\mathcal{N}))_{\theta,q}}\sim_{p,s,n,\theta,\eta,\varepsilon}\lVert f \rVert_{(\dot{\mathrm{H}}^{s+2\eta,p}(\mathbb{R}^n_+),\dot{\mathrm{H}}^{s+2\varepsilon,p}(\mathbb{R}^n_+)))_{\frac{\theta-\eta}{\varepsilon-\eta},q}}\sim_{p,s,n,\theta,\eta,\varepsilon} \lVert f \rVert_{\dot{\mathrm{B}}^{s+2\theta}_{p,q}(\mathbb{R}^n_+)}\text{. }
\end{align*}
If $q\in[1,+\infty)$, the result follows from Proposition \ref{prop:DensityResultHomNeumDirBesov}. The case $q=+\infty$ is obtained via the application of the reiteration theorem \cite[Theorem~3.5.3]{BerghLofstrom1976}. The case $s=1/p$ follows from reiteration theorem \cite[Theorem~3.5.3]{BerghLofstrom1976} between \textbf{Step 2} and this one.

{\textbf{Step 5 :}} The reverse embedding $s+2\theta\in(1+1/p,2+1/p)$, $\mathcal{J}=\mathcal{N}$.
For $f\in \mathrm{Y}_\mathcal{N}$, we reproduce again the \textbf{Step 2} up to \eqref{eq:referenceEstimate}. Now let $0<\eta<\theta$ such that $1+1/p<s+2\eta<s+2\theta<2+1/p$, we want to achieve the same estimate obtained at the end of \textbf{Step 3}.

Since $f\in\mathrm{Y}_\mathcal{N}\subset\dot{\mathrm{H}}^{s+2\eta,p}(\mathbb{R}^n_+)+\dot{\mathrm{H}}^{s+2,p}(\mathbb{R}^n_+)$, for $a,b\in\dot{\mathrm{H}}^{s+2\eta,p}(\mathbb{R}^n_+)\times\dot{\mathrm{H}}^{s+2,p}(\mathbb{R}^n_+)$ such that $f=a+b$, we get
\begin{align*}
    b=f-a \in \dot{\mathrm{H}}^{s+2\eta,p}(\mathbb{R}^n_+)\cap(\dot{\mathrm{H}}^{s+2,p}(\mathbb{R}^n_+)+\mathrm{Y}_\mathcal{N}) \subset \dot{\mathrm{H}}^{s+2\eta,p}(\mathbb{R}^n_+)\cap\dot{\mathrm{H}}^{s+2,p}(\mathbb{R}^n_+)\text{. }
\end{align*}
We want to fall in the expected homogeneous domains, \textit{i.e.} to get back the Neumann boundary condition, to do so, we use Lemma \ref{lem:ProjectDirNeuBC}, and we get
\begin{align*}
    f = \mathcal{P}_\mathcal{N}f = \mathcal{P}_\mathcal{N}a +\mathcal{P}_\mathcal{N}b\text{, }
\end{align*}
with estimates
\begin{align*}
    \lVert \mathcal{P}_\mathcal{N}a\rVert_{\dot{\mathrm{H}}^{s+2\eta,p}(\mathbb{R}^n_+)} \lesssim_{p,n,s,\eta}\lVert a\rVert_{\dot{\mathrm{H}}^{s+2\eta,p}(\mathbb{R}^n_+)} \,\text{ and }\, \lVert \mathcal{P}_\mathcal{N}b\rVert_{\dot{\mathrm{H}}^{s+2,p}(\mathbb{R}^n_+)} \lesssim_{p,n,s,\eta}\lVert b\rVert_{\dot{\mathrm{H}}^{s+2,p}(\mathbb{R}^n_+)}\text{. }
\end{align*}
By Proposition \ref{prop:DensityResultHomNeumÎntersecSobol}, there exists sequences $(\mathfrak{a}_j)_{j\in\mathbb{N}},(\mathfrak{b}_j)_{j\in\mathbb{N}}\subset Y_{\mathcal{N}}$ such that
\begin{align*}
    \lVert \mathfrak{a}_j - \mathcal{P}_\mathcal{N}a\rVert_{\dot{\mathrm{H}}^{s+2\eta,p}(\mathbb{R}^n_+)} + \lVert \mathfrak{b}_j - \mathcal{P}_\mathcal{N}b\rVert_{[\dot{\mathrm{H}}^{s+2\eta,p}\cap\dot{\mathrm{H}}^{s+2,p}](\mathbb{R}^n_+)} \xrightarrow[j\rightarrow +\infty]{} 0 \text{. }
\end{align*}
as in \textbf{Step 4}, we obtain
\begin{align*}
    \lVert t \Delta_{\mathcal{N}}e^{t\Delta_{\mathcal{N}}} \mathcal{P}_\mathcal{N}a\rVert_{\dot{\mathrm{H}}^{s,p}(\mathbb{R}^n_+)} \lesssim_{p,s,n,\eta} t^{\eta} \lVert  a\rVert_{\dot{\mathrm{H}}^{s+2\eta,p}(\mathbb{R}^n_+)} \,\text{, }\\
    \lVert t \Delta_{\mathcal{N}}e^{t\Delta_{\mathcal{N}}} \mathcal{P}_\mathcal{N}b\rVert_{\dot{\mathrm{H}}^{s,p}(\mathbb{R}^n_+)} \lesssim_{p,s,n,}  t\lVert  b\rVert_{\dot{\mathrm{H}}^{s+2,p}(\mathbb{R}^n_+)}\text{. }
\end{align*}

Therefore, by above estimates, the following estimate holds
\begin{align*}
    \lVert t^{1-\theta} \Delta_{\mathcal{N}}e^{t\Delta_{\mathcal{N}}} f\rVert_{\dot{\mathrm{H}}^{s,p}(\mathbb{R}^n_+)} &\lesssim_{n,p,s} t^{-(\theta-\eta)} \left( \lVert {a}\rVert_{\dot{\mathrm{H}}^{s+2\eta,p}(\mathbb{R}^n_+)} + t^{1-\eta} \lVert {b}\rVert_{\dot{\mathrm{H}}^{s+2,p}(\mathbb{R}^n_+)}\right)\text{. }
\end{align*}
Finally one may finish the present \textbf{Step 5} with the same arguments present in \textbf{Step 3}.

{\textbf{Step 6 :}} The case $\mathcal{J}=\mathcal{H}$. Let $k\in\llb 0,n \rrb$, from Lemma \ref{lem:Hs+2pestimateHodgeLap}, we deduce that the following holds with equivalence of norms
\begin{align*}
    \dot{\mathrm{D}}^{s}_{p}(\mathring{\Delta}_\mathcal{H,\mathfrak{t}},\Lambda^k) \simeq \dot{\mathrm{D}}^{s}_{p}(\mathring{\Delta}_\mathcal{D})^{\binom{n-1}{k-1}}\times \dot{\mathrm{D}}^{s}_{p}(\mathring{\Delta}_\mathcal{N})^{\binom{n-1}{k}}
\end{align*}
The result is then immediate, by all above steps. The case of the Hodge Laplacian with generalized normal boundary conditions admits a similar proof.
\end{proof}

Finally we want to compute interpolation spaces for the Hodge-Stokes and the Hodge Maxwell operators. To do so, we set for all $p\in(1,+\infty)$, $q\in[1,+\infty]$, $s\in(-1+1/p,2+1/p)$, such that \eqref{AssumptionCompletenessExponents} is satisfied, provided, $k\in\llb 0,n\rrb$,
\begin{align}\label{eq:SolenoidalBesovspaceswithcompcond}
    \dot{\mathrm{B}}^{s,\sigma}_{p,q,\mathcal{H}_\mathfrak{t}}(\mathbb{R}^n_+,\Lambda^k) &:= \begin{cases} \dot{\mathrm{B}}^{s,\sigma}_{p,q,{\mathfrak{t}}}(\mathbb{R}^n_+,\Lambda^k) & \text{, }s <  \frac{1}{p}\text{, }\\
    \left\{ \,u\in \dot{\mathrm{B}}^{s}_{p,q,\mathcal{H}_{\mathfrak{t}}}(\mathbb{R}^n_+,\Lambda^k)\, \big{|}\, \delta u=0 \,\right\} &\text{, }\frac{1}{p} <  s < 2+\frac{1}{p}\text{, } s\neq 1+1/p \text{. }
    \end{cases}\\
    \dot{\mathrm{B}}^{s,\gamma}_{p,q,\mathcal{H}_\mathfrak{t}}(\mathbb{R}^n_+,\Lambda^k) &:= \begin{cases} \dot{\mathrm{B}}^{s,\gamma}_{p,q}(\mathbb{R}^n_+,\Lambda^k) & \text{, }s <  \frac{1}{p}\text{, }\\
    \left\{ \,u\in \dot{\mathrm{B}}^{s}_{p,q,\mathcal{H}_{\mathfrak{t}}}(\mathbb{R}^n_+,\Lambda^k)\, \big{|}\, \mathrm{d} u=0 \,\right\} &\text{, }\frac{1}{p} <  s < 2+\frac{1}{p}\text{, } s\neq 1+1/p \text{. }
    \end{cases}
\end{align}
One may build similarly $\dot{\mathrm{B}}^{s,\sigma}_{p,q,\mathcal{H}_\mathfrak{n}}(\mathbb{R}^n_+,\Lambda^k)$ and $\dot{\mathrm{B}}^{s,\gamma}_{p,q,\mathcal{H}_\mathfrak{n}}(\mathbb{R}^n_+,\Lambda^k)$ replacing $(\mathfrak{t},\sigma,\gamma,\mathrm{d},\delta)$ by $(\mathfrak{n},\gamma,\sigma,\delta,\mathrm{d})$.

\begin{proposition}\label{prop:InterpHomDomainHodgeMaxwellStokes} Let $p\in(1,+\infty)$, $q\in[1,+\infty]$, and $s\in(-1+1/p,1/p)$. For all $\theta\in(0,1)$ such that $(\mathcal{C}_{s+2\theta,p,q})$ is satisfied, one has
\begin{align}
    (\dot{\mathrm{H}}^{s,p}_{{\mathfrak{t}},\sigma}(\mathbb{R}^n_+),\dot{\mathrm{D}}^{s}_{p}(\mathring{\mathbb{A}}_{\mathcal{H},\mathfrak{t}}))_{\theta,q} &= \dot{\mathrm{B}}^{s+2\theta,\sigma}_{p,q,\mathcal{H}_{\mathfrak{t}}}(\mathbb{R}^n_+)\text{, }\label{eq:1stInterpolSolenoidal}\\
    (\dot{\mathrm{H}}^{s,p}_{\gamma}(\mathbb{R}^n_+),\dot{\mathrm{D}}^{s}_{p}(\mathring{\mathbb{M}}_{\mathcal{H},\mathfrak{t}}))_{\theta,q} &= \dot{\mathrm{B}}^{s+2\theta,\gamma}_{p,q,\mathcal{H}_{\mathfrak{t}}}(\mathbb{R}^n_+)\text{, }\label{eq:2ndInterpolSolenoidal}
\end{align}
with equivalence of norms, whenever $s+2\theta\neq 1/p,1+1/p$.

The same result holds replacing $(\mathfrak{t},\sigma,\gamma,\mathbb{A},\mathbb{M})$ by $(\mathfrak{n},\gamma,\sigma,\mathbb{M},\mathbb{A})$.
\end{proposition}

\begin{proof}We only prove \eqref{eq:1stInterpolSolenoidal}, other equalities have  the same proof.

{\textbf{Step 1 :}}  We start with \cite[Proposition~3.33]{Gaudin2022} which yields the embedding
\begin{align*}
    (\dot{\mathrm{H}}^{s,p}_{\mathfrak{t},\sigma}(\mathbb{R}^n_+),\dot{\mathrm{D}}^{s}_{p}(\mathring{\mathbb{A}}_{\mathcal{H},\mathfrak{t}}))_{\theta,q} \hookrightarrow (\dot{\mathrm{H}}^{s,p}(\mathbb{R}^n_+),\dot{\mathrm{H}}^{s+2,p}(\mathbb{R}^n_+))_{\theta,q} = \dot{\mathrm{B}}^{s+2\theta}_{p,q}(\mathbb{R}^n_+) \text{. }
\end{align*}
Now, if $q<+\infty$, we recall that $\dot{\mathrm{H}}^{s,p}_{\mathfrak{t},\sigma}(\mathbb{R}^n_+)\cap\dot{\mathrm{D}}^{s}_{p}(\mathring{\mathbb{A}}_{\mathcal{H},\mathfrak{t}})$ is dense in $(\dot{\mathrm{H}}^{s,p}_{\mathfrak{t},\sigma}(\mathbb{R}^n_+),\dot{\mathrm{D}}^{s}_{p}(\mathring{\mathbb{A}}_{\mathcal{H},\mathfrak{t}}))_{\theta,q}$ by \cite[Theorem~3.4.2]{BerghLofstrom1976}, so that by continuity of traces,
\begin{align*}
    (\dot{\mathrm{H}}^{s,p}_{\mathfrak{t},\sigma}(\mathbb{R}^n_+),\dot{\mathrm{D}}^{s}_{p}(\mathring{\mathbb{A}}_{\mathcal{H},\mathfrak{t}}))_{\theta,q} \hookrightarrow \dot{\mathrm{B}}^{s+2\theta,\sigma}_{p,q,\mathcal{H}_{\mathfrak{t}}}(\mathbb{R}^n_+) \hookrightarrow \dot{\mathrm{B}}^{s+2\theta}_{p,q,\mathcal{H}_{\mathfrak{t}}}(\mathbb{R}^n_+) \text{. }
\end{align*}
Again density of $\dot{\mathrm{H}}^{s,p}_{\mathfrak{t},\sigma}(\mathbb{R}^n_+)\cap\dot{\mathrm{D}}^{s}_{p}(\mathring{\mathbb{A}}_{\mathcal{H},\mathfrak{t}})$ yields $\delta f = 0$ for all $f\in (\dot{\mathrm{H}}^{s,p}(\mathbb{R}^n_+),\dot{\mathrm{D}}^{s}_{p}(\mathring{\Delta}_\mathcal{J}))_{\theta,q}$.
The case $q=+\infty$ is left to the end of \textbf{Step 3}.

{\textbf{Step 2 :}} We want to extend the range of exponents for the boundedness of $\mathbb{P}$, and get a density result.

Let $f\in\dot{\mathrm{D}}^{s}_{p}({\Delta}_{\mathcal{H},\mathfrak{t}})\subset \dot{\mathrm{B}}^{s+2\theta}_{p,q,\mathcal{H}_{\mathfrak{t}}}(\mathbb{R}^n_+)$, we have $\mathbb{P}f\in\dot{\mathrm{D}}^{s}_{p}({\Delta}_{\mathcal{H},\mathfrak{t}})$ and by Proposition \ref{prop:InterpHomDomainLaplacians}, \cite[Proposition~2.12]{DanchinHieberMuchaTolk2020}, Corollary \ref{cor:commutrelationHodgeProjectorHspBspqRn+} and Theorem \ref{thm:HodgeDecompRn+}, we obtain successively
\begin{align*}
    \lVert \mathbb{P} f \rVert_{\dot{\mathrm{B}}^{s+2\theta}_{p,q}(\mathbb{R}^n_+)} &\lesssim_{p,s,n,\theta} \left( \int_{0}^{+\infty} \lVert t^{1-\theta} \Delta_{\mathcal{H},\mathfrak{t}}e^{t\Delta_{\mathcal{H},\mathfrak{t}}} \mathbb{P} f\rVert_{\dot{\mathrm{H}}^{s,p}(\mathbb{R}^n_+)}^q \frac{\mathrm{d}t}{t} \right)^\frac{1}{q}\\
    &\lesssim_{p,s,n,\theta} \left( \int_{0}^{+\infty} \lVert t^{1-\theta} \mathbb{P} \Delta_{\mathcal{H},\mathfrak{t}}e^{t\Delta_{\mathcal{H},\mathfrak{t}}}  f\rVert_{\dot{\mathrm{H}}^{s,p}(\mathbb{R}^n_+)}^q \frac{\mathrm{d}t}{t} \right)^\frac{1}{q}\\
    &\lesssim_{p,s,n,\theta} \left( \int_{0}^{+\infty} \lVert t^{1-\theta} \Delta_{\mathcal{H},\mathfrak{t}}e^{t\Delta_{\mathcal{H},\mathfrak{t}}}  f\rVert_{\dot{\mathrm{H}}^{s,p}(\mathbb{R}^n_+)}^q \frac{\mathrm{d}t}{t} \right)^\frac{1}{q}\\
    &\lesssim_{p,s,n,\theta} \lVert f \rVert_{\dot{\mathrm{B}}^{s+2\theta}_{p,q}(\mathbb{R}^n_+)}\text{. }
\end{align*}
From above estimates, if $q<+\infty$, by density of $\dot{\mathrm{D}}^{s}_{p}({\Delta}_{\mathcal{H},\mathfrak{t}})$ in $\dot{\mathrm{B}}^{s+2\theta}_{p,q,\mathcal{H}_{\mathfrak{t}}}(\mathbb{R}^n_+)$,  we have that
\begin{align*}
    \mathbb{P}\,:\,\dot{\mathrm{B}}^{s+2\theta}_{p,q,\mathcal{H}_{\mathfrak{t}}}(\mathbb{R}^n_+) \longrightarrow \dot{\mathrm{B}}^{s+2\theta}_{p,q,\mathcal{H}_{\mathfrak{t}}}(\mathbb{R}^n_+)
\end{align*}
extends uniquely as a bounded linear projection on $\dot{\mathrm{B}}^{s+2\theta}_{p,q,\mathcal{H}_{\mathfrak{t}}}(\mathbb{R}^n_+)$ with range $\dot{\mathrm{B}}^{s+2\theta,\sigma}_{p,q,\mathcal{H}_{\mathfrak{t}}}(\mathbb{R}^n_+)$. The result still holds for $q=+\infty$, by above \textbf{Step 1}, the reiteration theorem \cite[Theorem~3.5.3]{BerghLofstrom1976} and Proposition \ref{prop:InterpHomDomainLaplacians}.

In particular, $\dot{\mathrm{D}}^{s}_{p}({\mathbb{A}}_{\mathcal{H},\mathfrak{t}})=\mathbb{P}\dot{\mathrm{D}}^{s}_{p}({\Delta}_{\mathcal{H},\mathfrak{t}})$ is dense in $\dot{\mathrm{B}}^{s+2\theta,\sigma}_{p,q,\mathcal{H}_{\mathfrak{t}}}(\mathbb{R}^n_+)$, when $q<+\infty$.

{\textbf{Step 3 :}} For the reverse embedding. For $f\in\dot{\mathrm{D}}^{s}_{p}({\mathbb{A}}_{\mathcal{H},\mathfrak{t}})\subset \dot{\mathrm{B}}^{s+2\theta}_{p,q,\mathcal{H}_{\mathfrak{t}}}(\mathbb{R}^n_+)$ where we recall that $\dot{\mathrm{B}}^{s+2\theta}_{p,q,\mathcal{H}_{\mathfrak{t}}}(\mathbb{R}^n_+) =(\dot{\mathrm{H}}^{s,p}(\mathbb{R}^n_+), \dot{\mathrm{D}}^{s}_{p}(\mathring{\Delta}_{\mathcal{H},\mathfrak{t}}))_{\theta,q} \subset \dot{\mathrm{H}}^{s,p}(\mathbb{R}^n_+)+ \dot{\mathrm{D}}^{s}_{p}(\mathring{\Delta}_{\mathcal{H},\mathfrak{t}})$

If we let $(a,b)\in \dot{\mathrm{H}}^{s,p}(\mathbb{R}^n_+)\times \dot{\mathrm{D}}^{s}_{p}(\mathring{\Delta}_{\mathcal{H},\mathfrak{t}})$ such that $f=a+b$, by Proposition \ref{prop:assumptionsareCheckedforLap}, it is given that
\begin{align*}
    b=f-a \in (\dot{\mathrm{D}}^{s}_{p}({\Delta}_{\mathcal{H},\mathfrak{t}})+\dot{\mathrm{H}}^{s,p}(\mathbb{R}^n_+))\cap \dot{\mathrm{D}}^{s}_{p}(\mathring{\Delta}_{\mathcal{H},\mathfrak{t}} ) \subset \dot{\mathrm{D}}^{s}_{p}({\Delta}_{\mathcal{H},\mathfrak{t}} )
\end{align*}
and for the same reason $a\in \dot{\mathrm{D}}^{s}_{p}({\Delta}_{\mathcal{H},\mathfrak{t}})$. Therefore,
\begin{align*}
    f=\mathbb{P}f= \mathbb{P}a + \mathbb{P}b \in \dot{\mathrm{H}}^{s,p}_{\mathfrak{t},\sigma}(\mathbb{R}^n_+)+ \dot{\mathrm{D}}^{s}_{p}(\mathring{\mathbb{A}}_{\mathcal{H},\mathfrak{t}})\text{. }
\end{align*}
By \eqref{eq:defHodgeStokesAbsBC} and Corollary \ref{cor:commutrelationHodgeProjectorHspBspqRn+}, we have
\begin{align*}
    \lVert t^{1-\theta} \mathring{\mathbb{A}}_{\mathcal{H},\mathfrak{t}}e^{-t{\mathbb{A}}_{\mathcal{H},\mathfrak{t}}} f\rVert_{\dot{\mathrm{H}}^{s,p}(\mathbb{R}^n_+)} &\leqslant \lVert t^{1-\theta} \mathring{\mathbb{A}}_{\mathcal{H},\mathfrak{t}}e^{-t{\mathbb{A}}_{\mathcal{H},\mathfrak{t}}}  \mathbb{P}a\rVert_{\dot{\mathrm{H}}^{s,p}(\mathbb{R}^n_+)}+ \lVert t^{1-\theta} \mathring{\mathbb{A}}_{\mathcal{H},\mathfrak{t}}e^{-t{\mathbb{A}}_{\mathcal{H},\mathfrak{t}}}\mathbb{P}  b\rVert_{\dot{\mathrm{H}}^{s,p}(\mathbb{R}^n_+)}\\
    &\lesssim_{p,n,s}  \lVert t^{1-\theta}\mathbb{P} {\Delta}_{\mathcal{H},\mathfrak{t}}e^{t{\Delta}_{\mathcal{H},\mathfrak{t}}}  a\rVert_{\dot{\mathrm{H}}^{s,p}(\mathbb{R}^n_+)}+ \lVert t^{1-\theta} \mathbb{P} {\Delta}_{\mathcal{H},\mathfrak{t}}e^{t{\Delta}_{\mathcal{H},\mathfrak{t}}}  b\rVert_{\dot{\mathrm{H}}^{s,p}(\mathbb{R}^n_+)}\text{. }
\end{align*}
From there, we use analyticity of the semigroup, boundedness of $\mathbb{P}$ given by Theorem \ref{thm:HodgeDecompRn+}, to obtain
\begin{align*}
    \lVert t^{1-\theta} \mathring{\mathbb{A}}_{\mathcal{H},\mathfrak{t}}e^{-t{\mathbb{A}}_{\mathcal{H},\mathfrak{t}}} f\rVert_{\dot{\mathrm{H}}^{s,p}(\mathbb{R}^n_+)} &\lesssim_{p,n,s} t^{-\theta} \lVert  a\rVert_{\dot{\mathrm{H}}^{s,p}(\mathbb{R}^n_+)}+ t^{1-\theta}\lVert {\Delta}_{\mathcal{H},\mathfrak{t}}  b\rVert_{\dot{\mathrm{H}}^{s,p}(\mathbb{R}^n_+)}\text{. }
\end{align*}
Taking the infimum on all such pairs $(a,b)$ yields
\begin{align*}
    \lVert t^{1-\theta} \mathring{\mathbb{A}}_{\mathcal{H},\mathfrak{t}}e^{-t{\mathbb{A}}_{\mathcal{H},\mathfrak{t}}} f\rVert_{\dot{\mathrm{H}}^{s,p}(\mathbb{R}^n_+)} \lesssim_{p,n,s} t^{-\theta}K(t,f,\dot{\mathrm{H}}^{s,p}(\mathbb{R}^n_+), \dot{\mathrm{D}}^{s}_{p}(\mathring{\Delta}_{\mathcal{H},\mathfrak{t}}))\text{. }
\end{align*}
One may take the $\mathrm{L}^q_{\ast}$-norm of above inequality and apply \cite[Lemma~2.12]{DanchinHieberMuchaTolk2020} and above Proposition \ref{prop:InterpHomDomainLaplacians}, to deduce that
\begin{align*}
    \lVert f \rVert_{(\dot{\mathrm{H}}^{s,p}_{\mathfrak{t},\sigma}(\mathbb{R}^n_+),\dot{\mathrm{D}}^{s}_{p}(\mathring{\mathbb{A}}_{\mathcal{H},\mathfrak{t}}))_{\theta,q}}&\lesssim_{\theta}\left( \int_{0}^{+\infty} \lVert t^{1-\theta} \Delta_{\mathcal{H},\mathfrak{t}}e^{t\Delta_{\mathcal{H},\mathfrak{t}}} \mathbb{P} f\rVert_{\dot{\mathrm{H}}^{s,p}(\mathbb{R}^n_+)}^q \frac{\mathrm{d}t}{t} \right)\\
    &\lesssim_{p,n,s,\theta} \lVert f \rVert_{(\dot{\mathrm{H}}^{s,p}(\mathbb{R}^n_+),\dot{\mathrm{D}}^{s}_{p}(\mathring{\Delta}_{\mathcal{H},\mathfrak{t}}))_{\theta,q}}\\
    &\lesssim_{p,n,s,\theta} \lVert f \rVert_{\dot{\mathrm{B}}^{s+2\theta}_{p,q}(\mathbb{R}^n_+)}\text{. }
\end{align*}
With \textbf{Step 1}, one has for all $f\in \dot{\mathrm{D}}^{s}_{p}({\mathbb{A}}_{\mathcal{H},\mathfrak{t}})$,
\begin{align*}
    \lVert f \rVert_{(\dot{\mathrm{H}}^{s,p}_{\mathfrak{t},\sigma}(\mathbb{R}^n_+),\dot{\mathrm{D}}^{s}_{p}(\mathring{\mathbb{A}}_{\mathcal{H},\mathfrak{t}}))_{\theta,q}}\sim_{p,n,s,\theta} \lVert f \rVert_{\dot{\mathrm{B}}^{s+2\theta}_{p,q}(\mathbb{R}^n_+)}\text{. }
\end{align*}
If $q<+\infty$, then the end of above \textbf{Step 2}, and \cite[Lemma~2.10]{DanchinHieberMuchaTolk2020} allows to conclude by density. If $q=+\infty$, the result follows from the reiteration theorem \cite[Theorem~3.5.3]{BerghLofstrom1976} and the boundedness and the range of $\mathbb{P}$ in \textbf{Step 2} (use a retraction argument \cite[Theorem~6.4.2]{BerghLofstrom1976}).
\end{proof}

The \textbf{Step 2} from the proof above leads to the immediate following corollary.

\begin{corollary}\label{cor:extendedHodgeDecompComptabilityConditions}Let $p\in(1,+\infty)$, $q\in[1,+\infty]$, $s\in(-1+1/p,2+1/p)$, such that $s\notin \mathbb{N}+\frac{1}{p}$, and \eqref{AssumptionCompletenessExponents} is satisfied. Then,
\begin{align*}
     \mathbb{P}&\,:\,\dot{\mathrm{B}}^{s}_{p,q,\mathcal{H}_{\mathfrak{t}}}(\mathbb{R}^n_+) \longrightarrow \dot{\mathrm{B}}^{s,\sigma}_{p,q,\mathcal{H}_{\mathfrak{t}}}(\mathbb{R}^n_+)\text{, }\\
     [\mathrm{I}-\mathbb{P}]&\,:\,\dot{\mathrm{B}}^{s}_{p,q,\mathcal{H}_{\mathfrak{t}}}(\mathbb{R}^n_+) \longrightarrow \dot{\mathrm{B}}^{s,\gamma}_{p,q,\mathcal{H}_{\mathfrak{t}}}(\mathbb{R}^n_+)\text{, }
\end{align*}
are both well defined bounded linear projection, so that the following Hodge decomposition holds
\begin{align*}
    \dot{\mathrm{B}}^{s}_{p,q,\mathcal{H}_{\mathfrak{t}}}(\mathbb{R}^n_+) = \dot{\mathrm{B}}^{s,\sigma}_{p,q,\mathcal{H}_{\mathfrak{t}}}(\mathbb{R}^n_+) \oplus \dot{\mathrm{B}}^{s,\gamma}_{p,q,\mathcal{H}_{\mathfrak{t}}}(\mathbb{R}^n_+) \text{. }
\end{align*}
The result still holds if we replace $(\mathfrak{t},\mathbb{P})$ by $(\mathfrak{n},\mathbb{Q})$.
\end{corollary}

Finally, we mention without its proofs, that follows exactly the same lines, the result for interpolation spaces of the homogeneous domains with Besov spaces as an ambient function space, say, for $\theta\in(0,1)$, $r,q\in[1,+\infty]$, $p\in(1,+\infty)$, $-1+1/p<s<1/p$,
\begin{align*}
    \mathring{\eus{D}}^{s,p,r}_{-\Delta_{\mathcal{H}}}(\theta,q) = (\dot{\mathrm{B}}^{s}_{p,r}(\mathbb{R}^n_+),\dot{\mathrm{D}}^{s}_{p,r}(\mathring{\Delta}_{\mathcal{H}}))_{\theta,q} \text{. }
\end{align*}
We are able to obtain,

\begin{proposition}\label{prop:InterpHomDomainBesovLaplacians} Let $p\in(1,+\infty)$, $r\in[1,+\infty)$ $q\in[1,+\infty]$, and $s\in(-1+1/p,1/p)$.For all $\theta\in(0,1)$ such that $(\mathcal{C}_{s+2\theta,p,q})$ is satisfied, provided $\mathcal{J}\in\{\mathcal{D},\mathcal{N},\mathcal{H}\}$, one has
\begin{align*}
    (\dot{\mathrm{B}}^{s}_{p,r}(\mathbb{R}^n_+),\dot{\mathrm{D}}^{s}_{p,r}(\mathring{\Delta}_\mathcal{J}))_{\theta,q} = \dot{\mathrm{B}}^{s+2\theta}_{p,q,\mathcal{J}}(\mathbb{R}^n_+)\text{, }
\end{align*}
with equivalence of norms, whenever
\begin{itemize}
    \item $s+2\theta\neq 1/p$ if $\mathcal{J}=\mathcal{D}$,
    \item $s+2\theta\neq 1+1/p$ if $\mathcal{J}=\mathcal{N}$,
    \item $s+2\theta\neq 1/p,1+1/p$ if $\mathcal{J}=\mathcal{H}$.
\end{itemize}
\end{proposition}

Then by the mean of Corollary \ref{cor:extendedHodgeDecompComptabilityConditions},

\begin{proposition}\label{prop:InterpHomDomainBesovHodgeMaxwellStokes} Let $p\in(1,+\infty)$,$r\in[1,+\infty)$, $q\in[1,+\infty]$, and $s\in(-1+1/p,1/p)$. For all $\theta\in(0,1)$ such that $(\mathcal{C}_{s+2\theta,p,q})$ is satisfied, one has
\begin{align*}
    (\dot{\mathrm{B}}^{s,\sigma}_{p,r,\mathfrak{t}}(\mathbb{R}^n_+),\dot{\mathrm{D}}^{s}_{p,r}(\mathring{\mathbb{A}}_{\mathcal{H},\mathfrak{t}}))_{\theta,q} &= \dot{\mathrm{B}}^{s+2\theta,\sigma}_{p,q,\mathcal{H}_{\mathfrak{t}}}(\mathbb{R}^n_+)\text{, }\\
   (\dot{\mathrm{B}}^{s,\gamma}_{p,r}(\mathbb{R}^n_+),\dot{\mathrm{D}}^{s}_{p,r}(\mathring{\mathbb{M}}_{\mathcal{H},\mathfrak{t}}))_{\theta,q} &= \dot{\mathrm{B}}^{s+2\theta,\gamma}_{p,q,\mathcal{H}_{\mathfrak{t}}}(\mathbb{R}^n_+)\text{, }
\end{align*}
with equivalence of norms, whenever $s+2\theta\neq 1/p,1+1/p$.

The same result holds replacing $(\mathfrak{t},\sigma,\gamma,\mathbb{A},\mathbb{M})$ by $(\mathfrak{n},\gamma,\sigma,\mathbb{M},\mathbb{A})$.
\end{proposition}

\subsection{Maximal regularities for Hodge Laplacians and related operators}\label{sec:maxRegMaxHodgeStokes}

We will present here direct application of Theorem \ref{thm:LqMaxRegUMDHomogeneous}, and \cite[Theorem~2.20]{DanchinHieberMuchaTolk2020} with appropriate identification of real interpolation spaces, provided $p,r\in(1,+\infty)$, $s\in(-1+1/p,1/p)$,
\begin{align*}
    \mathring{\eus{D}}^{s,p}_{A}(\theta,q) \text{, }\mathring{\eus{D}}^{s,p,r}_{A}(\theta,q) \text{, } \theta\in(0,1) \text{, } q\in[1,+\infty] \text{ and } A\in\{-\Delta_{\mathcal{H}},\mathbb{A}_\mathcal{H},\mathbb{M}_\mathcal{H}\} 
\end{align*}
subject to either tangential or normal boundary conditions, see Propositions \ref{prop:InterpHomDomainLaplacians}, \ref{prop:InterpHomDomainHodgeMaxwellStokes}, \ref{prop:InterpHomDomainBesovLaplacians} and \ref{prop:InterpHomDomainBesovHodgeMaxwellStokes}.

We recall that the definition of involved spaces are given in Notations \ref{def:newNotationsSoleinodalSpaces}, see also \eqref{eq:defhomBesovwithCompCond} and \eqref{eq:SolenoidalBesovspaceswithcompcond}. To alleviate notations in inequalities, we drop the references to the open set $\mathbb{R}^n_+$.

We give first two Theorems in the case where the ambiant space is an UMD Banach space which is the case of $\dot{\mathrm{H}}^{s,p}$ and $\dot{\mathrm{B}}^{s}_{p,r}$, provided $p,r\in(1,+\infty)$, $s\in(-1+1/p,2+1/p)$. 

\begin{theorem}\label{thm:MaxRxBesovUMDHodgeStokesRn+}Let $p,q,r\in(1,+\infty)$, and for $\alpha\in(-1+1/q,1/q)$ fixed, we set $\alpha_q:= 1+\alpha-1/q$.

Let $s\in(-1+1/p,2+1/p)$ such that $s,s+2\alpha_q\notin \mathbb{N}+\frac{1}{p}$, $(\mathcal{C}_{s+2\alpha_q,p,q})$ is satisfied, and let $T\in(0,+\infty]$.

For any $f\in \dot{\mathrm{H}}^{\alpha,q}((0,T),\dot{\mathrm{B}}^{s}_{p,r,\mathcal{H}_\mathfrak{t}}(\mathbb{R}^n_+,\Lambda))$, $u_0\in \dot{\mathrm{B}}^{s+2\alpha_q}_{p,q,\mathcal{H}_\mathfrak{t}}(\mathbb{R}^n_+,\Lambda)$, there exists a unique mild solution $u \in \mathrm{C}^0_b([0,T],\dot{\mathrm{B}}^{s+2\alpha_q}_{p,q,\mathcal{H}_\mathfrak{t}}(\mathbb{R}^n_+,\Lambda))$ of
\begin{equation*}\tag{HHS${}_{\mathfrak{t}}$}\label{eq:HodgeHeatSystem}
    \left\{ \begin{array}{rclr}
         \partial_t u - \Delta u  & = &f,& \text{ on } (0,T)\times\mathbb{R}^n_+,\\
         \nu\iprod\mathrm{d} u_{|_{\partial\mathbb{R}^n_+}}& = & 0,& \text{ on } (0,T)\times\partial\mathbb{R}^n_+,\\
         \nu\iprod u_{|_{\partial\mathbb{R}^n_+}}& = & 0,& \text{ on } (0,T)\times\partial\mathbb{R}^n_+,\\
         u(0)& = & u_0,& \text{ in } \dot{\mathrm{B}}^{2+s-\frac{2}{q}}_{p,q}(\mathbb{R}^n_+,\Lambda),
    \end{array}
    \right.
\end{equation*}
with estimate
\begin{align*}
    \left\lVert u \right\rVert_{\mathrm{L}^\infty((0,T),\dot{\mathrm{B}}^{s+2\alpha_q}_{p,q})}  \lesssim_{p,q,n}^{\alpha,s} \left\lVert (\partial_t u, \nabla^2 u) \right\rVert_{\dot{\mathrm{H}}^{\alpha,q}((0,T),\dot{\mathrm{B}}^{s}_{p,r})} \lesssim_{p,q,n}^{\alpha,s} \left\lVert f \right\rVert_{\dot{\mathrm{H}}^{\alpha,q}((0,T),\dot{\mathrm{B}}^{s}_{p,r})} + \left\lVert u_0 \right\rVert_{\dot{\mathrm{B}}^{s+2\alpha_q}_{p,q}}.
\end{align*}
For all $\beta\in[0,1]$, we also have
\begin{align}\label{eq:fractionalEstimatePruss1}
    \left\lVert (-\partial_t)^\beta(-\Delta_\mathcal{H,\mathfrak{t}})^{1-\beta} u \right\rVert_{\dot{\mathrm{H}}^{\alpha,q}((0,T),\dot{\mathrm{B}}^{s}_{p,r})} \lesssim_{p,q,n}^{s,\alpha} \left\lVert f \right\rVert_{\dot{\mathrm{H}}^{\alpha,q} ((0,T),\dot{\mathrm{B}}^{s}_{p,r})} + \left\lVert u_0 \right\rVert_{\dot{\mathrm{B}}^{s+2\alpha_q}_{p,q}}\text{. }
\end{align}
\end{theorem}

\begin{proof} From Theorem \ref{thm:MetaThm1HodgeLaplacianRn+} we have the bounded holomorphic calculus of $-\Delta_\mathcal{H,\mathfrak{t}}$ on $\dot{\mathrm{B}}^{s}_{p,r,\mathcal{H}_{\mathfrak{t}}}(\mathbb{R}^n_+)$, so that we may apply Theorem \ref{thm:LqMaxRegUMDHomogeneous} to obtain maximal regularity estimates, whereas Proposition \ref{prop:InterpHomDomainLaplacians} gives an exact description of interpolation spaces.
\end{proof}

\begin{theorem}\label{thm:MaxRxHspUMDHodgeStokesRn+}Let $p,q\in(1,+\infty)$, $s\in(-1+1/p,1/p)$, and $\alpha\in(-1+1/q,1/q)$ fixed, we set $\alpha_q:= 1+\alpha-1/q$. We assume that $s+2\alpha_q\notin \mathbb{N}+\frac{1}{p}$, $(\mathcal{C}_{s+2\alpha_q,p,q})$ is satisfied, and let $T\in(0,+\infty]$.

For any $f\in \dot{\mathrm{H}}^{\alpha,q}((0,T),\dot{\mathrm{H}}^{s,p}_{\mathfrak{n},\gamma}(\mathbb{R}^n_+,\Lambda))$, $u_0\in \dot{\mathrm{B}}^{s+2\alpha_q,\gamma}_{p,q,\mathcal{H}_\mathfrak{n}}(\mathbb{R}^n_+,\Lambda)$, there exists a unique mild solution $u \in \mathrm{C}^0_b([0,T],\dot{\mathrm{B}}^{s+2\alpha_q,\gamma}_{p,q,\mathcal{H}_\mathfrak{n}}(\mathbb{R}^n_+,\Lambda))$ of
\begin{equation*}\tag{HMS${}_{\mathfrak{n}}$}\label{eq:HodgeMaxwellBCSystem}
    \left\{ \begin{array}{rclr}
         \partial_t u - \Delta u  & = &f,& \text{ on } (0,T)\times\mathbb{R}^n_+,\\
         \mathrm{d} u& = & 0,& \text{ on } (0,T)\times\mathbb{R}^n_+,\\
         \nu\wedge\delta u_{|_{\partial\mathbb{R}^n_+}}& = & 0,& \text{ on } (0,T)\times\partial\mathbb{R}^n_+,\\
         \nu\wedge u_{|_{\partial\mathbb{R}^n_+}}& = & 0,& \text{ on } (0,T)\times\partial\mathbb{R}^n_+,\\
         u(0)& = & u_0,& \text{ in } \dot{\mathrm{B}}^{2+s-\frac{2}{q}}_{p,q}(\mathbb{R}^n_+,\Lambda),
    \end{array}
    \right.
\end{equation*}
with estimate
\begin{align*}
    \left\lVert u \right\rVert_{\mathrm{L}^\infty((0,T),\dot{\mathrm{B}}^{s+2\alpha_q}_{p,q})} \lesssim_{p,q,n}^{s,\alpha} \left\lVert (\partial_t u, \nabla^2 u) \right\rVert_{\dot{\mathrm{H}}^{\alpha,q}((0,T),\dot{\mathrm{H}}^{s,p})} \lesssim_{p,q,n}^{s,\alpha} \left\lVert f \right\rVert_{\dot{\mathrm{H}}^{\alpha,q}((0,T),\dot{\mathrm{H}}^{s,p})} + \left\lVert u_0 \right\rVert_{\dot{\mathrm{B}}^{s+2\alpha_q}_{p,q}}.
\end{align*}
For all $\beta\in[0,1]$, we also have
\begin{align}\label{eq:fractionalEstimatePruss2}
    \left\lVert (-\partial_t)^{\beta}(\mathbb{M}_\mathcal{H,\mathfrak{n}})^{1-\beta} u \right\rVert_{\dot{\mathrm{H}}^{\alpha,q}((0,T),\dot{\mathrm{H}}^{s,p})} \lesssim_{p,q,n}^{s,\alpha} \left\lVert f \right\rVert_{\dot{\mathrm{H}}^{\alpha,q}((0,T),\dot{\mathrm{H}}^{s,p})} + \left\lVert u_0 \right\rVert_{\dot{\mathrm{B}}^{s+2\alpha_q}_{p,q}}\text{. }
\end{align}
\end{theorem}

\begin{proof} From Theorem \ref{thm:HinftyFuncCalcHodgeStokesMaxwell} we have the bounded holomorphic calculus of $\mathbb{M}_\mathcal{H,\mathfrak{n}}$ on $\dot{\mathrm{H}}^{s,p}(\mathbb{R}^n_+)$, so that we may apply Theorem \ref{thm:LqMaxRegUMDHomogeneous} to obtain maximal regularity estimates, whereas Proposition \ref{prop:InterpHomDomainBesovLaplacians} gives an exact description of interpolation spaces.
\end{proof}

Finally, it remains to apply the homogeneous Da Prato-Grisvard Theorem \cite[Theorem~2.20]{DanchinHieberMuchaTolk2020} to state our last $\mathrm{L}^q$-maximal regularity theorem.

\begin{theorem}\label{thm:MaxRxBesovHodgeStokesRn+}Let $p\in(1,+\infty)$, $q\in[1,+\infty)$, $s\in(-1+1/p,1/p+2/q)$, such that $s,s+2-2/q\notin \mathbb{N}+\frac{1}{p}$ and $(\mathcal{C}_{s+2-2/q,p,q})$ is satisfied and let $T\in(0,+\infty]$.

For any $f\in \mathrm{L}^q((0,T),\dot{\mathrm{B}}^{s,\sigma}_{p,q,\mathcal{\mathcal{H}_{\mathfrak{t}}}}(\mathbb{R}^n_+,\Lambda))$, $u_0\in \dot{\mathrm{B}}^{2+s-\frac{2}{q},\sigma}_{p,q,\mathcal{H}_\mathfrak{t}}(\mathbb{R}^n_+,\Lambda)$, there exists a unique mild solution $u \in \mathrm{C}^0_b([0,T],\dot{\mathrm{B}}^{2+s-\frac{2}{q},\sigma}_{p,q,\mathcal{H}_\mathfrak{t}}(\mathbb{R}^n_+,\Lambda))$ of
\begin{equation*}\tag{HSS${}_{\mathfrak{t}}$}\label{eq:HodgeStokesAbsoluteBCSystem}
    \left\{ \begin{array}{rclr}
         \partial_t u - \Delta u  & = &f,& \text{ on } (0,T)\times\mathbb{R}^n_+,\\
         \delta u& = & 0,& \text{ on } (0,T)\times\mathbb{R}^n_+,\\
         \nu\iprod\mathrm{d} u_{|_{\partial\mathbb{R}^n_+}}& = & 0,& \text{ on } (0,T)\times\partial\mathbb{R}^n_+,\\
         \nu\iprod u_{|_{\partial\mathbb{R}^n_+}}& = & 0,& \text{ on } (0,T)\times\partial\mathbb{R}^n_+,\\
         u(0)& = & u_0,& \text{ in } \dot{\mathrm{B}}^{2+s-\frac{2}{q}}_{p,q}(\mathbb{R}^n_+,\Lambda),
    \end{array}
    \right.
\end{equation*}
with estimate
\begin{align*}
    \left\lVert u \right\rVert_{\mathrm{L}^\infty((0,T),\dot{\mathrm{B}}^{2+s-\frac{2}{q}}_{p,q})} + \left\lVert (\partial_t u, \nabla^2 u) \right\rVert_{\mathrm{L}^q((0,T),\dot{\mathrm{B}}^{s}_{p,q})} \lesssim_{p,q,n}^{s} \left\lVert f \right\rVert_{\mathrm{L}^q((0,T),\dot{\mathrm{B}}^{s}_{p,q})} + \left\lVert u_0 \right\rVert_{\dot{\mathrm{B}}^{2+s-\frac{2}{q}}_{p,q}}.
\end{align*}
In the case $q=+\infty$, if we assume in addition $u_0\in \dot{\mathrm{D}}^{s}_{p}(\mathbb{A}_{\mathcal{H},\mathfrak{t}}^2)$, we have
\begin{align*}
     \left\lVert (\partial_t u, \nabla^2 u) \right\rVert_{\mathrm{L}^\infty((0,T),\dot{\mathrm{B}}^{s}_{p,\infty})} \lesssim_{p,s,n} \left\lVert f \right\rVert_{\mathrm{L}^\infty((0,T),\dot{\mathrm{B}}^{s}_{p,\infty})} + \left\lVert \mathbb{A}_{\mathcal{H},\mathfrak{t}} u_0 \right\rVert_{\dot{\mathrm{B}}^{s}_{p,\infty}}.
\end{align*}
\end{theorem}

\begin{remark}Notice that above Theorem \ref{thm:MaxRxBesovHodgeStokesRn+} is the only one presented here that allows $\mathrm{L}^1$ and $\mathrm{L}^\infty$ in time maximal regularity estimates.\\
In particular, one should notice that in the case $q=1$, that above solution $u$ satisfies for almost every $t\in\mathbb{R}_+$,
\begin{align*}
    u(t),\,\partial_t u(t),\, \nabla^2 u(t) \in  \dot{\mathrm{B}}^{s}_{p,1}(\mathbb{R}^n_+)\text{. }
\end{align*}
\end{remark}

\begin{proof} We may apply Theorem \ref{thm:DaPratoGrisvardHom2020} to obtain maximal regularity estimates, since Proposition \ref{prop:InterpHomDomainHodgeMaxwellStokes} gives an exact description of interpolation spaces.
\end{proof}

\begin{remark}One may perform a cyclic permutation of systems \eqref{eq:HodgeHeatSystem}, \eqref{eq:HodgeMaxwellBCSystem} and \eqref{eq:HodgeStokesAbsoluteBCSystem}, but also exchange $\mathfrak{t}$ and $\mathfrak{n}$, up to appropriate modification on boundary conditions and considered function spaces, to obtain each type of results for each operator
    \begin{align*}
        \{-\Delta_{\mathcal{H},\mathfrak{t}},\mathbb{A}_{\mathcal{H},\mathfrak{t}},\mathbb{M}_{\mathcal{H},\mathfrak{t}},-\Delta_{\mathcal{H},\mathfrak{n}},\mathbb{A}_{\mathcal{H},\mathfrak{n}},\mathbb{M}_{\mathcal{H},\mathfrak{n}}\} \text{. }
    \end{align*}
\end{remark}

\subsection{Maximal regularity for the Stokes system with Navier-slip boundary conditions}\label{sec:NaviernoSlipBCRn+}

The flatness of $\partial\mathbb{R}^n_+$ has the interesting consequence that, for $u\,:\,\mathbb{R}^n_+\longrightarrow\mathbb{C}^n\simeq \Lambda^1$ regular enough, the tangential Hodge boundary conditions
\begin{equation*}
    \left\{ \begin{array}{rcl}
         \nu \iprod u_{|_{\partial\mathbb{R}^n_+}}& = & 0,\\
         \nu\iprod \mathrm{d} u_{|_{\partial\mathbb{R}^n_+}}& = & 0.
    \end{array}
    \right.
\end{equation*}
are equivalent to the Navier-slip boundary conditions
\begin{equation}\label{eq:NavierSlipBCs}
    \left\{ \begin{array}{rcl}
         \nu \cdot u_{|_{\partial\mathbb{R}^n_+}}& = & 0,\\
         {{[} (\prescript{t}{}{{\nabla u}} + \nabla u)\nu {]}_{\mathrm{tan}}}_{|_{\partial\mathbb{R}^n_+}} & = & 0.
    \end{array}
    \right.
\end{equation}

Indeed, recalling that $\nu = -\mathfrak{e}_n$, since one has
\begin{align*}
    {{[} (\prescript{t}{}{{\nabla u}} + \nabla u)\nu {]}_{\mathrm{tan}}} &=  (\prescript{t}{}{{\nabla u}} + \nabla u)(-\mathfrak{e}_n) - [(\prescript{t}{}{{\nabla u}} + \nabla u)(-\mathfrak{e}_n)\cdot (-\mathfrak{e}_n)](-\mathfrak{e}_n)\\  &= -\sum_{k=1}^{n-1} (\partial_{x_k}u_n + \partial_{x_n}u_k)\mathfrak{e}_k\text{. }
\end{align*}
We may use $ -\mathfrak{e}_n \cdot u_{|_{\partial\mathbb{R}^n_+}} = u_n(\cdot,0) = 0$, yielding for all $k\in\llb 1,n-1\rrb$
\begin{align*}
    0=\sum_{k=1}^{n-1} (\partial_{x_k}u_n(\cdot,0) + \partial_{x_n}u_k(\cdot,0))\mathfrak{e}_k = \sum_{k=1}^{n-1}  \partial_{x_n}u_k(\cdot,0)\mathfrak{e}_k \text{.}
\end{align*}
This implies that $u$ satisfies the exact same $n-1$ Neumann boundary conditions, and a single Dirichlet boundary condition on $u_n$. This stands exactly as in Lemma \ref{lem:Hs+2pestimateHodgeLap}. The converse also holds.

As long as one has enough regularity on $u$ at least in the Sobolev / Besov sense on $\mathbb{R}^n_+$, one is still able to perform the same decoupling for the boundary conditions.

This occurs when $s\in(-1+1/p,1+1/p)$, $p\in(1,+\infty)$ for the spaces $\mathrm{H}^{s+2,p}$, $\dot{\mathrm{H}}^{s,p}\cap\dot{\mathrm{H}}^{s+2,p}$, $\dot{\mathrm{H}}^{s+2,p}$ when those are complete. It still occurs when we replace ${\mathrm{H}}^{\cdot,p}$ by ${\mathrm{B}}^{\cdot}_{p,q}$, $q\in[1,+\infty]$.

Therefore, in each previous definitions restricted to $\Lambda^1\simeq\mathbb{C}^n$, such as e.g. \eqref{eq:defhomBesovwithCompCond}, one may replace the boundary condition $\nu\iprod \mathrm{d} u_{|_{\partial\mathbb{R}^n_+}} =0$ by ${{[} (\prescript{t}{}{{\nabla u}} + \nabla u)\nu {]}_{\mathrm{tan}}}_{|_{\partial\mathbb{R}^n_+}}=0$.

However, we mention the fact that, as exhibited in \cite[Section~2]{MitreaMonniaux2009-1}, such identification is no longer true for (even smooth) domains $\Omega$ with non-flat boundary. In this case, the equivalence holds up to correction term $\mathcal{W}u$, \textit{i.e.} \eqref{eq:NavierSlipBCs} is equivalent to
\begin{equation*}
    \left\{ \begin{array}{rcl}
        \nu \cdot u_{|_{\partial\Omega}}& = & 0,\\
        \nu \iprod \mathrm{d}u + \mathcal{W}u_{|_{\partial\Omega}} & = & 0.
    \end{array}
    \right.
\end{equation*}
Here, $\mathcal{W}$ is the \textbf{Weingarten map}. It is linear in $u$ and its coefficients depends linearly on the first derivatives of the outward unit normal $\nu$, requiring then some smoothness on the boundary $\partial\Omega$. With the flat  boundary $\partial\mathbb{R}^n_+$, the outward normal $\nu=-\mathfrak{e}_n$ is constant, which explains why the map $\mathcal{W}$ vanishes.

We can then exhibit the following maximal regularity result, where we identify $\Lambda^1$ with $\mathbb{C}^n$. Similar results such as Theorems \ref{thm:MaxRxBesovUMDHodgeStokesRn+} and \ref{thm:MaxRxHspUMDHodgeStokesRn+} are also available.

\begin{theorem}\label{thm:MaxRxBesovNavierSlipStokesRn+}Let $p\in(1,+\infty)$, $q\in[1,+\infty)$, $s\in(-1+1/p,1/p+2/q)$, such that $s,s+2-2/q\notin \mathbb{N}+\frac{1}{p}$ and $(\mathcal{C}_{s+2-2/q,p,q})$ is satisfied and let $T\in(0,+\infty]$.

For any $f\in \mathrm{L}^q((0,T),\dot{\mathrm{B}}^{s}_{p,q,\mathcal{\mathcal{H}_{\mathfrak{t}}}}(\mathbb{R}^n_+,\mathbb{C}^n))$, $u_0\in \dot{\mathrm{B}}^{2+s-\frac{2}{q},\sigma}_{p,q,\mathcal{H}_\mathfrak{t}}(\mathbb{R}^n_+,\mathbb{C}^n)$, there exists a unique mild solution $(u,\nabla \mathfrak{p}) \in \mathrm{C}^0_b([0,T],\dot{\mathrm{B}}^{2+s-\frac{2}{q},\sigma}_{p,q,\mathcal{H}_\mathfrak{t}}(\mathbb{R}^n_+,\mathbb{C}^n))\times \mathrm{L}^q((0,T),\dot{\mathrm{B}}^{s}_{p,q}(\mathbb{R}^n_+,\mathbb{C}^n))$ of
\begin{equation*}\tag{NSS}\label{eq:StokesNavierSlipBCSystem}
    \left\{ \begin{array}{rclr}
         \partial_t u - \Delta u+\nabla \mathfrak{p}  & = &f,& \text{ on } (0,T)\times\mathbb{R}^n_+,\\
         \div u& = & 0,& \text{ on } (0,T)\times\mathbb{R}^n_+,\\
         {{[} (\prescript{t}{}{{\nabla u}} + \nabla u)\nu {]}_{\mathrm{tan}}}_{|_{\partial\mathbb{R}^n_+}}& = & 0,& \text{ on } (0,T)\times\partial\mathbb{R}^n_+,\\
         \nu \cdot u_{|_{\partial\mathbb{R}^n_+}}& = & 0,& \text{ on } (0,T)\times\partial\mathbb{R}^n_+,\\
         u(0)& = & u_0,& \text{ in } \dot{\mathrm{B}}^{2+s-\frac{2}{q}}_{p,q}(\mathbb{R}^n_+,\mathbb{C}^n),
    \end{array}
    \right.
\end{equation*}
with estimate
\begin{align*}
    \left\lVert u \right\rVert_{\mathrm{L}^\infty((0,T),\dot{\mathrm{B}}^{2+s-\frac{2}{q}}_{p,q})} + \left\lVert (\partial_t u, \nabla^2 u,\nabla \mathfrak{p}) \right\rVert_{\mathrm{L}^q((0,T),\dot{\mathrm{B}}^{s}_{p,q})} \lesssim_{p,q,n}^{s} \left\lVert f \right\rVert_{\mathrm{L}^q((0,T),\dot{\mathrm{B}}^{s}_{p,q})} + \left\lVert u_0 \right\rVert_{\dot{\mathrm{B}}^{2+s-\frac{2}{q}}_{p,q}}.
\end{align*}
In the case $q=+\infty$, if we assume in addition $u_0\in \dot{\mathrm{D}}^{s}_{p}(\mathbb{A}_{\mathcal{H},\mathfrak{t}}^2)$, we have
\begin{align*}
     \left\lVert (\partial_t u, \nabla^2 u,\nabla \mathfrak{p}) \right\rVert_{\mathrm{L}^\infty((0,T),\dot{\mathrm{B}}^{s}_{p,\infty})} \lesssim_{p,s,n} \left\lVert f \right\rVert_{\mathrm{L}^\infty((0,T),\dot{\mathrm{B}}^{s}_{p,\infty})} + \left\lVert \mathbb{A}_{\mathcal{H},\mathfrak{t}} u_0 \right\rVert_{\dot{\mathrm{B}}^{s}_{p,\infty}}.
\end{align*}
\end{theorem}


\appendix

\section{Appendix: Partial traces of differential forms}\label{Append:TracesofFunctions}

We state here a trace theorem for generalized normal and tangential traces of differential forms.  The general case for vector fields in inhomogeneous function spaces is well-known, also is the differential form in the inhomogeneous case on bounded Lipschitz domains on Riemannian manifolds, see \cite[Section~4]{MitreaMitreaShaw2008} and references therein.

We recall that $\nu=-\mathfrak{e_n}$ is the outer normal unit at boundary $\partial\mathbb{R}^n_+$.

\begin{theorem}\label{thm:TracesDifferentialformsInhomSpaces}Let $p\in(1,+\infty)$, $q\in[1,+\infty]$, $s\in(-1+\frac{1}{p},\frac{1}{p})$ and let $k\in\llb 0,n\rrb$.
\begin{enumerate}[label=($\roman*$)]
    \item For all $u\in{\mathrm{D}}^{s}_{p}( \mathbf{\delta},\mathbb{R}^{n}_+,\Lambda^k)$. Then there exists a unique function $\nu\iprod u_{|_{\partial\mathbb{R}^n_+}}\in \mathrm{B}^{s-\frac{1}{p}}_{p,p}(\mathbb{R}^{n-1},\Lambda^{k-1})$ called the generalized \textbf{tangential trace}, such that
    \begin{align}
        \int_{\mathbb{R}^{n-1}}\langle \nu\iprod u_{|_{\partial\mathbb{R}^n_+}}, {\Psi}_{|_{\partial\mathbb{R}^n_+}}   \rangle &= \int_{\mathbb{R}^{n}_+} \langle u, \mathrm{d} {\Psi}\rangle - \int_{\mathbb{R}^{n}_+} \langle \mathbf{\delta} u, {\Psi}\rangle  \label{eq:IntegbyPartTangTrace}
    \end{align}
    for all $\Psi\in \mathrm{H}^{1-s,p'}(\mathbb{R}^n_+, \Lambda)$, with estimates
    \begin{align*}
        \lVert\nu\iprod u_{|_{\partial\mathbb{R}^n_+}} \rVert_{\mathrm{B}^{s-\frac{1}{p}}_{p,p}(\mathbb{R}^{n-1})}\lesssim_{p,s,n} \lVert u \rVert_{{\mathrm{H}}^{s,p}(\mathbb{R}^n_+)} + \lVert \mathbf{\delta} u \rVert_{{\mathrm{H}}^{s,p}(\mathbb{R}^n_+)} \text{.}
    \end{align*}
    The same result holds with corresponding estimate, for $u\in {\mathrm{D}}^{s}_{p}( \mathrm{d},\mathbb{R}^{n}_+,\Lambda^k)$ we have a partial trace $\nu\wedge u_{|_{\partial\mathbb{R}^n_+}}\in \mathrm{B}^{s-\frac{1}{p}}_{p,p}(\mathbb{R}^{n-1},\Lambda^{k+1})$ called the generalized \textbf{normal trace}, satisfying the identity
    \begin{align}
        \int_{\mathbb{R}^{n-1}}\langle \nu\wedge u_{|_{\partial\mathbb{R}^n_+}}, {\Psi}_{|_{\partial\mathbb{R}^n_+}}   \rangle &= \int_{\mathbb{R}^{n}_+} \langle \mathrm{d}u, {\Psi}\rangle - \int_{\mathbb{R}^{n}_+} \langle u, \mathbf{\delta} {\Psi}\rangle \text{.} \label{eq:IntegbyPartNormTrace}
    \end{align}
    \item For all $u\in{\mathrm{D}}^{s}_{p,q}(\mathbf{\delta},\mathbb{R}^{n}_+,\Lambda^k)$ we have  $\nu\iprod u_{|_{\partial\mathbb{R}^n_+}}\in \mathrm{B}^{s-\frac{1}{p}}_{p,q}(\mathbb{R}^{n-1},\Lambda^{k-1})$, such that formula \eqref{eq:IntegbyPartTangTrace} holds for all $\Psi\in \mathrm{B}^{1-s}_{p',q'}(\mathbb{R}^n_+, \Lambda)$. Moreover, we have estimates
    \begin{align*}
        \lVert\nu\iprod u_{|_{\partial\mathbb{R}^n_+}} \rVert_{\mathrm{B}^{s-\frac{1}{p}}_{p,q}(\mathbb{R}^{n-1})}\lesssim_{p,s,n} \lVert u \rVert_{{\mathrm{B}}^{s}_{p,q}(\mathbb{R}^n_+)} + \lVert \mathbf{\delta} u \rVert_{{\mathrm{B}}^{s}_{p,q}(\mathbb{R}^n_+)} \text{.}
    \end{align*}
    The same results holds with corresponding estimate, for $u\in {\mathrm{D}}^{s}_{p,q}(\mathrm{d},\mathbb{R}^{n}_+,\Lambda^k)$ we have a partial trace $\nu\wedge u_{|_{\partial\mathbb{R}^n_+}}\in \mathrm{B}^{s-\frac{1}{p}}_{p,q}(\mathbb{R}^{n-1},\Lambda^{k+1})$ such that \eqref{eq:IntegbyPartNormTrace} is satisfied.
    \item For all $u\in {\mathrm{B}}^{s+1}_{p,q}(\mathbb{R}^{n}_+,\Lambda^k)$, we have 
    \begin{align*}
        (\nu\iprod u\oplus\nu\wedge u )_{|_{\partial\mathbb{R}^n_+}}\in \mathrm{B}^{s+1-\frac{1}{p}}_{p,q} (\mathbb{R}^{n-1},\Lambda^{k-1}\oplus\Lambda^{k+1})
    \end{align*}
    with estimate
    \begin{align*}
        \lVert (\nu\iprod u\oplus\nu\wedge u  )_{|_{\partial\mathbb{R}^n_+}} \rVert_{\mathrm{B}^{s+1-\frac{1}{p}}_{p,q}(\mathbb{R}^{n-1})} \lesssim_{p,s,n} \lVert u \rVert_{{\mathrm{B}}^{s+1}_{p,q}(\mathbb{R}^n_+)} \text{,}
    \end{align*}
    and everything still hold with ${\mathrm{H}}^{s+1,p}$ instead of $\mathrm{B}^{s+1}_{p,q}$, when $q=p$.
\end{enumerate}
\end{theorem}

\begin{theorem}\label{thm:TracesDifferentialformsHomSpaces}Let $p\in(1,+\infty)$, $q\in[1,+\infty]$, $s\in(-1+\frac{1}{p},\frac{1}{p})$ and let $k\in\llb 0,n\rrb$.
\begin{enumerate}[label=($\roman*$)]
    \item For all $u\in{\dot{\mathrm{D}}}^{s}_{p}( \mathbf{\delta},\mathbb{R}^{n}_+,\Lambda^k)$, 
    
    \begin{itemize}
    \item If $s\leqslant 0$, then there exists a unique function $\nu\iprod u_{|_{\partial\mathbb{R}^n_+}}\in \mathrm{B}^{s-\frac{1}{p}}_{p,p}(\mathbb{R}^{n-1},\Lambda^{k-1})$ such that the formula \eqref{eq:IntegbyPartTangTrace} holds for all $\Psi\in \mathrm{H}^{1-s,p'}(\mathbb{R}^n_+, \Lambda)$, with estimate
    \begin{align*}
        \lVert\nu\iprod u_{|_{\partial\mathbb{R}^n_+}} \rVert_{\mathrm{B}^{s-\frac{1}{p}}_{p,p}(\mathbb{R}^{n-1})}\lesssim_{p,s,n} \lVert u \rVert_{{\dot{\mathrm{H}}}^{s,p}(\mathbb{R}^n_+)} + \lVert \mathbf{\delta} u \rVert_{{\dot{\mathrm{H}}}^{s,p}(\mathbb{R}^n_+)} \text{.}
    \end{align*}
    \item If $s> 0$, for $\frac{1}{r}=\frac{1}{p}-\frac{s}{n}\in(\frac{n-1}{pn},\frac{1}{p})$, there exists a unique function $\nu\iprod u_{|_{\partial\mathbb{R}^n_+}}\in \mathrm{B}^{-\frac{1}{r}}_{r,r}(\mathbb{R}^{n-1},\Lambda^{k-1})$, such that the formula \eqref{eq:IntegbyPartTangTrace} holds for all $\Psi\in \mathrm{H}^{1,r'}(\mathbb{R}^n_+, \Lambda^{k-1})$ with estimate
    \begin{align*}
        \lVert\nu\iprod u_{|_{\partial\mathbb{R}^n_+}} \rVert_{\mathrm{B}^{-\frac{1}{r}}_{r,r}(\mathbb{R}^{n-1})}\lesssim_{r,p,s,n} \lVert u \rVert_{{\dot{\mathrm{H}}}^{s,p}(\mathbb{R}^n_+)} + \lVert \mathbf{\delta} u \rVert_{{\dot{\mathrm{H}}}^{s,p}(\mathbb{R}^n_+)} \text{.}
    \end{align*}
    \end{itemize}
    The same result, up to appropriate changes, still holds for $u\in {\dot{\mathrm{D}}}^{s}_{p}(\mathrm{d},\mathbb{R}^{n}_+,\Lambda^k)$ with partial trace $\nu\wedge u_{|_{\partial\mathbb{R}^n_+}}$ satisfying the formula \eqref{eq:IntegbyPartNormTrace}.
    \item For all $u\in {\dot{\mathrm{D}}}^{s}_{p,q}(\delta,\mathbb{R}^{n}_+,\Lambda^k)$,
    \begin{itemize}
    \item If $s < 0$, there exists a unique $\nu\iprod u_{|_{\partial\mathbb{R}^n_+}}\in \mathrm{B}^{s-\frac{1}{p}}_{p,q}(\mathbb{R}^{n-1},\Lambda^{k-1})$ such that the formula formula \eqref{eq:IntegbyPartTangTrace} holds for all $\Psi\in [\eus{S}\cap\mathrm{B}^{1-s}_{p',q'}](\mathbb{R}^n_+, \Lambda^{k-1}\oplus\Lambda^{k+1})$, with estimates
    \begin{align*}
        \lVert\nu\iprod u_{|_{\partial\mathbb{R}^n_+}} \rVert_{\mathrm{B}^{s-\frac{1}{p}}_{p,q}(\mathbb{R}^{n-1})}\lesssim_{p,s,n} \lVert u \rVert_{{\dot{\mathrm{B}}}^{s}_{p,q}(\mathbb{R}^n_+)} + \lVert \mathbf{\delta} u \rVert_{{\dot{\mathrm{B}}}^{s}_{p,q}(\mathbb{R}^n_+)} \text{.}
    \end{align*}
    \item If $s> 0$, for $\frac{1}{r}=\frac{1}{p}-\frac{s}{n}\in(\frac{n-1}{pn},\frac{1}{p})$, there exists a unique $\nu\iprod u_{|_{\partial\mathbb{R}^n_+}}\in \mathrm{B}^{-\frac{1}{\tilde{r}}-\varepsilon}_{\tilde{r},q}(\mathbb{R}^{n-1},\Lambda^{k-1})$, for any sufficiently small $\varepsilon>0$, with $\frac{1}{r}-\frac{\varepsilon}{n}=\frac{1}{\tilde{r}}$, such that the formula \eqref{eq:IntegbyPartTangTrace} holds for all $\Psi\in [\eus{S}\cap\mathrm{B}^{1+\varepsilon}_{\tilde{r}',q'}](\mathbb{R}^n_+, \Lambda)$ with estimate
    \begin{align*}
        \lVert\nu\iprod u_{|_{\partial\mathbb{R}^n_+}} \rVert_{\mathrm{B}^{-\frac{1}{\tilde{r}}-\varepsilon}_{\tilde{r},q}(\mathbb{R}^{n-1})}\lesssim_{p,s,n,\varepsilon} \lVert u \rVert_{{\dot{\mathrm{B}}}^{s}_{p,q}(\mathbb{R}^n_+)} + \lVert \mathbf{\delta} u \rVert_{{\dot{\mathrm{B}}}^{s}_{p,q}(\mathbb{R}^n_+)} \text{.}
    \end{align*}
    \item If $s = 0$, there exists a unique $\nu\iprod u_{|_{\partial\mathbb{R}^n_+}}\in \mathrm{B}^{-\frac{1}{r}-\varepsilon}_{r,q}(\mathbb{R}^{n-1},\Lambda^{k-1})$, where  $\frac{1}{r}=\frac{1}{p}-\frac{\varepsilon}{n}$, for any sufficiently small $\varepsilon>0$, such that the formula \eqref{eq:IntegbyPartTangTrace} holds for all $\Psi\in [\eus{S}\cap\mathrm{B}^{1+\varepsilon}_{r',q'}](\mathbb{R}^n_+, \Lambda)$, with estimates
    \begin{align*}
        \lVert\nu\iprod u_{|_{\partial\mathbb{R}^n_+}} \rVert_{\mathrm{B}^{-\frac{1}{r}-\varepsilon}_{r,q}(\mathbb{R}^{n-1})}\lesssim_{p,s,n,\varepsilon} \lVert u \rVert_{{\dot{\mathrm{B}}}^{0}_{p,q}(\mathbb{R}^n_+)} + \lVert \mathbf{\delta} u \rVert_{{\dot{\mathrm{B}}}^{0}_{p,q}(\mathbb{R}^n_+)} \text{.}
    \end{align*}
    \end{itemize}
    The same result, up to appropriate changes, still holds for $u\in {\dot{\mathrm{D}}}^{s}_{p,q}(\mathrm{d},\mathbb{R}^{n}_+,\Lambda^k)$ with partial trace $\nu\wedge u_{|_{\partial\mathbb{R}^n_+}}$ satisfying the formula \eqref{eq:IntegbyPartNormTrace}.
    \item For all $u\in [{\dot{\mathrm{B}}}^{s}_{p,q}\cap{\dot{\mathrm{B}}}^{s+1}_{p,q}](\mathbb{R}^{n}_+,\Lambda^k)$, if $q\neq +\infty$, we have
    \begin{align*}
        (\nu\iprod u\oplus\nu\wedge u )_{|_{\partial\mathbb{R}^n_+}}\in \mathrm{B}^{s+1-\frac{1}{p}}_{p,q} (\mathbb{R}^{n-1},\Lambda^{k-1}\oplus\Lambda^{k+1})
    \end{align*}
    with estimate
    \begin{align*}
        \lVert (\nu\iprod u\oplus\nu\wedge u  )_{|_{\partial\mathbb{R}^n_+}} \rVert_{\mathrm{B}^{s+1-\frac{1}{p}}_{p,q}(\mathbb{R}^{n-1})} \lesssim_{p,s,n} \lVert u \rVert_{[{\dot{\mathrm{B}}}^{s}_{p,q}\cap{\dot{\mathrm{B}}}^{s+1}_{p,q}](\mathbb{R}^n_+)} \text{,}
    \end{align*}
    and everything still hold with $({\dot{\mathrm{H}}}^{s,p},{\dot{\mathrm{H}}}^{s+1,p})$ instead of $(\dot{\mathrm{B}}^{s}_{p,q},\dot{\mathrm{B}}^{s+1}_{p,q})$, when $q=p$.
    
    If $q=+\infty$, we have
    \begin{align*}
        (\nu\iprod u\oplus\nu\wedge u )_{|_{\partial\mathbb{R}^n_+}}\in \mathrm{L}^{p} (\mathbb{R}^{n-1},\Lambda^{k-1}\oplus\Lambda^{k+1})
    \end{align*}
    with corresponding estimate.
\end{enumerate}
\end{theorem}

\begin{remark}The proof of Theorem \ref{thm:TracesDifferentialformsInhomSpaces} in case of inhomogeneous function spaces follows straightforward the same proof provided for corresponding results in \cite[Section~4]{MitreaMitreaShaw2008} and is somewhat sharp. 

Notice that Theorem \ref{thm:TracesDifferentialformsHomSpaces} is certainly not sharp, and investigation of sharp range for partial traces could be of great interest in the treatment of inhomogeneous boundary value problems in homogeneous function spaces.
\end{remark}

\begin{proof}Without loss of generality we only investigate the case of tangential traces $\nu\iprod u_{|_{\partial\mathbb{R}^n_+}}$.

{\textbf{Step 1.1 :}} Proof of \textit{(i)}, for $s\leqslant 0$. For $u\in \dot{\mathrm{D}}^{s}_{p}( {\mathbf{\delta}},\mathbb{R}^{n}_+,\Lambda^k)$, for all $\psi\in\mathrm{B}^{\frac{1}{p}-s}_{p',p'}(\mathbb{R}^{n-1},\Lambda^{k-1})$, and $\Psi\in \mathrm{H}^{1-s,p'}(\mathbb{R}^n_+,\Lambda^{k-1})$ such that $\Psi_{|_{\partial\mathbb{R}^n_+}}=\psi$, we can define the following functional,
\begin{align*}
    \kappa_{u}(\Psi) :=  \int_{\mathbb{R}^{n}_+} \langle u, \mathrm{d} {\Psi}\rangle - \int_{\mathbb{R}^{n}_+} \langle \mathbf{\delta} u, {\Psi}\rangle\text{. }
\end{align*}
First, the map $(u,\Psi)\mapsto\kappa_{u}(\Psi)$ is well defined and bilinear on $\dot{\mathrm{D}}^{s}_{p}( \mathbf{\delta},\mathbb{R}^{n}_+,\Lambda^k)\times\mathrm{H}^{1-s,p'}(\mathbb{R}^n_+,\Lambda^{k-1})$, \textit{i.e.} in particular only depends on the boundary value $\psi$ of $\Psi$. It is straightforward from duality that,
\begin{align*}
    \lvert \kappa_{u}(\Psi) \rvert &\lesssim_{s,p,n}  \lVert u \rVert_{\dot{\mathrm{H}}^{s,p}(\mathbb{R}^n_+)}\lVert \mathrm{d}\Psi \rVert_{\dot{\mathrm{H}}^{-s,p'}(\mathbb{R}^n_+)} + \lVert\mathbf{\delta} u \rVert_{\dot{\mathrm{H}}^{s,p}(\mathbb{R}^n_+)}\lVert\Psi \rVert_{\dot{\mathrm{H}}^{-s,p'}(\mathbb{R}^n_+)}\\
    &\lesssim_{s,p,n}  \lVert u \rVert_{\dot{\mathrm{H}}^{s,p}(\mathbb{R}^n_+)}\lVert \Psi \rVert_{\dot{\mathrm{H}}^{1-s,p'}(\mathbb{R}^n_+)} + \lVert\mathbf{\delta} u \rVert_{\dot{\mathrm{H}}^{s,p}(\mathbb{R}^n_+)}\lVert\Psi \rVert_{\dot{\mathrm{H}}^{-s,p'}(\mathbb{R}^n_+)}\\
    &\lesssim_{s,p,n}  \lVert u \rVert_{\dot{\mathrm{H}}^{s,p}(\mathbb{R}^n_+)}\lVert \Psi \rVert_{{\mathrm{H}}^{1-s,p'}(\mathbb{R}^n_+)} + \lVert\mathbf{\delta} u \rVert_{\dot{\mathrm{H}}^{s,p}(\mathbb{R}^n_+)}\lVert\Psi \rVert_{{\mathrm{H}}^{-s,p'}(\mathbb{R}^n_+)}\\
    &\lesssim_{s,p,n} (\lVert u \rVert_{\dot{\mathrm{H}}^{s,p}(\mathbb{R}^n_+)}+\lVert \mathbf{\delta} u \rVert_{\dot{\mathrm{H}}^{s,p}(\mathbb{R}^n_+)}) \lVert \Psi \rVert_{{\mathrm{H}}^{1-s,p'}(\mathbb{R}^n_+)}\text{, }
\end{align*}
where above inequalities follows from ${\mathrm{H}}^{-s,p'}(\mathbb{R}^n_+)\hookrightarrow \dot{\mathrm{H}}^{-s,p'}(\mathbb{R}^n_+)$, ${\mathrm{H}}^{1-s,p'}(\mathbb{R}^n_+)\hookrightarrow \dot{\mathrm{H}}^{1-s,p'}(\mathbb{R}^n_+)$, since $-1+\frac{1}{p}<s\leqslant0$, then from ${\mathrm{H}}^{1-s,p'}(\mathbb{R}^n_+)\hookrightarrow {\mathrm{H}}^{-s,p'}(\mathbb{R}^n_+)$.

Now, if we have $\Psi_1,\Psi_2\in \mathrm{H}^{1-s,p'}(\mathbb{R}^n_+,\Lambda^{k-1})$ such that ${\Psi_1}_{|_{\partial\mathbb{R}^n_+}}={\Psi_2}_{|_{\partial\mathbb{R}^n_+}}=\psi$, we introduce $\Psi_0 = \Psi_1 - \Psi_2 \in \mathrm{H}^{1-s,p'}_0(\mathbb{R}^n_+,\Lambda^{k-1})$. Therefore, let's consider $(\Phi_{k})_{k\in\mathbb{N}}\subset \mathrm{C}_c^\infty(\mathbb{R}^n_+,\Lambda^{k-1})$ such that,
\begin{align*}
    \Phi_k \xrightarrow[k\rightarrow +\infty]{} \Psi_0\text{ in } \mathrm{H}^{1-s,p'}_0(\mathbb{R}^n_+,\Lambda^{k-1}) \text{. }
\end{align*}
We can deduce,
\begin{align*}
    \kappa_{u}(\Psi_1)- \kappa_{u}(\Psi_2) = \kappa_{u}(\Psi_0) = \lim_{k\rightarrow +\infty}  \left[ \int_{\mathbb{R}^{n}_+} \langle u, \mathrm{d} {\Phi_k}\rangle - \int_{\mathbb{R}^{n}_+} \langle \mathbf{\delta} u, {\Phi_k}\rangle \right] = 0\text{. }
\end{align*}
Thus, the following equality occurs, where we also consider the extension operator $\mathrm{Ext}_{\mathbb{R}^n_+}$ from \cite[Theorem~3.1]{Gaudin2022},
\begin{align*}
    \Tilde{\kappa}_{u}(\psi):={\kappa}_{u}(\mathrm{Ext}_{\mathbb{R}^n_+} \otimes\psi) = \kappa_{u}(\Psi),
\end{align*}
and with estimate, also obtained from \cite[Theorem~3.1]{Gaudin2022},
\begin{align*}
    \lvert\Tilde{\kappa}_{u}(\psi)\rvert \lesssim_{s,p,n} (\lVert u \rVert_{\dot{\mathrm{H}}^{s,p}(\mathbb{R}^n_+)}+\lVert\mathrm{d}u \rVert_{\dot{\mathrm{H}}^{s,p}(\mathbb{R}^n_+)}) \lVert \psi \rVert_{{\mathrm{B}}^{\frac{1}{p}-s}_{p',p'}(\mathbb{R}^{n-1})}\text{. }
\end{align*}
By duality, there exists a unique function depending linearly on $u$, $\nu\iprod u_{|_{\partial\mathbb{R}^n_+}}\in {\mathrm{B}}^{s-\frac{1}{p}}_{p,p}(\mathbb{R}^{n-1},\Lambda^{k-1})$, such that \eqref{eq:IntegbyPartTangTrace} holds. 

To guarantee that the representation formula makes sense, one may use the usual integration by parts formula with $u,\Psi \in\eus{S}(\overline{\mathbb{R}^n_+},\Lambda)$.

{\textbf{Step 1.2 :}} Proof of \textit{(i)}, for $s> 0$. For same assumption on $u$, as before $\Psi$, everything works similarly except the way we bounded bilinearly $(u,\Psi)\mapsto\kappa_{u}(\Psi)$ on $\dot{\mathrm{D}}^{s}_{p}( \mathbf{\delta},\mathbb{R}^{n}_+,\Lambda^k)\times\mathrm{H}^{1-s,p'}(\mathbb{R}^n_+,\Lambda^{k-1})$.
For $r\in(1,+\infty)$ such that  $\frac{1}{r}=\frac{1}{p}-\frac{s}{n}$, we deduce from Sobolev embeddings and duality that
\begin{align*}
    \lvert \kappa_{u}(\Psi) \rvert &\lesssim_{s,p,n}  \lVert u \rVert_{\dot{\mathrm{H}}^{s,p}(\mathbb{R}^n_+)}\lVert \mathrm{d}\Psi \rVert_{\dot{\mathrm{H}}^{-s,p'}(\mathbb{R}^n_+)} + \lVert\mathbf{\delta} u \rVert_{{\mathrm{L}}^{r}(\mathbb{R}^n_+)}\lVert\Psi \rVert_{{\mathrm{L}}^{r'}(\mathbb{R}^n_+)}\\
    &\lesssim_{r,s,p,n}  \lVert u \rVert_{\dot{\mathrm{H}}^{s,p}(\mathbb{R}^n_+)}\lVert \Psi \rVert_{{\mathrm{H}}^{1,r'}(\mathbb{R}^n_+)} + \lVert\mathbf{\delta} u \rVert_{\dot{\mathrm{H}}^{s,p}(\mathbb{R}^n_+)}\lVert\Psi \rVert_{{\mathrm{H}}^{1,r'}(\mathbb{R}^n_+)}\\
    &\lesssim_{r,s,p,n} (\lVert u \rVert_{\dot{\mathrm{H}}^{s,p}(\mathbb{R}^n_+)}+\lVert \mathbf{\delta} u \rVert_{\dot{\mathrm{H}}^{s,p}(\mathbb{R}^n_+)}) \lVert \Psi \rVert_{{\mathrm{H}}^{1,r'}(\mathbb{R}^n_+)}\text{. }
\end{align*}
Thus everything goes similarly.

{\textbf{Step 2.1 :}} Proof of \textit{(ii)}, for $s< 0$, is very similar to the one of above \textbf{Step 1.1}.

{\textbf{Step 2.2 :}} Proof of \textit{(ii)}, for $s> 0$, is somewhat similar to the one of \textbf{Step 1.2} but needs further explanations. We use Sobolev embeddings, and generalized H\"{o}lder inequalities using Lorentz spaces, $\dot{\mathrm{B}}^{s}_{p,q}(\mathbb{R}^n_+) \hookrightarrow {\mathrm{L}}^{r,q}(\mathbb{R}^n_+)$, ${\mathrm{B}}^{\varepsilon}_{\tilde{r}',q'}(\mathbb{R}^n_+) \hookrightarrow {\mathrm{L}}^{r',q'}(\mathbb{R}^n_+) \hookrightarrow\dot{\mathrm{B}}^{-s}_{p',q'}(\mathbb{R}^n_+)$, for $\varepsilon>0$,
\begin{align*}
    \lvert \kappa_{u}(\Psi) \rvert &\lesssim_{s,p,n}  \lVert u \rVert_{\dot{\mathrm{B}}^{s}_{p,q}(\mathbb{R}^n_+)}\lVert \mathrm{d}\Psi \rVert_{\dot{\mathrm{B}}^{-s}_{p',q'}(\mathbb{R}^n_+)} + \lVert\mathbf{\delta} u \rVert_{{\mathrm{L}}^{r,q}(\mathbb{R}^n_+)}\lVert\Psi \rVert_{{\mathrm{L}}^{r',q'}(\mathbb{R}^n_+)}\\
    &\lesssim_{r,s,p,n}  \lVert u \rVert_{\dot{\mathrm{B}}^{s}_{p,q}(\mathbb{R}^n_+)}\lVert \mathrm{d}\Psi \rVert_{{\mathrm{L}}^{r',q'}(\mathbb{R}^n_+)} + \lVert\mathbf{\delta} u \rVert_{\dot{\mathrm{B}}^{s}_{p,q}(\mathbb{R}^n_+)}\lVert\Psi \rVert_{{\mathrm{L}}^{r',q'}(\mathbb{R}^n_+)}\\
    &\lesssim_{r,s,p,n} (\lVert u \rVert_{\dot{\mathrm{B}}^{s}_{p,q}(\mathbb{R}^n_+)}+\lVert \mathbf{\delta} u \rVert_{\dot{\mathrm{B}}^{s}_{p,q}(\mathbb{R}^n_+)}) \lVert \Psi \rVert_{{\mathrm{B}}^{1+\varepsilon}_{\tilde{r}',q'}(\mathbb{R}^n_+)}\text{. }
\end{align*}

{\textbf{Step 2.3 :}} Proof of \textit{(ii)}, for $s= 0$ is shown via similar Sobolev embeddings arguments and is left to the reader.

{\textbf{Step 3 :}} Proof of \textit{(iii)}, follows from \cite[Proposition~3.8]{Gaudin2022}, with explicit representation formula for any suitable $k$-differential forms $u$:
\begin{align*}
     \nu\iprod u_{|_{\partial\mathbb{R}^n_+}}=-\mathfrak{e}_n\iprod u_{|_{\partial\mathbb{R}^n_+}} &= (-1)^k\sum_{1\leqslant \ell_1<\ldots<\ell_{k-1}< n} u_{\ell_1 \ell_2\ldots \ell_{k-1} n}(\cdot,0)\, \mathrm{d}x_{\ell_1}\wedge\ldots\wedge\mathrm{d}x_{\ell_{k-1}}\\
     &= (-1)^k\sum_{I'\in\mathcal{I}^{k-1}_{n-1}} u_{I', n}(\cdot,0)\,\mathrm{d}x_{I'}\text{. }
\end{align*}
A similar treatment yields the same conclusion for the boundary term $\nu\wedge u_{|_{\partial\mathbb{R}^n_+}}$, so that one ends the proof here.
\end{proof}

\begin{remark}Let's make further comment about estimates used in above proof of Theorem \ref{thm:TracesDifferentialformsHomSpaces}, in particular the ones used in $\text{{\textbf{Step 2.2}}}$.

We recall that from Sobolev embeddings, see \cite[Proposition~2.31]{Gaudin2022}, for $0<s_0<s<s_1<1/p$, $r_0,r_1,p\in(1,+\infty)$, $1/r_j = 1/p-s_j/n$, and appropriate $q_j\in[1,+\infty]$ we have
    \begin{align*}
       \dot{\mathrm{B}}^{s_j}_{p,q_j}(\mathbb{R}^n_+) \hookrightarrow \mathrm{L}^{r_j}(\mathbb{R}^n_+)    \,\text{, }\, j\in\{ 0, 1\} \text{. }
    \end{align*}
    If $(s,1/r)=(1-\theta)(s_0,1/r_0)+ \theta (s_1,1/r_1)$, by real interpolation, for $q\in[1,+\infty]$ we obtain,
    \begin{align*}
       \dot{\mathrm{B}}^{s}_{p,q}(\mathbb{R}^n_+) = (\dot{\mathrm{B}}^{s_0}_{p,q_0}(\mathbb{R}^n_+),\dot{\mathrm{B}}^{s_1}_{p,q_1}(\mathbb{R}^n_+))_{\theta,q} \hookrightarrow({\mathrm{L}}^{r_0}(\mathbb{R}^n_+),{\mathrm{L}}^{r_1}(\mathbb{R}^n_+))_{\theta,q} = {\mathrm{L}}^{r,q}(\mathbb{R}^n_+) \text{. }
    \end{align*}
    And one may proceed similarly for the reverse embedding,
    \begin{align*}
        {\mathrm{L}}^{r',q'}(\mathbb{R}^n_+) \hookrightarrow\dot{\mathrm{B}}^{-s}_{p',q'}(\mathbb{R}^n_+)\text{. }
    \end{align*}
    
For more details about Lorentz spaces and their interpolation, one could consult \cite[Section~1,~Examples~1.10,~1.11~\&~1.27]{bookLunardiInterpTheory} and \cite[Chapter~5,~Section~5.3]{BerghLofstrom1976}.
\end{remark}

\typeout{}                                
\bibliographystyle{alpha}
{\footnotesize
\bibliography{Biblio}}

\end{document}